\documentclass[a4paper, reqno, 11pt]{amsart}

\usepackage[english]{babel}
\usepackage{amsmath, amssymb, amsthm, amsfonts, mathtools} 
\usepackage{graphicx, caption, xcolor, placeins} 
\usepackage{enumerate}
\usepackage{ifthen}
\usepackage{bbm}
\usepackage{geometry}
\provideboolean{shownotes} 
\setboolean{shownotes}{true}

\usepackage{color}

\usepackage{tikz} %% 
\usetikzlibrary{arrows.meta} %% 
\tikzset{font=\small, 
point/.style={fill, circle, inner sep=1.2pt}, % definition of points 
>={Straight Barb[round,angle=60:1.2mm 1]} % define arrows
} 

\usepackage{hyperref}
\usepackage{setspace}
\usepackage{soul}

\usepackage{xpatch}
\makeatletter
\patchcmd{\@tocline}{\hfil}
{\nobreak\leaders\hbox{\ifnum#1<2\hfill\else$\m@th%
\mkern 4.5 mu\hbox{.}\mkern 4.5 mu$\fi}\hfill\nobreak}{}{}
\def\l@section{\@tocline{1}{10pt}{1pc}{}{\bfseries}}
\def\l@subsection{\@tocline{2}{0pt}{\dimexpr 1pc+2em}{}{}}
\makeatother
\newcommand{\hole}[1]{
\ifthenelse{\boolean{shownotes}}%
{\begin{center} \fbox{ \rule {.25cm}{0cm}
\rule[-.1cm]{0cm}{.4cm} \parbox{.85\textwidth}{\begin{center}
\texttt{#1}\end{center}} \rule {.25cm}{0cm}}\end{center}}
{}
}

\newtheorem{theorem}{Theorem}[section]
\newtheorem{proposition}[theorem]{Proposition}
\newtheorem{lemma}[theorem]{Lemma}

\newtheorem{definition}[theorem]{Definition}

\allowdisplaybreaks
\theoremstyle{remark}
\newtheorem{remark}[theorem]{Remark}

\newcommand{\R}{\mathbb{R}}
\newcommand{\T}{\mathbb{T}^d}

\newcommand{\N}{\mathbb{N}}
\newcommand{\Z}{\mathbb{Z}}
\newcommand{\cR}{\mathcal{R}}
\newcommand{\cS}{\mathcal{S}}

\newcommand{\dive}{\mathop{\mathrm {div}}}
\renewcommand{\div}{\mathrm {div}}
\newcommand{\tr}{\mathrm {tr}}
\newcommand{\curl}{\mathop{\mathrm {curl}}}

\newcommand{\sgn}{\mathrm {sgn}}

\newcommand{\charf}{{{\text{\rm 1}}\kern-.24em {\text{\rm l}}}}

\newcommand{\del}{\partial}
\newcommand{\eps}{\varepsilon}
\newcommand{\Tone}{\mathbb{T}^1}
\newcommand{\cE}{\mathcal{E}}
\newcommand{\supp}{{\rm supp\,}}
\newcommand{\disxtx}{{\mathcal{D}}^\prime_{t,x,\xi}}
\newcommand{\disxt}{{\mathcal{D}}^\prime_{t,x}}

\newcommand{\dis}{{\mathcal{D}}^\prime}

\newcommand{\RR}{\mathbb{R}}
\newcommand{\Rd}{\mathbb{R}^d}

\newcommand{\tcb}{\textcolor{blue}}

\numberwithin{equation}{section}

\begin{document}

\title[Oscillations and Homogenization]{
Derivation of effective kinetic equations
\\
describing oscillations in viscoelasticity 
\\
and in compressible Navier-Stokes
}

\author[A.E. Tzavaras]{Athanasios E. Tzavaras}
\address[Athanasios E. Tzavaras]{
\newline
Computer, Electrical and Mathematical Science and Engineering Division 
\newline
King Abdullah University of Science and Technology (KAUST)
\newline 
Thuwal 23955-6900,  Saudi Arabia
}
\email{athanasios.tzavaras@kaust.edu.sa}

\baselineskip=18pt

\begin{abstract}
These lecture notes are devoted to solutions of hyperbolic-parabolic systems with persistent oscillations. 
We consider two examples both from mechanics: (i)  The system of viscoelasticity of Kelvin-Voigt type with strain energies 
involving double well potentials, as employed in phase transitions. (ii)  The compressible Navier-Stokes equations for a barotropic 
gas. For each system we construct solutions with persistent oscillations. In a later part we consider the nonlinear homogenization
problem. For the systems of viscoelasticity in one-space dimension in Lagrangian coordinates, and for 
the compressible Navier-Stokes system  for barotropic fluids we  show how
ideas from the kinetic formulation of conservation laws can be used to derive effective equations. The effective equation consists
by a kinetic equation coupled with the macroscopic flow.
\end{abstract}

\maketitle

\tableofcontents

\section{Introduction}

Solutions of Partial Differential Equations with oscillatory structures appear in several problems including
homogenization, turbulence and phase transitions. The origin and provenance of oscillations might well differ in distinct contexts
and is not at present well understood from a mathematical perspective even in the most prominent application of
turbulence. In this review we summarize certain mathematical results on the effects of regularizing mechanisms 
on the persistence of oscillations and the related problem of homogenization.  To a large extent our perspective is
focussed on problems of phase transitions, be that in solids or in liquids, and manifested through the lack of
convexity in the strain energy or the internal energy function.

The effectiveness of dissipation on smoothing of discontinuities was proposed by Dafermos \cite{Dafermos81} as 
a way of classifying hyperbolic-parabolic systems of thermomechanics and was developed in a set
of unpublished Lecture Notes \cite{Dafermos85}. The situation ranges from the full absence of dissipative
mechanisms leading to formation of shocks, to full parabolic systems that instantaneously regularize solutions,
with hyperbolic-parabolic systems in the middle with some fields regularizing while others develop singularities.
The Green function for the linearized operator around a constant state was 
demonstrated,   in the penetrating studies of Liu and Zeng \cite{LZ97,Liu97}, to be
an object capturing the efficacy of dissipation around shocks.

The situation is less understood regarding propagation of oscillations.
It is well-known that linear hyperbolic systems propagate 
oscillations from the initial-data to solutions, and  propagation and/or cancellations of oscillations results as
a consequence of nonlinear response captured by entropy inequalities. 
On the other extreme, fully-parabolic systems instantaneously smoothen initial oscillations.
Less was known for the intermediate case of hyperbolic-parabolic systems.
Examples were recently produced in \cite{Tzavaras23} and will be extensively reviewed here both for the linear case
but also nonlinear cases of  hyperbolic-elliptic systems with partial dissipation.

The program will be developed using as paradigms two canonical systems.
The Cauchy problem for nonlinear viscoelastic materials of Kelvin-Voigt type
\begin{align}\label{eq:wave}
&\partial_{tt}y-\dive  \Big (  \frac{\del W}{\del F} (\nabla y) \Big ) = \Delta \partial_{t}y \, , 
\\
\nonumber
&y|_{t=0}=y_0,\quad  \partial_{t}y|_{t=0}=v_0 \, .
\end{align}
Introducing the velocity $v = \del_t y$ and the deformation gradient $F = \nabla y$,  the second order evolution \eqref{eq:wave} 
is expressed as a hyperbolic-parabolic system
\begin{equation}\label{eq:main}
\begin{aligned}
\partial_{t}v-\dive(S(F)) &= \Delta\,v \\
\partial_{t}F-\nabla\,v&=0\\
\curl\,F&=0\, .
\end{aligned}
\end{equation}

The second system of interest is the compressible Navier-Stokes system for a barotropic fluid in Eulerian coordinates,
\begin{equation}
\label{eq:compNS}
\begin{aligned}
\del_t \rho + \dive \rho u &= 0
\\
\del_t \rho u + \dive \rho u \otimes u  + \nabla p (\rho) &=  \dive \Big ( \mu (\nabla u + \nabla u^T) + \lambda (\dive u )  \, I \, \Big ) \, ,
\end{aligned}
\end{equation}
where $\rho$ and $u$ are the density and velocity of the fluid. The constants 
$\mu$ is the shear viscosity,  $\lambda$ is called second viscosity coefficient 
and is connected to the bulk viscosity $\zeta$ through $\zeta = \lambda + \tfrac{2}{3} \mu$. Sometimes 
$\lambda$, $\mu$ are called Lam\'e coefficients and may depend on the density (and on temperature when the
latter is taken into account). Here, it will be assumed that they are constants that satisfy $\mu > 0$, $\lambda + \tfrac{2}{3} \mu > 0$. 
This system has analogies to  viscoelasticity, both conceptually and in terms of mathematical properties.
In the one-dimensional case, the system of equations
\begin{equation}
\label{intro-vhcg}
\begin{aligned}
w_t - v_x &= 0 \, ,
\\
v_t - \sigma(w)_x &= ( \frac{\mu}{w} v_x )_x  \, ,
\end{aligned}
\end{equation}
describes both the equations for longitudinal motions of a (one-dimensional) viscoelastic bar
and the equations of a compressible viscous gas in Lagrangean coordinates.

The Lecture Notes are organized as follows. In Section \ref{sec:motivation} we survey the existence
theory of nonlinear viscoelasticity of Kelvin-Voigt type. Subsequently, we outline remarks 
on the existence theory of compressible Navier-Stokes
indicating why the existence of oscillatory solutions may be anticipated. 
In fact, the persistence of oscillations in compressible Navier-Stokes was conjectured by P.L. Lions \cite{b-Lions98}
in connection to the development of the existence theory of global weak solutions
for this system.

The general problem is placed within the context of
 oscillations in quasilinear hyperbolic-parabolic systems.
Starting in Section \ref{sec:linosc} oscillatory solutions for linear systems are constructed.
Then in Section \ref{sec:phasetransitions} we outline various examples of oscillatory solutions
for quasilinear viscoelastic systems with nonconvex strain energies, which are employed in
modeling phase transitions. Solutions with persistent oscillations have been observed
in phenomena like twinning of crystals and have been studied in a context of minimization problems
 by Ball and James \cite{James81}, \cite{BJ87}. We show that they are also a special class 
 of the dynamic oscillatory solutions pursued here.

In Section \ref{sec:comprns} we focus on the compressible Navier-Stokes system with
non-monotone pressure of the type appearing in the Van-der-Waals equation of state.
The function
\begin{equation}\label{intro-ps}
\begin{aligned}
\rho (t,y) &= 
\begin{cases} 
\; \frac{a}{t^d}  \quad & \qquad \quad k t < |y| < (k + \theta) t \\[5pt]
\; \frac{b}{t^d}  \quad  & \; (k + \theta) t < |y| < (k + 1) t \\
\end{cases}
\; , \qquad k \in \mathbb{N}_0 = \{ 0, 1, 2, ... \} \, ,
\\[5pt]
u(t, y) &= \frac{y}{t} \, .
\end{aligned}
\end{equation}
is shown to form an example of solution with persistent oscillations. It exhibits a smooth velocity field 
with jump discontinuities of the density occurring across characteristics.

The presence  of persistent oscillations raises the problem of deriving effective equations for the description 
of propagating oscillations. This problem goes under the name homogenization and in the context of
compressible Navier-Stokes has received  attention by Serre \cite{Serre91}, Hillairet \cite{Hillairet07}.
The homogenization problem is there addressed using Young measures for renormalized solutions. 
Young measures are not amenable to analysis of dynamics and this has presented an obstacle 
to the further development of the homogenization theory.  An interesting connection was developed in  \cite{Hillairet07}  
between a specific natural ansatz for the Young measure of oscillations in density
 to problems of multi-phase flows. 
 
 An alternate tool to study the dynamics was 
 proposed in \cite{Tzavaras24}. It employs ideas from the kinetic formulation for scalar
 conservation laws \cite{LPT94, b-Perthame02} introducing this approach to the 
 derivation of effective equations for hyperbolic-parabolic systems. In Section \ref{sec:kineticcl} we review
  how homogenization problems are addressed using the kinetic formulation in the context
 of zero-viscosity limits for scalar conservation laws
 \begin{equation}\label{intro-sccl}
 \del_t u^n + \del_x f(u^n) = \eps_n \del^2_x u^n \, .
\end{equation}
We explain the relation between Young measures and kinetic solutions and how the latter give rise
to a transport problem and the notion of generalized kinetic solutions. In Section \ref{sec:genkinproof}, 
we present an alternative proof  of Tartar's Theorem \cite{Tartar79} on cancellation of oscillations for scalar one-dimensional conservation laws.
The proof, due to \cite{PT00}, utilizes the correspondence  (for scalar oscillating functions)
between Young measures and their distribution functions, and used the distribution function to derive information on
the support of the Young measure.

Then we focus on the main paradigm, the problem of deriving an effective equation describing 
propagating oscillations (induced by oscillatory data) for solutions $(v^\eps, u^\eps)$ 
to the system
\begin{equation}
\label{eq:onedvisco}
\begin{aligned}
u_t &= v_x
\\
v_t &= \sigma(u)_x + v_{xx}\, ,
\end{aligned}
\end{equation}
when $\sigma(u)$ is possibly nonmonotone. It is known that for non-monotone stresses oscillations may develop and persist.
As in that case (due to the viscosity) oscillations are only expected in the scalar field of the strain, this problem is a good
paradigm to test the use of ideas from the kinetic formulation for that problem.
A key point is that the weak limits of a uniformly bounded family of functions
$\{ u^\eps \}$ may be represented by the Young measures, but also by the weak-limit of the kinetic function
$$
\charf_{u^\eps < \xi} \rightharpoonup F  \quad \mbox{wk-$\ast$ in $L^\infty$}.
$$
The function $F$ can also be understood as the distribution function of the Young measure $\nu_{t,x} (\xi)$, $\xi \in \R$. We prove in Theorem \ref{mainthm}
that oscillations in solutions $(u^\eps, v^\eps)$ of \eqref{eq:onedvisco} are described by the effective system
\begin{equation}\label{intro-effs}
\begin{aligned}
\del_t F &+ \del_\xi \Big ( \big ( v_x + \overline{\sigma(u)} - \sigma(\xi) \big ) F \Big ) + \sigma^\prime (\xi) F = 0   \, ,
\\[5pt]
\del_t v &= \del_{xx} v + \del_x  \left (\overline{\sigma(u)} \right )    \, ,
\\[5pt]
S &=  \overline{\sigma(u)}  + v_x   = \int \sigma(\xi) dF_{t,x} (\xi)  + v_x  \, .   \\[5pt]
\end{aligned}
\end{equation}
The propagation of $F$ is described by a kinetic equation depending on moments as is typical in the kinetic formulation of systems
of conservation laws, \cite{LPT94, PT00}.  We indicate by means of examples that the kinetic equation \eqref{intro-effs}$_1$ 
has the anticipated response of separating phases when data are posed crossing the unstable region. This leads
to anticipate that the model \eqref{intro-effs} will have efficacy in describing the homogenization of system \eqref{eq:onedvisco}.

In Section \ref{sec:homcomNS}, we take up the problem of deriving effective equations for oscillating solutions of
 the compressible Navier-Stokes system with non-monotone pressure. We first describe the  result
of Hillairet \cite{Hillairet07}  on homogenization  for \eqref{eq:compNS} using Young measures of renormalized solutions.
Because the density is the only variable that may exhibit propagating oscillations, 
 the Young measure can be substituted by its distribution function. 
This observation is motivated by the use of kinetic functions to describe oscillations in scalar conservation laws, 
its use for \eqref{eq:compNS} is suggested in \cite{b-PS12}.
We outline two methodological approaches to derive
kinetic equations that describe the evolution of approximate solutions and their effective limits, 
see Sections \ref{sec:effmethod1}, \ref{sec:effmethod2}.  We refer to \eqref{approxeqn} 
in Section \ref{sec:effmethod1}
for the form of an approximate kinetic equation and to \eqref{homoS2} in Section \ref{sec:effmethod2} for a rendition of the homogenized system.

%\vfil\eject

\section{Existence theory of exemplar hyperbolic-parabolic systems}\label{sec:motivation}

We review elements of the existence theory first for \eqref{eq:wave} and then for  \eqref{eq:compNS}
both viewed as important and representative examples of hyperbolic-parabolic systems with singular diffusion matrices. 
Various reasons lead to conjecture that oscillations in the data partly propagate to solutions of these systems.

\subsection{The system of viscoelasticity of Kelvin-Voigt type}

The system describing the equations of nonlinear viscoelasticity in Lagrangean coordinates is
\begin{equation}\label{verate}
 \frac{\del^2 y }{\del t^2} = \div  \frac{\del W}{\del F} (F) +  \mu \Delta \frac{\del y}{\del t} \, .
\end{equation}
Here  $y(x,t)$ is a function describing the motion of a reference configuration, 
$v = \frac{\del y}{\del t}$ is the velocity and $F = \nabla y$ the deformation gradient;  note that 
$\dot F = \nabla v$.
The system is obtained by combining the balance of momentum equation $\frac{\del^2 y }{\del t^2} = \div T$,
where $T$ describes the stress tensor, with the constitutive assumption of viscoelasticity of the rate type
\begin{equation}\label{veKV}
T  = \frac{\del W}{\del F} (F) +  \mu \dot F
\end{equation}
where $\mu > 0$ is the viscosity coefficient.
This leads to the equation \eqref{verate} of viscoelasticity of Kelvin-Voigt type.
The constitutive hypothesis \eqref{veKV} is not frame indifferent 
- which is considered a  deficiency from a perspective of theoretical mechanics - 
yet offers a good starting point to understand some basic mathematical features.

The limiting case $\mu = 0$, when rate-dependence is ignored, leads to $T = \frac{\del W}{\del F} (F) $ and is
called hyperelasticity. The elastic stresses are induced by a strain-energy function $W(F)$. The
system of hyperelasticity can be expressed as a first-order system of nonlinear partial differential equations,
written in coordinates as
\begin{align}
\label{esystem}
&\begin{aligned}
&\del_t F_{i \alpha} = \del_\alpha v_i
\\
&\del_t v_i = \del_\alpha \tfrac{\del W}{\del F_{i \alpha}} (F)
\end{aligned}
\\[5pt]
&\del_\alpha F_{i \beta} - \del_\beta F_{i \alpha} = 0 \, , 
\label{econstraint}
\end{align}
where the summation convention over repeated indices is adopted.
The equation \eqref{econstraint} is concisely expressed as $\curl F = 0$ and is a constraint expressing the requirement that $F$ be a gradient
 $F = \nabla y$. The constraint is an involution along the evolution,  that is it is propagated from the initial data to the solution 
 via the kinematic relation \eqref{esystem}$_1$.
 
Analysis of hyperbolicity for the elasticity system \eqref{esystem} in dimension $d=3$ indicates that hyperbolicity is equivalent to rank-1 convexity of the 
strain energy function:
\begin{align*}
\mbox{Hyperbolicity of \eqref{esystem} } \quad &\Longleftrightarrow \quad W(F) \quad \mbox{is rank-1 convex}
\\
&\Longleftrightarrow
\frac{\del^2 W}{\del F_{i \alpha} \del F_{j \beta}} (F) \; \xi_i \xi_j \nu_\alpha \nu_\beta > 0
\quad \forall \xi \ne 0 \, , \, \nu \in \cS^2
\end{align*}
There are twelve characteristic speeds of \eqref{esystem}:  the first six are 
$\lambda_1 = \dots  = \lambda_6 = 0$ while the last six are the positive and negative square roots of the
eigenvalues of the acoustic tensor
$$
Q_{i j} = \frac{\del^2 W}{\del F_{i \alpha} \del F_{j \beta}} \nu_\alpha \nu_\beta 
$$
The eigenvalues are real provided the stored energy function is rank-one convex.

\subsubsection{Existence Theory}
Henceforth we set the viscosity $\mu =1$ and consider \eqref{verate} expressed as a system
\begin{equation}\label{vesystem}
\begin{aligned}
\del_t v &= \div \Big (  \frac{\del W}{\del F} (F) \Big ) +   \Delta v
\\
\del_t F &= \nabla v
\\
&\curl F = 0
\end{aligned}
\end{equation}
The existence theory of \eqref{vesystem} utilizes the (formal) energy identity
\begin{equation}\label{veenergy}
\begin{aligned}
    \del_t \left ( \tfrac{1}{2} |v|^2 + W(F)  \right ) -  \div \left ( v \cdot  \Big ( \frac{\del W}{\del F} (F) +  \nabla v \Big ) \right )  +  |\nabla v|^2  = 0
\end{aligned}
\end{equation}

%\tcb{Viscoelasticity of rate-type $d=1$}  - shear motions  $u = \frac{\del y}{\del x} $ shear strain
%\begin{equation*}
%\begin{aligned}
%u_t &= v_x
%\\
%v_t &=  \del_x  \sigma(u)  +      \del_x  \left (  v_{x} \right )
%\end{aligned}
%\end{equation*}
%

We place hypotheses on the stored energy function $W(F)$,  designed to
allow for stored energies that exhibit double well potentials, associated with phase transitions. We assume that $W(F)$ is
a $C^2$-function that satisfies
\begin{align}
c |F|^p - \tilde c  \le W(F) &\le C (1 + |F|^p)  \qquad \mbox{for $ p \ge 2$}
\nonumber
\\[5pt]
|DW(F)|  &\le C ( 1 + |F|^{p-1} )
\label{hypsemiconvex}
\tag {H}
\\[5pt]
W(F)  \mbox{ is  semiconvex }   &\mbox{, i.e.  $W(F) + \frac{K}{2} |F |^2$ is convex for some $K>0$} \, .
\nonumber
\end{align}
The hypothesis $p \ge 2$ is for technical convenience, but plays an important role in technical aspects of the analysis.
More important, the smoothness hypothesis in \eqref{hypsemiconvex} notably excludes potentials that blow up as $\det F \to 0$;
the inclusion of such potentials is a minimum requirement to avoid interpenetration of matter. How to include such potentials 
and to avoid interpenetration of matter is an outstanding open problem for the dynamical theory.

Several existence theorems for global solvability  for the system \eqref{vesystem} of viscoelasticity of Kelvin-Voigt type and its variants are available
in various settings and boundary conditions, 
{\it e.g.} \cite{FD97,Rybka92,Demoulini00,KLST23} and references therein. 
In the following theorem, avoiding  technical details, we summarize 
results for  the initial value problem on the torus $\mathbb{T}^d$.

\begin{theorem}\label{veweak} Let $d > 1$.
 \begin{itemize}
 \item[(i)] \cite{Rybka92,FD97} If the initial data $v_0 \in L^2 \, , F_0 \in L^p$, $p \ge 2$, there exists a global weak solution
 $(v,F)$ of \eqref{vesystem} of class
 $$
 v \in L^\infty (L^2) \cap L^2(H^1) \, , \quad F \in L^\infty (L^p) \, ,
 $$
 which satisfies the energy inequality
\begin{equation}
\label{energybb}
\int  \Big (  \tfrac{1}{2} |v|^2 + W(F) \Big ) dx  + \int_0^t \int |\nabla v|^2  \le \int \tfrac{1}{2} |v_0|^2 + W(F_0) \, .
\end{equation}

 \item[(ii)] \cite{KLST23} If  $v_0 \in L^2 \, , F_0 \in L^p \cap H^1$, $p \ge 2$, then in addition $F \in L^\infty (H^1)$. The solution is here
 (almost) what is called a strong solution. In addition, the solution is unique for dimension $d=2$
and satisfies energy conservation
\begin{equation}
\label{energyiden}
\int  \tfrac{1}{2} |v|^2 + W(F) dx  + \int_0^t \int |\nabla v|^2  = \int \tfrac{1}{2} |v_0|^2 + W(F_0)
\end{equation}
under the restrictions
$$
\begin{cases}
2 \le p & d=2
\\
2 \le p \le 4 & d=3 \, .
\end{cases}
$$

\end{itemize}

\end{theorem}

The approach to the proof of  Friesecke-Dolzmann \cite{FD97} uses the natural energy bounds obtained from \eqref{veenergy} and the hypotheses
\eqref{hypsemiconvex} on $W$ but with one important caveat outlined below. They construct solutions by a time-stepping algorithm, iteratively solving a minimization problem
\begin{itemize}
\item[]
Given $(v^0 , F^0)$ find the minimizer of 
\begin{equation*}
\min \int \tfrac{1}{2} |v-v^0|^2 + \frac{1}{2h} |F - F^0|^2 + W(F) dx
\end{equation*}
subject to the affine constraint
\begin{equation*}
\frac{F-F^0}{h} = \nabla v 
\end{equation*}
\end{itemize}
This problem is convex for $h$ sufficiently small due to the semiconvexity hypothesis.
They then build approximate solutions $y^h (t,x)$, $F^h (t,x)$, $v^h (t,x)$ via piecewise affine 
and piecewise linear interpolations, and in a last important step they show a propagation of 
compactness result
$$
\begin{aligned}
&\limsup_{h \to 0} \int_0^t\int \big |F^h (t,x) - F(t,x) \big |^2 dx dt 
 \le \left ( \limsup_{h\to 0} \int |F^h_0 - F_0|^2 dx \right ) e^{Kt}
\end{aligned}
$$
that allows to prove compactness, see \cite{FD97} and \cite{Demoulini00} that develops
similar ideas in a more complex setup.

The approach of \cite{KLST23} is based on a transfer of dissipation identity that we outline at the level
of the system with variable viscosity $\mu$:
\begin{equation*}
\begin{aligned}
\del_t v -  \div S(F) - \mu  \Delta v &=0   \qquad \eps > 0
\\
\del_t F - \nabla v &= 0
\end{aligned}
\end{equation*}
One then has the energy identity 
\begin{equation}
\label{energy-1}
\frac{d}{dt}\int  \tfrac{1}{2} |v|^2 + W(F) dx  + \mu \int |\nabla v|^2  = 0
\end{equation}

It is possible to obtain an additional estimate for periodic solutions. This is effected
by combining the compensated compactness type of bracket
\begin{equation}
\begin{aligned}
\del_t \big ( v \cdot \div F - \mu |\div F|^2 \big ) - \div ( v \cdot \nabla v )
= \div S(F) \cdot \div F  -  \nabla v : \nabla v
\end{aligned}
\end{equation}
with a scaled version of  the energy identity \eqref{energy-1} 
leading to the identity
\begin{equation}\label{transdiss}
\frac{d}{dt}\int  \tfrac{1}{2} |v - \tfrac{\mu}{2}\div F|^2  + \tfrac{\mu^2}{4} | \div F|^2 + W(F)  dx  
+ \tcb{ \tfrac{\mu}{2} } \int D^2 W : ( \nabla F, \nabla F) + |\nabla v|^2  = 0 \, .
\end{equation}
The latter manifests transfer of dissipation from momentum to the strain under improved initial
regularity of the strain, see \cite{KLST23}. Clearly for $W(F)$ convex equation \eqref{transdiss} provides
a gradient estimate for the deformation gradient.

This estimate can also be achieved for semiconvex energies \eqref{hypsemiconvex} by a small adaptation. Setting
$$
\tilde W (F) = W(F) + \frac{K}{2} |F|^2
$$
we note $\tilde W (F)$ is convex and use $\mu = 2$ to conclude
\begin{equation}
\label{unicab}
\frac{d}{dt} \int   |  v -  \dive F|^2 +  |\nabla F |^2+  2 W(F) \, dx 
+  \int  \Big (  D^2 \tilde W : (\nabla F, \nabla F) +  |\nabla v|^2 \Big ) dx \le K \int |\nabla F|^2 dx.
\end{equation}
Using Gr\"onwall's lemma, \eqref{unicab} yields control of $\int |\nabla F|^2 dx$.

Summarizing, there is an available existence theory for initial data $F_0 \in L^p$ and a second theory
for data $F_0 \in L^p \cap H^1$. The former is based on the property that strong convergence of the
initial data $F_{0n} \to F_0 \; \mbox{ in $L^p$}$ implies strong convergence of the corresponding
solutions $F_n \to F \; \mbox{in $L^p_{t,x}$}$.
The latter assumes $F_0 \in H^1$ and obtains the same regularity $F \in H^1$ for the solution.
This raises the question: what happens if $F_{0n} \rightharpoonup F_0$ weakly in $L^p$. This 
question is undertaken in later sections.

\subsubsection{Relation between existence and strong convergence of the initial data}
We show next an argument that illustrates the role of compactness of initial data. Let $y_n$ be a family of periodic, approximate solutions generated
by the equations
\begin{equation}
\label{appreqn}
y^n_{t t} = div_x \left ( \frac{\del W}{\del F} (\nabla y^n) + \nabla y^n_t \right ) + f_n  \, ,  \qquad (t,x) \in (0,T) \times \T \, ,
\end{equation}
where $f_n \to f$ in $L^2_{t, x}$,  and suppose the initial data are uniformly bounded in energy,
$$
\int  \tfrac{1}{2} |v_0^n|^2 + W(F_0^n) dx  \le C \, .
$$
Assume that $W(F)$ is strictly convex and that $p > 2$.
The sequence  $(v^n , F^n)$ of approximate solutions are uniformly bounded in energy by \eqref{energybb} which together
with the Aubin-Lions lemma and the representation theorem via Young measures \cite{Ball89} implies that
along a subsequence $(v_n, F_n)$ satisfies
$$
\begin{aligned}
v^n \to v  \; \mbox{ in $L^2$}   \qquad &F^n \rightharpoonup F \;  \; \mbox{wk-$\ast$ in $L^\infty(L^p)$}
\\
g(F^n)  \rightharpoonup  \int g(\lambda_F) \,  d\nu_{t,x} (\lambda_F)  =:  \overline{g(\lambda_F)}  
\quad &\mbox{ for $g$ continuous s.t. $\lim_{|\xi|\to\infty} \frac{g(\xi)}{ 1 + |\xi|^p} = 0$}
\end{aligned}
$$
Here, we used the notation $\nu_{t,x}$ for the
parametrized family of probability measures (called Young measures) that represent the weak convergence
properties of (a subsequence of) $\{ F^n\}$.

We multiply the equation by $y^n$, integrate on the torus $\T$ and by parts to obtain
$$
\int_{\T} (v^n)^2 - \int_{\T} \left ( \frac{\del W}{\del F} (\nabla y^n) + \nabla y^n_t \right ) \cdot \nabla y^n + \int_{\T} f^n y^n = 
\del_t \int_{\T} y^n v^n
$$
which gives
\begin{equation}\label{form101}
\begin{aligned}
&\int_0^t \int_{\T} (v^n)^2  - \int_0^t \int_{\T}  \frac{\del W}{\del F} ( F^n)  \cdot F^n + \int_0^t \int_{\T} f^n y^n 
\\
&\quad = \int_{\T} \left [ \tfrac{1}{2} |F^n|^2 (x,t) + (y^n v^n)(x,t) \right ] dx - \int_{\T} \left [ \tfrac{1}{2} |F_0^n|^2 (x) + (y_0^n v_0^n)(x) \right ] \, dx 
\end{aligned}
\end{equation}
We want to pass to the limit $n \to \infty$. The difficulty arises that the function $\frac{\del W}{\del F}(F) \cdot F$, under the hypothesis 
\eqref{hypsemiconvex}, has growth of order $O(|F|^p)$ and its weak limit cannot be represented via the usual Young measure Theorem.

To bypass this difficulty we will use a variant of a representation result from \cite[App A.2]{DST12} that we state and prove below.
Let $F(t,x)$ be a given matrix valued function, $\lambda_F$ a matrix and define 
$$
\phi(\lambda_F, F) = \Big (\frac{\del W}{\del F} (\lambda_F) - \frac{\del W}{\del F} (F) \Big ) : (\lambda_F - F ) 
$$
Observe $\phi(\lambda_F, F) \ge 0$ and by \eqref{hypsemiconvex} we have $|\phi(\lambda_F, F)| \le C \big ( |\lambda_F|^p + |F|^p + 1)$.
We prove the following representation theorem.

\begin{lemma}\label{replemma}
There exists a positive measure $\mu(dx dt) \ge 0$ so that
$$
\phi (F^n , F) 
\rightharpoonup \int  \phi(\lambda_F, F) \,  d\nu_{t,x} (\lambda_F)  + d\mu (dx dt)
$$
in the sense of measures, where the meaning of the integral is defined in the proof.
\end{lemma}

\begin{proof}
First, we define $\int  \phi(\lambda_F, F) \,  d\nu_{t,x} (\lambda_F)$. Let $R > 0$ and consider the truncated functions
$$
\phi_R (\lambda_F, F) = 
\begin{cases}
\phi(\lambda_F, F)  & \mbox{if}  \; \;  \phi(\lambda_F, F) < R \\
R  & \mbox{if}  \; \;  \phi(\lambda_F, F) \ge R  \\
\end{cases}
$$
Then $0 \le \phi_R (\lambda_F, F) \nearrow \phi(\lambda_F, F)$ a.e. on $Q_T := (0,T) \times \T$ and the monotone convergence Theorem 
implies
\begin{equation}\label{defnulimit}
\int  \phi(\lambda_F, F) \,  d\nu_{t,x} (\lambda_F)  = \lim_{R \nearrow \infty} \int \phi_R (\lambda_F, F) d\nu_{t,x} (\lambda_F)  \quad \mbox{a.e. $(t,x)$}.
\end{equation}
Observe that $0 \le \phi_R (F^n, F) \le R$ and the representation of weak limits via Young measures implies
that along a subsequence $\{ F^n \}$ we have 
$$
\phi_R (F^n, F) \rightharpoonup \int  \phi(\lambda_F, F) \,  d\nu_{t,x} (\lambda_F) \, ,  \quad \mbox{weak-$\ast$ in $L^\infty$}
$$
Formula \eqref{defnulimit} serves as the definition of the representation of the weak limit.

Consider next the sequence $\Big \{ \phi(F^n , F) - \int  \phi(\lambda_F, F) \,  d\nu_{t,x} (\lambda_F) \Big \}$ and observe that
$$
\sup_n \int_{Q_T} \Big | \phi(F^n , F) - \int  \phi(\lambda_F, F) \,  d\nu_{t,x} (\lambda_F) \Big | dt dx \le C
$$
Therefore, there exists a bounded measure $\mu (dx dt)$ such that
$$
\int_{Q_T} \Big (  \phi(F^n , F) - \int  \phi(\lambda_F, F) \,  d\nu_{t,x} (\lambda_F) \Big ) \theta (t, x)  dt dx  \to
\int_{Q_T} \theta(t,x) \mu (dt dx) 
$$
for all $\theta \in C (\overline{Q_T})$. Finally,  for $\theta(t,x) \ge 0$,  the chain of inequalities 
$$
\begin{aligned}
\int_{Q_T} \Big (  \int  &\phi(\lambda_F, F) \,  d\nu_{t,x} (\lambda_F) \Big ) \theta(t,x) dt dx 
\\
&=
\sup_{R>0} 
\int_{Q_T} \Big (  \int  \phi_R (\lambda_F, F) \,  d\nu_{t,x} (\lambda_F) \Big ) \theta(t,x) dt dx 
\\
&\le \sup_{R > 0} \lim_{n \to \infty} \int_{Q_T} \phi_R (F^n, F ) \theta (t, x) dt dx
\\
&\le \lim_{n \to \infty} \int_{Q_T} \phi ( F^n, F) \theta(t,x) dt dx
\end{aligned}
$$
implies that $\mu(dt dx) \ge 0$ as asserted.
\end{proof}

Using the representation lemma \ref{replemma} and letting $n \to \infty$ in \eqref{form101} gives
$$
\begin{aligned}
&\int_0^t \int_{\T} v^2  - \int_0^t \int_{\T}  \Big ( \overline{\big ( \frac{\del W}{\del F} (\lambda_F) - \frac{\del W}{\del F} (F) \big ) : (\lambda_F - F ) } 
 + \overline{\frac{\del W}{\del F} ( \lambda_F) } :  F \Big )  dx dt 
- \iint d\mu (dx dt) 
\\
&\quad + \int_0^t \int_{\T} f y  = \int_{\T} \left ( \overline{\tfrac{1}{2} |\lambda_F|^2 }  + (y v ) \right ) dx -  \limsup_{n \to \infty} \int_{\T}   \tfrac{1}{2} |F_0^n|^2 \, dx
- \int_{\T} y_0 v_0 dx 
\end{aligned}
$$

Consider next the weak form of the limiting equation to \eqref{appreqn}, this is what is often called measure-valued solution.
It has the form:
\begin{equation}
y_{t t} = \div_x \left ( \overline{\frac{\del W}{\del F} (\lambda_F) } + \nabla y_t \right ) + f  \, ,  \qquad (t,x) \in (0,T) \times \T \, .
\end{equation}
We multiply by $y$ and perform the analogous steps to obtain
$$
\begin{aligned}
&\int_0^t \int_{\T} v^2  - \int_0^t \int_{\T}  \overline{\frac{\del W}{\del F} ( \lambda_F) } : F dx dt + \int_0^t \int_{\T} f y 
\\
&\quad = \int_{\T} \left ( \tfrac{1}{2} |F|^2(x,t)  + (y v )(x,t) \right ) dx - \int_{\T} \left ( \tfrac{1}{2} |F_0|^2 + y_0 v_0 \right )dx 
\end{aligned}
$$

By comparing the identities obtained for the approximate family $(v^n, F^n)$ and for the measure valued solution $\big ( (v, F); \nu_{(t,x)} \big )$  
one obtains after some cancellations that
\begin{equation}
\label{veweaklimit}
\begin{aligned}
\int_0^t \int  \overline{\big ( \frac{\del W}{\del F} (\lambda_F) - \frac{\del W}{\del F} (F) \big ) : (\lambda_F - F ) }  \, dx \, dt
&+ \tfrac{1}{2} \int \overline{ |\lambda_F - F|^2 } dx + \iint \mu (dx dt)
\\
&=
\limsup_{n\to \infty} \int _{\T}   \tfrac{1}{2}  |F_0^n - F_0|^2 dx
\end{aligned}
\end{equation}
If the initial data  converge strongly $F_0^n \to F_0$ in $L^2$
then $\nu_{(0, x)} = \delta_{F_0(x)}$ and \eqref{veweaklimit} for  $W(F)$ strictly convex implies that $\nu_{(t,x)} = \delta_{F(t,x)}$ for a.e. $(t,x)$. 
The converse is also clear,
that is if $\nu_{(t,x)} = \delta_{F(t,x)}$ and the concentration measure $\mu (dx dt) = 0$ then the initial data converge strongly in $L^2$.

In fact, something even stronger is proved in  \cite{KLST23}: 
Even when $W(F)$ is semiconvex satisfying \eqref{hypsemiconvex}, if the initial data $F_0^n$ converge strongly to $F_0$ that
implies strong convergence of solutions.

\subsection{Finite-energy weak solutions for the compressible Navier-Stokes system}

The compressible Navier-Stokes system 
\begin{equation}
\label{eq:compNS1}
\begin{aligned}
\del_t \rho + \dive \rho u &= 0
\\
\del_t \rho u + \dive \rho u \otimes u  + \nabla p (\rho) &=  \dive \Big ( \mu (\nabla u + \nabla u^T) + \lambda (\dive u )  \, I \, \Big )
\end{aligned}
\end{equation}
describes the evolution of $(\rho, u)$,  where $\rho$ and $u$ are the density and velocity of the fluid,
of a compressible, barotropic gas with pressure $p (\rho)$. 
The pressure might be monotone, but in applications there is 
interest in non-monotone pressures like in case of a Van der Waals type pressure.

The parameters $\mu, \lambda$ are called shear and bulk viscosity. They satisfy conditions that ensure
the viscous stress which is isotropic is also dissipative. The viscous stress $S$ may be expressed via the symmetric part of the strain rate
$D = \tfrac{1}{2} (\nabla u + \nabla u^T )$ as 
$$
\begin{aligned}
S &= \mu (\nabla u + \nabla u^T ) + \lambda (\div u ) I
\\
&= 2 \mu D + \lambda (\tr D ) I
\end{aligned}
$$
Accordingly, the viscous dissipation reads
$$
\begin{aligned}
S : \nabla u = S : D &= 2 \mu |D|^2 + \lambda (\tr D)^2
\\
&= 2 \mu \big ( D - \tfrac{1}{3} (\tr D) I \big ) : \big ( D - \tfrac{1}{3} (\tr D) I \big ) + \big ( \lambda + \tfrac{2\mu}{3} \big ) (\tr D)^2  \, .
\end{aligned}
$$
Therefore, the viscous dissipation $S : D > 0$ for all $D \ne 0$ 
if and only if $\mu >0$ and $\lambda +  \tfrac{2\mu}{3} > 0$. (This condition is replaced by $\mu >0$ and $\lambda +  \tfrac{2\mu}{d} > 0$
for dimension $d$.)

The Cauchy problem is considered with either Dirichlet or periodic boundary conditions, and initial conditions for $x \in \Omega$,
\begin{equation}\label{ID}
\rho(0,x) = \rho_0 (x) \, , \quad u(0,x) = u_0(x) \, . 
\end{equation}
The solution $(\rho, u)$ is sought in $Q_T = (0,T) \times \Omega$. Local existence of classical solutions and 
global existence of small perturbations of constant states are classical results.
A global existence result for weak solutions for initial densities close to a constant was proved in \cite{Hoff95}.
Global weak solutions that dissipate the energy were constructed in the celebrated results of
\cite{b-Lions98}, \cite{b-Feireisl04}.

The equations \eqref{eq:compNS1} are equipped with the (formal) energy estimate
$$
\frac{d}{dt} E [\rho, \rho u] (t) + \int_\Omega \mu |\nabla u|^2 + (\lambda + \mu) (\div u )^2 dx = 0
$$
where the energy functional is given by
$$
E [\rho, \rho u] (t)  = \int_\Omega \rho |u|^2 + h (\rho) dx \, .
$$
The energy identity is obtained by multiplying respectively \eqref{eq:compNS1}$_1$ by $h^\prime (\rho)$ and \eqref{eq:compNS1}$_2$ by $u$ and adding the results.
The internal energy per unit volume $h(\rho)$ is connected to the pressure $p(\rho)$ by the thermodynamic relation
$$
p(\rho) = \rho h^\prime (\rho) - h (\rho) \, .
$$

The case of the isentropic gas law
\begin{equation}\label{isengas}
p(\rho) = \rho^\gamma \, , \quad h(\rho) = \tfrac{1}{\gamma -1} \rho^\gamma \, , \quad \mbox{ and $\gamma > 1$}.
\tag{IG}
\end{equation}
has provided a well studied  example for which there exists an existence theory of global weak solutions, usually called
the Lions-Feireisl theory. The theory was first developed for various initial-boundary value problems in the book of 
Lions \cite{b-Lions98} with significant improvements obtained by Feireisl-Novotn\'y{}-Petzeltov\'a \cite{FNP02} and Feireisl \cite{b-Feireisl04}.
Consider an isentropic gas law \eqref{isengas} with Dirichlet boundary conditions
$u = 0$ on $\del \Omega$, or periodic boundary conditions $\Omega = \T$.

\begin{definition}\label{def:fews}
A pair $(\rho, u)$ on $Q_T$ is called a finite energy weak solution if 
\begin{itemize}
\item [(i)] $\rho \in L^\infty \big ( (0,T) ; L^\gamma (\Omega) \big )$, $u \in L^2 \big ( (0,T) ; H_0^1 (\Omega) \big )$
\item [(ii)]  $(\rho, u)$ satisfies \eqref{eq:compNS1} in distributions
\item[(iii)] $\rho$ is a renormalized solution of the transport equation, that is 
$$
\del_t b(\rho) + \div_x ( b(\rho) u) + \big ( \rho b^\prime(\rho) - b(\rho) \big ) (div_x u ) = 0 \, , \quad \mbox{in $\mathcal{D}^\prime(Q_T)$}
$$
where $b \in C^1 ([0, \infty))$ such that $b^\prime(\xi) = 0$ for $\xi$ sufficiently large.

\item[(iv)]  $E [\rho, \rho u] (t)  \in L^1_{loc} (0,T)$ satisfies
$$
\frac{d}{dt} E [\rho, \rho u] (t) + \int_\Omega \mu |\nabla u|^2 + (\lambda + \mu) (\div u )^2 dx \le 0 \, , \quad \mbox{in $\mathcal{D}^\prime (0,T)$}.
$$
\end{itemize}
\end{definition}

The existence result in \cite{b-Lions98,FNP02,b-Feireisl04} is stated:

\begin{theorem} Let $p(\rho)$ satisfy \eqref{isengas} with $\gamma > \frac{3}{2}$ and assume $\rho^0 \in L^\gamma (\Omega)$ and $m^0$
satisfy $\rho^0 \ge 0$
$$
\begin{aligned}
m^0 (x) = 0 \; \; &\mbox{whenever} \; \; \rho^0 (x) = 0 
\\
\frac{|m^0|^2}{\rho^0} &\in L^1 (\Omega)
\end{aligned}
$$
Then there exists a global finite energy solution of \eqref{eq:compNS1}, \eqref{ID} with Dirichlet or periodic boundary conditions.
\end{theorem}

The details of the Lions-Feireisl theory are not within the scope of this work, and, besides the aforementioned references, the reader is referred 
to the book by  Novotn\'y{} and Stra\v skraba \cite{b-NS04}. The work of Bresch-Jabin \cite{BJ18} contains a very nice and concise outline of the key aspects
of the Lions-Feireisl theory, and provides extensions and quantitative compactness estimates for general pressure laws
$p = p(\rho)$ including non-monotone pressures.

The Lions-Feireisl existence theory establishes compactness properties for finite energy approximate solutions
$(\rho^n , u^n)$. A key ingredient is a propagation of compactness result, initially formulated by Lions, that indicates that if the initial data for the density
converge strongly $\rho^n (x,0) \to \rho (x,0)$ then this 
property transfers to solutions $\rho^n (\cdot, t) \to \rho (\cdot, t)$; see \cite[Sec 5.2]{b-Lions98}.
It is also noted in Remark 5.9 of \cite[Sec 5.2]{b-Lions98} that (for $p(\rho) = \rho^\gamma$ in the whole space $\R^d$) if the weak limit $(\rho, u)$ is a solution of compressible Navier-Stokes then necessarily $\rho_0^n$ converges strongly to $\rho_0$. 

Lions raised the issue of propagation of oscillations and homogenization \cite[Rmk 5.8]{b-Lions98}. He constructs  an exact periodic solution of the 
problem
\begin{equation}
\begin{aligned}
\del_t \rho + \dive \rho u &= 0
\\
\dive \Big ( \mu (\nabla u + \nabla u^T) + \lambda (\dive u )  \, I \, \Big ) - \nabla \rho^\gamma &=0
\\
\int _{\T } \rho u dx = \int_{\T} m_0 dx
\end{aligned}
\end{equation}
in $\R^d \times [0,\infty)$ with $m_0$ given, and uses it to produce a sequence $(\rho^n, u^n)$ of periodic solutions
with persistent oscillations of a forced variant of \eqref{eq:compNS1} 
with $p(\rho) = \rho^\gamma$.

An example of a sequence $(\rho^n, u^n)$ with persistent oscillations for \eqref{eq:compNS1} with non-monotone pressure
is constructed in \cite{Tzavaras24}. The example consists of uniform shear flows that have continuous velocities 
but discontinuous densities, 
jumping between different phases across characteristic surfaces,  see Section \ref{subseq:CNS}.

%
%
%

%\vfil\eject

\section{Oscillations in Linear Hyperbolic-Parabolic Systems}\label{sec:linosc}

In this section we study oscillations in linear hyperbolic-parabolic systems. 
 The form of solutions is motivated by examples that arise
in the theory of thermoviscoelasticity and relate to the mechanical interpretation of such models.

\subsection{Linear hyperbolic systems}
Linear hyperbolic systems (in 1-d for simplicity)
$$
\del_t U + A \del_x U = 0
$$
admit oscillatory solutions of the form
$$
U(t,x) = \xi e^{ i n \big ( x - \lambda t \big ) }
$$
where $\lambda$ are the characteristic speeds, defined as the eigenvalues of the problem
$$
(A - \lambda I ) \xi = 0 \, .
$$
Oscillations persist for linear hyperbolic systems, they are induced by the initial data, and they travel on characteristic directions.

\subsection{Linear hyperbolic-parabolic systems}
By contrast, when the viscosity matrix $B > 0$ of a hyperbolic-parabolic system
\begin{equation}\label{hyperpara}
\del_t U + A \del_x U = B \del^2_x U
\end{equation}
then oscillations are suppressed.

Consider now the question:  What happens for the case of Hyperbolic-Parabolic Systems with $B$ singular?
We focus on examples of systems with singular diffusion matrices 
originating from thermomechanics.

\medskip
\subsubsection{An illustrating example} 
\label{sec:ex1d}
The special case of linear viscoelasticity in one-space dimension is the equation
\begin{equation}
\label{linearvisc}
u_{tt} = \lambda u_{xx} + \mu u_{t x x}  \qquad x \in (-\pi, \pi) \, , t > 0 \, ,
\end{equation}
where $u : (0,T)\times \Tone \to \R$. This equation is closely related to the so called telegraph equation.

We explore an ansatz of the form 
$$
u(t,x) = \frac{1}{n} \alpha(t) e^{ i n x} 
$$
and examine conditions for existence of such solutions with high frequency $n$.
The amplitude $\alpha (t)$ needs to satisfy
\begin{equation}
\label{exode1}
\frac{d^2\alpha}{d t^2} + \mu n^2 \frac{d \alpha}{d t} + n^2 \lambda \alpha = 0 \, ,
\end{equation}
where the associated characteristic polynomial  $\rho^2 + \mu n^2 + \lambda n^2 = 0$ has roots
$$
\rho_\pm = \frac{\mu n^2}{2} \left ( - 1 \pm \sqrt{ 1 - \frac{4 \lambda}{\mu^2 n^2}} \right ) \, .
$$
These eigenvalues are real for $n$ large, and enjoy the asymptotic expansions
$$
\begin{aligned}
\rho_-  &= - \mu n^2 + \frac{\lambda}{\mu} + \frac{\lambda^2}{\mu^3 n^2} + O \big ( \tfrac{1}{n^4} \big)
\\
\rho_+  &= - \frac{\lambda}{\mu} - \frac{\lambda^2}{\mu^3 n^2} + O \big ( \tfrac{1}{n^4} \big)
\end{aligned}
$$
They give rise to solutions $\alpha_\pm = e^{\rho_\pm t}$ for \eqref{exode1} and $u_{n , \pm} = \frac{1}{n}  e^{ i n x + \rho_\pm t}$ for \eqref{linearvisc}.
Of interest here is the solution associated to $\rho_+$ and we consider $(u_n, v_n)$ where 
\begin{equation}
\label{beh1}
u_{n}  (t,x) =  \frac{1}{n} \exp \left \{   i n x   - \frac{\lambda}{\mu} t  - \frac{\lambda^2}{\mu^3 n^2} t + O \big ( \tfrac{1}{n^4}   \big)  t   \right \}
\end{equation}
and $v_n = \frac{\del u_n }{\del t}$.   Note that $u_n \, ,  v_n \to 0$ as $n \to \infty$.  By contrast,
\begin{equation}
\label{beh2}
\begin{aligned}
\frac{\del u_n}{\del x}  &= i  \left ( 1 + O \big ( \frac{1}{n^2} \big ) \right ) \exp \left \{   i n x   - \frac{\lambda}{\mu} t   \right \}
\\
\frac{\del v_n}{\del x} &=  i  \left (  - \frac{\lambda}{\mu}  + O \big ( \frac{1}{n^2} \big ) \right ) \exp \left \{   i n x   - \frac{\lambda}{\mu} t   \right \}
\end{aligned}
\end{equation}
exhibit  persistent oscillations induced by the initial data.

To put this example into the perspective, note that \eqref{linearvisc} may be expressed as the hyperbolic-parabolic system,
\begin{equation}
\label{exode2}
\begin{aligned}
\frac{\del w}{\del t}   &= \; \frac{\del v}{\del x}
\\
\frac{\del v}{\del t} &= \lambda \frac{\del w}{\del x} +  \mu \frac{\del^2 v }{\del x^2} 
\end{aligned}
\end{equation}
with $w = \frac{\del u}{\del x}$. 
The oscillation in \eqref{beh2} reflects a balance between the terms $w$ and $\frac{\del v}{\del x}$
and is stationary in time.

It should be contrasted with the usual oscillatory solutions of hyperbolic systems. Recall, that the hyperbolic model
\begin{equation}
\label{exode4}
\begin{aligned}
\frac{\del w}{\del t}   &= \; \frac{\del v}{\del x}
\\[3pt]
\frac{\del v}{\del t} &= \lambda \frac{\del w}{\del x} 
\end{aligned}
\end{equation}
has oscillatory solutions propagating in the directions of the wave speeds
$$
\begin{bmatrix}
w  \\[5pt]  v  
\end{bmatrix} (t,x) 
= \xi_\pm  \exp  \big \{  i n (x - \pm{\sqrt{\lambda}} t  )  \big \}
$$
where $\xi_\pm$ are the eigenvectors associated to the wave speeds $A \xi_\pm = \pm \sqrt{\lambda} \xi_\pm$. These propagate 
with speeds $\pm \sqrt{\lambda}$ and are of different nature than \eqref{beh1}, \eqref{beh2} which are  stationary in time.

\subsubsection{1-d linear thermoviscoelasticity} \label{sec:osc1d}
The results of the example of Section \ref{sec:ex1d} are now extended to the system of 1-d thermoviscoelasticity
\begin{equation}
\label{lintve1}
\begin{aligned}
\frac{\del^2 u}{\del t^2} &= \lambda \frac{\del^2 u}{\del x^2} + m \frac{\del \theta}{\del x} + \mu \frac{\del^2}{\del x^2} \frac{\del u}{\del t}
\\[3pt]
\frac{\del \theta}{\del t} &= \kappa \frac{\del^2 \theta}{\del x^2} + m \frac{\del^2 u}{\del x \del t}
\end{aligned}
\end{equation}
Note that  $(u(t,x), \theta(t,x))$ describes a thermomechanical process consisting of the motion $u$ and the temperature $\theta$,
while $\lambda > 0$, $m \in \R$, $\mu > 0$, $\kappa > 0$ are parameters describing  the elasticity modulus, thermoelastic coupling, 
viscosity and heat diffusivity, respectively. The system \eqref{lintve1} is  expressed as a hyperbolic-parabolic system by introducing
the velocity  $v = \frac{\del u}{\del t}$ and the elastic strain $w = \frac{\del u}{\del x}$ and writing it as
\begin{equation}
\label{lintve2}
\begin{aligned}
\del_t  \begin{pmatrix} w \\ v \\ \theta \end{pmatrix}
= 
\begin{pmatrix} 0 & 1 & 0 \\ \lambda & 0 & m \\ 0 & m & 0 \end{pmatrix}
\del_x \begin{pmatrix} w \\ v \\  \theta \end{pmatrix} 
+
\begin{pmatrix} 0 & 0 & 0 \\  0 & \mu & 0 \\ 0 & 0 & \kappa \end{pmatrix}
\del_{x x} \begin{pmatrix} w \\ v \\ \theta \end{pmatrix}
\end{aligned}
\end{equation}
Solutions of \eqref{lintve1} satisfy the energy identity
\begin{equation}
\begin{aligned}
&\frac{\del}{\del t} \left ( \tfrac{1}{2} \Big | \frac{\del u}{\del t} \Big |^2 + 
 \tfrac{\lambda}{2} \Big | \frac{\del u}{\del x} \Big |^2  + \tfrac{1}{2} | \theta  |^2   \right )
+ \kappa \Big | \frac{\del \theta }{\del x} \Big |^2  + \mu \Big | \frac{\del^2 u}{\del x \del t}  \Big |^2 
\\
&\qquad = 
\frac{\del}{\del x} \left (  \frac{\del u}{\del t} \big (  \lambda  \frac{\del u}{\del x}  +  m  \theta + \mu  \frac{\del^2 u}{\del x \del t}  \big )  \right )
+ \frac{\del}{\del x} \left ( \kappa \theta \frac{\del \theta}{\del x} \right ) 
\end{aligned}
\end{equation}
Under appropriate boundary conditions (either periodic or no mechanical work and adiabatic at the boundaries) we obtain the energy
dissipation identity
\begin{equation}
\label{energy1}
\frac{d}{dt} \int \Big ( \tfrac{1}{2} \Big | \frac{\del u}{\del t} \Big |^2 +   \tfrac{\lambda}{2} \Big | \frac{\del u}{\del x} \Big |^2  + \tfrac{1}{2} | \theta  |^2  \Big )\, dx
+ \int \kappa \Big | \frac{\del \theta }{\del x} \Big |^2  + \mu \Big | \frac{\del^2 u}{\del x \del t}  \Big |^2  = 0 \, .
\end{equation}
which illustrates that the dissipation is partial arising from  a singular diffusion matrix.

Next, apply an {\it ansatz} of solutions of the form
\begin{equation}
\label{osc1}
\begin{aligned}
u(t,x) &= \frac{1}{n} a(t) e^{ i n x}
\\
\theta(t,x)  &=   - i b  (t) e^{ i n x}
\end{aligned}
\end{equation}
where $i$ is the imaginary unit, $n$ an integer, and $(a(t), - i b(t))$ are complex amplitude functions.
Introducing \eqref{osc1} to \eqref{lintve1} we deduce that 
 $(a(t), b(t))$ satisfies the system of ordinary differential equations
\begin{equation}
\label{ode1}
\begin{aligned}
\ddot{a} &= - \lambda n^2  a + m n^2 b - \mu n^2   \dot a
\\
\dot{b} &= - \kappa n^2 b  - m  \dot{a}
\end{aligned}
\end{equation}
The amplitude $(a, b)$ in \eqref{ode1} was introduced as complex-valued but it is also consistent to select it real-valued, and
we opt for that selection here.
For the {\it ansatz} \eqref{osc1} we have
$$
\begin{aligned}
\frac{\del u}{\del x} &= i a(t) e^{i n x}
\\
\frac{\del u}{\del t} &= \frac{1}{n} \dot{a}(t) e^{i n x}
\\
\frac{\del^2 u}{\del x \del t} &= i \dot{a}(t) e^{i n x}
\end{aligned}
$$
We expect  oscillations of $u(t,x)$ and  the velocity $v = \frac{\del u}{\del t}(t,x)$ to decay to zero (for $n$ large)
while oscillations of the strain $w = \frac{\del u}{\del x}(t,x)$, the strain rate $\frac{\del v}{\del x} = \frac{\del^2 u}{\del x \del t}(t,x)$ 
are expected to persist in time. This has to be validated by studying the behavior of solutions of \eqref{ode1}.

Setting $v = \dot{a}$, \eqref{ode1} is expressed as a first order system,
\begin{equation}
\label{ode2}
\frac{d}{dt} 
\begin{bmatrix}
a \\ v \\ b
\end{bmatrix}
=
\begin{bmatrix}
0 & 1 & 0 \\
- \lambda n^2 &  - \mu n^2   &   m n^2  \\
0 & - m & - \kappa n^2 
\end{bmatrix}
\begin{bmatrix}
a \\ v \\ b
\end{bmatrix}
\end{equation}
Let $A = A(n)$ be the coefficient matrix. The eigenvalues $\rho (n)$ of $A(n)$ are computed as 
roots of the cubic polynomial $\det (A - \rho I) = 0$ which reads
\begin{equation}
\label{algeq1}
\rho^3 + (\kappa + \mu) n^2 \rho^2 + (\kappa \mu n^4 + \lambda n^2 + m^2 n^2 ) \rho + \kappa \lambda n^4 = 0 \, .
\end{equation}

To proceed one considers the eigenvalue problem
$$
( A(n) - \rho(n) I) \xi(n) = 0
$$
and expands the eigenvalues and eigenvectors as functions of $\frac{1}{n}$ to study the behavior of high frequencies.
After a tedious but straightforward calculations (see \cite{Tzavaras23}) one arrives at an asymptotic expansion of
the oscillatory solutions for \eqref{lintve1} which read
\begin{equation}
\label{oscsol1}
\begin{bmatrix}
u \\[5pt]  \frac{\del u}{\del t}  \\[5pt] \theta
\end{bmatrix} (t,x)
=
\begin{bmatrix} \tfrac{1}{n} a(t)  \\  \tfrac{1}{n} v(t)   \\ - i b(t) \end{bmatrix} 
e^{i n x}  
=
\begin{bmatrix}   \frac{1}{n} \mu - \frac{\lambda}{\mu} \frac{1}{n^3}  +  O ( \frac{1}{n^5} )  \\[5pt] -\frac{1}{n} \lambda +  \frac{m^2 \lambda}{\kappa \mu} \frac{1}{n^3} +  O ( \frac{1}{n^5} ) \\[5pt]  
- i  \frac{m \lambda}{\kappa} \frac{1}{n^2} + O ( \frac{1}{n^4} ) \end{bmatrix} 
\exp \Big \{ i n x -\frac{\lambda}{\mu} t +O \big (  \frac{1}{n^2} \Big )  t )  \Big \}
\end{equation}
Considering \eqref{oscsol1} we see that $u$, $\frac{\del u}{\del t}$, $\theta$ and $\frac{\del \theta}{\del x}$ converge to zero for $n$ large, but oscillations persist
for the strain $w = \frac{\del u}{\del x}$, the strain rate $\frac{\del v}{\del x}$ and $\frac{\del^2 \theta}{\del x^2}$. To leading order these solutions 
behave like
\begin{equation}
\label{oscsol2}
\begin{bmatrix}
w  \\[5pt]  \frac{\del v}{ \del x}  \\[5pt] \frac{\del^2 \theta}{\del x^2}
\end{bmatrix} (t,x) 
=
\begin{bmatrix}
\frac{\del u}{\del x}  \\[5pt]  \frac{\del^2 u}{\del t \del x}  \\[5pt] \frac{\del^2 \theta}{\del x^2}
\end{bmatrix} (t,x) 
=
i \, \begin{bmatrix}   \mu + O ( \frac{1}{n^2} )  \\[5pt] -  \lambda +   O ( \frac{1}{n^2} ) \\[5pt]  
  \frac{m \lambda}{\kappa} + O ( \frac{1}{n^2} ) \end{bmatrix} 
\left ( 1 +O \big (  \frac{1}{n^2}   \big ) t  \right ) \exp \Big \{ i n x -\frac{\lambda}{\mu} t  \Big \}
\end{equation}

\begin{remark}
For the special case of \eqref{linate} with $\kappa = 0$, one may construct solutions of \eqref{lintve2} with $\kappa = 0$ for which 
oscillations appear to leading order for the vector function $(w, v_x, \theta)$. This should be contrasted to the behavior
of \eqref{lintve2} (which has $\kappa \ne 0$) where oscillations appear to  leading order for the vector function $(w, v_x, \theta_{xx})$.
\end{remark}

\subsection{Multi-dimensional linear thermoviscoelasticity} \label{sec:oscmd}
We devise next some analogous solutions for the multidimensional system of linear thermoviscoelasticity
\begin{equation}
\label{lintved}
\begin{aligned}
\frac{\del^2 u_k}{\del t^2} &= \frac{\del}{\del x_\alpha} \left ( A_{k l \alpha \beta} \frac{\del u_l }{\del x_\beta} + M_{k \alpha} \theta \right ) 
+ \mu \Delta \left ( \frac{\del u_k}{\del t} \right )
\\[5pt]
\frac{\del \theta}{\del t} &= \kappa  \Delta \theta + M_{k \alpha} \frac{\del}{\del x_\alpha} \left ( \frac{\del u_k}{\del t} \right )
\end{aligned}
\end{equation}
The system describes the evolution of a thermomechanical process $(u, \theta) : (0,\infty)\times \R^d \to \R^d \times \R^+$. Here, the summation convention is used, 
the parameters $\kappa > 0, \mu > 0$ describe the viscosity and thermal diffusivity, $A_{k l \alpha \beta}$ is a fourth order rank-1 convex tensor which
is symmetric in the sense $A_{k l  \alpha \beta} = A_{l k  \alpha \beta}$, $A_{k l \alpha \beta} = A_{k  l \beta \alpha}$ and describes the elastic response, 
while $M_{k \alpha}$ is a second order tensor describing the thermoelastic coupling.

Consider an {\it ansatz} of oscillating solutions of the form
\begin{equation}
\label{oscd}
\begin{aligned}
u_k (t, x) &= \frac{1}{n} a_k (t) e^{ i n \nu \cdot x}   \qquad k = 1, ... , d
\\
\theta(t,x) &= - i b (t) e^{ i  n \nu \cdot x}
\end{aligned}
\end{equation}
where $\nu = (\nu_1, ... , \nu_d)$ is a unit vector,  $| \nu | =1$, and n is a natural number. Introducing the ansatz \eqref{oscd}  to \eqref{lintved} we see
that  the amplitudes $a_k (t)$, $\beta (t)$ satisfy the system of ordinary differential equations
\begin{equation}
\label{oded}
\begin{aligned}
\ddot{a}_k  &= - n^2  (A_{k l \alpha \beta} \nu_\alpha \nu_\beta)  a_l + n^2 (M_{k \alpha} \nu_\alpha ) b- \mu n^2 \dot{a}_k
\\
\dot{b} &= - n^2 \kappa b - ( M_{k \alpha} \nu_\alpha ) \dot{a}_k
\end{aligned}
\end{equation}
Note that \eqref{oded} admits real valued solutions $(a_1, ... , a_d , b)$ which by taking real and imaginary parts in \eqref{oscd} give rise to 
real-valued oscillating solutions for \eqref{lintved}.

Recalling that the summation convention is used, we proceed to solve \eqref{oded}. As already noted,
the fourth order tensor $A_{k l  \alpha \beta}$ is assumed symmetric and rank-1 convex:
\begin{align}
\label{hypsym}
A_{k l  \alpha \beta} = A_{l k  \alpha \beta} \,  , \quad &A_{k l \alpha \beta} = A_{k  l \beta \alpha} \, ,
\tag{H$_1$}
\\
\label{rank1cvx}
A_{k l  \alpha \beta}  \nu_\alpha \nu_\beta \xi_k \xi_l > 0 \quad &\forall \xi \in \R^d - \{0\} \, ,  \nu \in S^{d-1} \, .
\tag{H$_2$}
\end{align}
Hypothesis \eqref{rank1cvx} implies that the acoustic tensor $Q_{k l} := A_{k l  \alpha \beta}  \nu_\alpha \nu_\beta $ is symmetric, 
positive definite. It thus has positive eigenvalues $\lambda^r > 0$, and a complete set of linearly independent eigenvectors $\xi^r$,
$r = 1, ... , d$, satisfying
\begin{equation}
\label{eigenv}
Q_{k l} \xi_l^r = \lambda^r \xi_k^r  \quad k = 1, ... , d \, \; \; \mbox{ normalized so that $|\xi^r| = 1$. }
\end{equation}
We now place an assumption on the thermoelastic interaction matrix $M_{k \alpha}$. Given $\xi^r$ we assume there is $\nu \in S^{d-1}$ 
such that for some $m \in \R$
\begin{equation}
\label{hypm}
M_{k \alpha } \nu_\alpha =  m \xi^r_k   \quad k = 1, ... d .
\tag{H$_3$}
\end{equation}
Note that $|\xi^r|=1$ implies $m^r = \xi_k M_{k \alpha} \nu_\alpha$.
Hypothesis \eqref{hypm} is automatically satisfied when $\text{rank} M = d$. Otherwise, it is an assumption connecting
pairs of eigenvectors $\xi^r$ to  associated propagation directions $\nu$.

With these in place, fix $r = 1, ... , d$ and proceed to solve \eqref{oded}. We make the ansatz
\begin{equation}
\label{anform}
a_k (t) = \alpha(t) \xi_k^r  \quad \mbox{ where  $\xi^r \in \R^d$, $|\xi^r | = 1$ and $\alpha(t)$ scalar valued}.
\end{equation}
Then  $(\alpha (t), b (t))$ satisfies the system of ordinary differential equations
\begin{equation}
\label{ode3}
\begin{aligned}
\ddot{\alpha} &=  - n^2 \lambda^r \alpha + n^2 m^r b - \mu n^2 \dot{\alpha}
\\
\dot{b} &= - \kappa n^2 b - m^r \dot{\alpha}
\end{aligned}
\end{equation}
where $m^r = \xi^r \cdot M \nu$, and $\lambda^r$,  $\xi^r$ are an eigenvalue and corresponding eigenvector of the acoustic tensor.
Conversely, \eqref{hypm}, \eqref{eigenv} imply solutions of \eqref{ode3} generate via \eqref{anform}  solutions of \eqref{oded}.

The system \eqref{ode3} is precisely \eqref{ode2} studied  in Section \ref{sec:osc1d}.
Combining the analysis of subsection \ref{sec:osc1d} with \eqref{anform} we conclude that if \eqref{hypsym}, \eqref{rank1cvx} and \eqref{hypm} are satisfied 
for an eigenvalue-eigenfunction pair $(\lambda^r, \xi^r)$ and associated direction $\nu$, there will be oscillating solutions for \eqref{lintved}
of the form
\begin{equation}
\label{oscsold}
\begin{aligned}
\begin{bmatrix}
u \\[5pt]  \frac{\del u}{\del t}  \\[5pt] \theta
\end{bmatrix} (t,x)
&=
\begin{bmatrix} \tfrac{1}{n} \alpha(t) \xi^r  \\[5pt]  \tfrac{1}{n} \dot{\alpha} (t) \xi^r   \\[5pt]  - i b (t) \end{bmatrix} 
e^{i n \nu \cdot x}  
=
\begin{bmatrix}   \left ( \frac{1}{n} \mu + O \big ( \frac{1}{n^3}  \big ) \right ) \xi^r  \\[8pt] 
\left ( -\frac{1}{n} \lambda^r +  O ( \frac{1}{n^3} )  \right ) \xi^r \\[8pt]  
- i \frac{1}{n^2} \frac{m^r  \lambda^r }{\kappa} + O ( \frac{1}{n^4} ) \end{bmatrix} 
\exp \Big \{ i n \nu \cdot x -\frac{\lambda^r}{\mu} t +O \big (  \frac{1}{n^2} \Big )  t )  \Big \}
\end{aligned}
\end{equation}
where $(\alpha(t), b(t))$ satisfy \eqref{ode3}.
These provide oscillating progressive waves in the direction $\nu$. Again $u$, $v = \frac{\del u}{\del t}$
and $\theta$ converge to zero, $\nabla u$, $\nabla v$ and $\Delta \theta$ have persistent oscillations
\begin{equation}
\label{behexa2}
\begin{aligned}
\nabla u &=  i \mu  ( \xi^r \otimes \nu )  \left ( 1 +  O \big (  \frac{1}{n^2} \Big )  t \right ) \exp \Big \{ i n \nu \cdot x -\frac{\lambda^r}{\mu} t  \Big \}
\\
\Delta \theta &=  i  \frac{m^r  \lambda^r }{\kappa}   \left ( 1 +  O \big (  \frac{1}{n^2} \Big )  t \right ) \exp \Big \{ i n \nu \cdot x -\frac{\lambda^r}{\mu} t  \Big \}
\end{aligned}
\end{equation}
induced by oscillations in the initial data.

\subsection{Linear adiabatic thermoviscoelasticity}
\label{sec:ex2}
Another system where oscillatory solutions may be constructed explicitly is the case of adiabatic thermoviscoelastcity, namely
the special case of \eqref{lintved} when $\kappa = 0$, which reads
\begin{equation}
\label{linate}
\begin{aligned}
\frac{\del^2 u_k}{\del t^2} &= \frac{\del}{\del x_\alpha} \left ( A_{k l \alpha \beta} \frac{\del u_l }{\del x_\beta} + M_{k \alpha} \theta \right ) 
+ \mu \frac{\del}{\del x_\alpha} \frac{\del}{\del x_\alpha} \left ( \frac{\del u_k}{\del t} \right )
\\[5pt]
\frac{\del \theta}{\del t} &=  M_{k \alpha} \frac{\del}{\del x_\alpha} \left ( \frac{\del u_k}{\del t} \right )
\end{aligned}
\end{equation}
Introducing \eqref{oscd} leads to $(a_1, ... a_d , b)$ solving the system of ordinary differential equations \eqref{oded}
with $\kappa = 0$, If  $a(t) = (a_1, ... a_d ) (t)$, $b(t)$ are selected by solving
\begin{align}
\label{odeds}
\ddot{a}_k  &= - n^2  (A_{k l \alpha \beta} \nu_\alpha \nu_\beta)  a_l -  n^2 (M_{k \alpha} \nu_\alpha ) ( M_{l \beta} \nu_\beta ) a_l  - \mu n^2 \dot{a}_k
\\
\label{eqb}
b &= - ( M_{k \alpha} \nu_\alpha ) a_k \, ,
\end{align}
then  $(a, b)$  produce  via \eqref{oscd} a solution of \eqref{linate}.

We impose the hypotheses \eqref{hypsym}, \eqref{rank1cvx} of rank-1 convexity for $A_{k l \alpha \beta}$ (but not hypothesis \eqref{hypm}).
Fix a direction $\nu \in S^{d-1}$ and consider the modified acoustic tensor
$$
Q^m_{kl} = A_{k l \alpha \beta} \nu_\alpha \nu_\beta + (M_{k \alpha} \nu_\alpha ) ( M_{l \beta} \nu_\beta )
$$
By \eqref{rank1cvx} the matrix $Q^m_{kl}$ is positive definite. Let $\sigma^{ r} > 0$ be an eigenvalue and $\zeta^r$ the corresponding eigenvector,
$$
Q^m \zeta^r = \big ( Q + ( M \nu ) \otimes (M \nu) \big ) \zeta^r = \sigma^r \zeta^r \, , \quad | \zeta^r| =1  \, .
$$
Set $a_k = \alpha (t) \zeta^r_k$, then \eqref{odeds} is equivalent to solving the equation for the amplitude $\alpha (t)$
\begin{equation}
\label{exode3}
\frac{d^2\alpha}{d t^2} + \mu n^2 \frac{d \alpha}{\del t} + n^2 \sigma^r \alpha = 0 \, ,
\end{equation}
which is precisely the example studied in Section \ref{sec:ex1d}.

The analysis in Section \ref{sec:ex1d} indicates there are two solutions of \eqref{exode3} one that decays fast, and
a second denoted there by  $\alpha_+ (t) = e^{\rho_+ t}$ which decays slowly. Utilizing \eqref{oscd}, \eqref{eqb} the slowly decaying solution we conclude that \eqref{linate}
has solutions of the form
\begin{align}
u_k = \frac{1}{n} \alpha (t) \zeta_k^r e^{i n \nu \cdot x}  &= \frac{1}{n}  \zeta_k^r  \exp \left \{ i n \nu \cdot x - \frac{\sigma^r}{\mu}t + O \big ( \frac{1}{n^2} \big ) t \right \}
\\
\nonumber
\theta = i ( M_{k \alpha} \nu_\alpha  \zeta^r_k ) \alpha (t) e^{i n \nu \cdot x}  &= i ( M_{k \alpha} \nu_\alpha  \zeta^r_k ) \exp \left \{ i n \nu \cdot x - \frac{\sigma^r}{\mu}t + O \big ( \frac{1}{n^2} \big ) t \right \}
\\
\label{behtheta}
&= i (\zeta^r \cdot M \nu ) \Big ( 1 + O \big ( \frac{1}{n^2} \big )  t \Big )  \exp  \left \{ i n \nu \cdot x - \frac{\sigma^r}{\mu}t  \right \}
\end{align}
Note that  $u_n , v_n \to 0$ as $n \to \infty$ but oscillations persist for $\theta$,  
\begin{equation}
F = \nabla u = i ( \zeta^r \otimes \nu)   \Big ( 1 + O \big ( \frac{1}{n^2} \big )  t \Big )   \,  \exp  \left \{ i n \nu \cdot x - \frac{\sigma^r}{\mu}t  \right \}
\end{equation}
and  $\nabla v$.

In the above example there are oscillations of $\theta$ to leading order, see \eqref{behtheta}. This behavior should be contrasted to the example of Section 
\ref{sec:oscmd} where, due to the presence of heat diffusion, oscillations first appear in $\Delta \theta$, see \eqref{behexa2}.

%\vfil\eject

\section{Quasi-linear dynamical models in the dynamics of phase transitions}\label{sec:phasetransitions}

In this section examples are recorded of solutions exhibiting sustained oscillations for 
various quasi-linear hyperbolic-parabolic systems. The first class of models concerns viscoelastic systems of
the type used in modeling of phase transitions.

\subsection{Longitudinal motions of a viscoelastic bar}\label{sec:1dnonmon}

We outline an oscillatory solution for 
the system 
\begin{equation}
\label{longonev}
\begin{aligned}
u_t &= v_x
\\
v_t &= \sigma(u)_x +  \del_x \left ( \frac{\mu}{u} v_{x} \right ) \, ,
\end{aligned}
\end{equation}
which describes longitudinal motions  $y(t,x) :(0,T) \times [0,1] \to \R$ for a one-dimensional viscoelastic bar.
In this context $u = y_x > 0$ is the longitudinal strain, $v= y_t$ the velocity, and the stress 
$$
S = \sigma (y_x) +    \frac{\mu }{y_x} y_{t x}
$$ 
has an elastic and a viscous component. The viscosity in Lagrangian coordinates  is expressed as $\frac{\mu}{u}$.
We take $\mu = 1$ and the function $\sigma (u)$ is smooth and  non-monotone. Such models have
been used for modeling phase transitions.

Assume two positive states $0 < a < b$ are fixed and suppose the stress function $\sigma (u)$
satisfies
\begin{equation}
\label{condnonm}
 \sigma (\tau  a) = \sigma ( \tau  b) \qquad \mbox{ for $\tau \in [1,2]$} \, .
\end{equation}
Clearly, \eqref{condnonm} requires that $\sigma(u)$ is nonmonotone.
%\begin{figure}[htbp]
%\begin{center}
%{\sc  Insert FIGURE  gr1.pdf}
%%\includegraphics[scale=0.3]{gr1.pdf}
%\caption{The block parts of the graph satisfy \eqref{condnonm} for $1\le t\le2$; the dotted points interpolate between these parts.}
%\label{figgraph}
%\end{center}
%\end{figure}
The example detailed below is based on two properties of the system \eqref{longonev}:
\begin{itemize}
\item[(i)]  It admits special solutions of the form
\begin{equation}\label{unifshear}
\bar u (t) = \kappa t \, , \quad \bar v (x) = \kappa x
\end{equation}
where $\kappa > 0$ is a rate of stretching. These solutions are universal in the sense that they satisfy
\eqref{longonev} for any function $\sigma (u)$.

\item[(ii)]  \eqref{longonev} admits piecewise smooth solutions that are 
continuous in $v$ but discontinuous in $u$ and $v_x$, see \cite{Hoff86}, provided the Rankine-Hugoniot conditions 
$$
\begin{aligned}
-s [u] &= [v]
\\
- s [v] &= \left [ \sigma (u) +\frac{v_x}{u} \right ]
\end{aligned}
$$
are satisfied, where $s$ is the shock speed and $[q] = q_+ - q_-$ the jump of the quantity $q$. As $v$ is continuous, 
the shocks are stationary $s=0$ and $[u] \ne 0$ has to satisfy
\begin{equation}\label{sec7RH}
s= 0 \, , \quad \left [ \sigma (u) + \frac{u_t }{u} \right ] = 0 \, .
\end{equation}
\end{itemize}

A family of solutions to \eqref{longonev} is constructed defined for $(t,x) \in [1,2]\times \R$ with $u$ and $v_x$ periodic. Fix states $a, b$ such that
$0 < a < 2a < b < 2b$ and $\sigma(u)$ so that \eqref{condnonm} is satisfied. We denote by $S(t)$ the common value
\begin{equation*}
S(t) := \frac{1}{t a} a  + \sigma (t a) = \frac{1}{t b} b  + \sigma (t b)  \, ,  \qquad 1 \le t \le 2 \, .
\end{equation*}
For $0 < \theta < 1$ define the periodic function
\begin{equation}\label{perU}
U(t,x)  := \begin{cases}
 a \,  t   &  \quad  \; \;  k <  x <  k + \theta
 \\
 b \,  t   &  k + \theta < x < k + 1
 \end{cases}
\quad \; k \in \mathbb{Z}  
\end{equation}
Next, let $c_\theta = \theta a + (1-\theta) b$ and set
\begin{equation}
\label{perY}
\begin{aligned}
Y(t,x) = \int_0^x U (t, y) dy &= 
\begin{cases}
k c_\theta t + (x-k) a t  & \quad  \; \;  k <  x <  k + \theta
\\
 k c_\theta t + \theta a t + \big ( x - k - \theta) b t   &  k + \theta < x < k + 1
\end{cases} \, , \quad k \in \mathbb{Z} \, ,
\\
V(t,x) = \del_t Y (t,x) &=
\begin{cases}
k c_\theta  + (x-k) a   & \quad  \; \;  k <  x <  k + \theta
\\
 k c_\theta  + \theta a  + \big ( x - k - \theta) b    &  k + \theta < x < k + 1
\end{cases} \, , \quad k \in \mathbb{Z} \, .
\end{aligned}
\end{equation}
Then $(U, V)$ is a weak solution of \eqref{longonev} on $ [1,2]\times\R$.
It satisfies the equations in a classical sense on $(1,2)\times (k, k + \theta)$, $(1,2)\times(k + \theta , k+1)$ and 
the Rankine-Hugoniot conditions \eqref{sec7RH} at the interfaces $x = k$ and  $x = k + \theta$ for $1 \le t \le 2$, $k \in \Z$.
%\begin{figure}
%\centering
%\vspace{-0.1cm}
%\includegraphics[scale=0.3]{gr1}
%\vspace{-0.1cm}
%\caption{The block parts of the graph satisfy \eqref{condnonm} for $1\le t\le2$; the dotted points interpolate between these parts}.
%\label{figgraph}
%\end{figure}

The function $Y(t,x)$ is then rescaled and restricted to the interval $(t,x) \in  Q = (1,2)\times (-1,1)$ to define
\begin{equation}
\label{exsoln3}
y_n (t,x) = \frac{1}{n} Y(t, nx)  \, \quad u_n = \del_x y_n = U  (t, nx)  \, \quad {v_n} = \del_t y_n  =  \frac{1}{n}  V (t, nx)
\end{equation}
One checks that $(u_n , v_n )$ is a weak solution of \eqref{longonev} which is stationary for the momentum equation.
Moreover, one easily computes the limits
$$
\begin{aligned}
u_n  &\rightharpoonup  (a\theta + b(1-\theta) ) t  \quad \mbox{ weakly-$\star$ in  $L^\infty \big ( Q \big )$ }
\\
\del_x v_n &\rightharpoonup  (a\theta + b(1-\theta) )  \quad \mbox{ weakly-$\star$ in  $L^\infty \big ( Q \big )$ }
\\
v_n &\to (a\theta + b(1-\theta)) x  \quad \mbox{ strongly in $L^2 \big ( Q \big )$ }
%\\
%&\sigma(u_n) +  \frac{1}{u_n} \partial_x v_n  = S(t)
\end{aligned}
$$
and 
$$
\sigma (u_n) \rightharpoonup  \theta \sigma ( at) + (1-\theta) \sigma(bt) \ne \sigma \big (  \theta a t + (1-\theta) b t) \, ,
$$
weak-$\star$ in $L^\infty(Q)$. The oscillations in solutions are characterized by the Young measure 
$$\nu = \theta \delta_{at} + (1-\theta) \delta_{b t}$$
and are induced by oscillations in the initial data $u_n (1,x)$ and $\del_x v_n (1, x)$.

\begin{remark}\label{rmk:oscshear}
For the system describing one-dimensional shear motions,
\begin{equation}
\label{onedvisco}
\begin{aligned}
u_t &= v_x
\\
v_t &= \sigma(u)_x + v_{xx} \, ,
\end{aligned}
\end{equation}
for a viscoelastic material of rate type, a similar periodic solution can be constructed.
This is achieved provided the function $\sigma (u)$ satisfies the condition  
\begin{equation}
\label{condnonmv}
a + \sigma (t a) = b + \sigma (t b) \, , 
\end{equation}
for two states $a < b$ and for $t \in [1,2]$. 
Here,  the states $a,b$ are fixed and $t$ is thought as a parameter. 
The condition \eqref{condnonmv} restricts the form of $\sigma(u)$. We give an example
to show that it can be fulfilled. For $a, b >0$ fixed so that $0 < a < 2a < b < 2b$
suppose the graph of $\sigma (u)$ is given and is strictly increasing for $u \in (b, 2b)$.
Then \eqref{condnonmv} fully determines the graph of $\sigma (u)$ for $u \in (a, 2a)$
from the graph in $(b, 2b)$.
The emerging graph is increasing in $(a,2a)$ but the full graph will be non-monotone.
Once \eqref{condnonmv} holds in an interval $t \in [1,2]$, a solution $(v,u)$  for \eqref{onedvisco} with $v$, $u_x$ periodic is constructed 
 for $t \in [1,2]$, $x \in \R$ similarly as \eqref{perY}. This gives rise to an example of solutions for \eqref{onedvisco} which exhibits sustained oscillations and whose weak limit is no longer  a solution of \eqref{onedvisco}.
%\begin{figure}[htbp]
%\begin{center}
%\includegraphics[scale=0.3]{gr1.pdf}
%\caption{The block parts of the graph satisfy \eqref{condnonm} for $1\le t\le2$; the dotted points interpolate between these parts.}
%\label{figgraph}
%\end{center}
%\end{figure}
\end{remark}

\begin{remark}
We refer to \cite[sec 4]{Tzavaras23} for a construction of periodic solutions and a discussion of
sustained oscillations for the system of das dynamics of viscous, heat conducting gases:
\begin{equation}
\label{vhcg}
\begin{aligned}
u_t - v_x &= 0
\\
v_t - \sigma(u,\theta)_x &= ( \frac{\mu}{u} v_x )_x 
\\
\big ( \tfrac{1}{2} v^2 + e(u, \theta) \big )_t - \Big ( \sigma(u,\theta) \, v \Big )_x 
&= \Big ( \frac{\mu}{u} \,  v \, v_x \Big )_x +  \Big ( \frac{\kappa}{u} \theta_x \Big )_x  \, .
\end{aligned}
\end{equation}
In this model $u$ stands for the specific volume (the inverse of the density), $v$ is the longitudinal velocity,
and $\theta$ the temperature,  while the internal energy $e$ and  the stress $\sigma$ are determined via 
appropriate constitutive relations satisfying the Maxwell relations;
in this interpretation $u > 0$,  $\theta > 0$, and the viscosity and heat conductivity $\mu, \kappa \ge 0$.  
Oscillating solutions are constructed for adiabatic gases, $\kappa = 0$, with viscosity $\mu > 0$ 
for specially designed pressure functions $\sigma(u,\theta)$ that are nonmonotone in $u$.
\end{remark}

\subsection{ Time-dependent twinning solutions in nonlinear viscoelasticity}\label{sec:twin}

In preparation, we mention a well known result from continuum mechanics. Let $\psi(x)$ be a displacement field defined on
a reference configuration $\cR$ and suppose that $\cR$ is split into two subdomains, $\cR = \cR^+ \cup \cR^-$, by a smooth 
nonsingular hypersurface $\cS$ defined by
$$
f(x) = 0  \quad \mbox{with $\nabla f \ne 0$}.
$$
Suppose that $\psi$ enjoys the regularity $\psi$ is continuous on $\cR$, $\psi \in C^1\big(\overline{\cR^-}\big)$, $\psi \in C^1\big( \overline{\cR^+}\big)$,
and the limits
$$
\lim_{ \scriptsize\begin{matrix} x\to x_0 \\ x> x_0 \end{matrix}} \nabla \psi(x) = F^+ \, , \quad \lim_{\scriptsize \begin{matrix} x\to x_0 \\ x< x_0 \end{matrix} }\nabla \psi(x) = F^- \, ,
$$
exist for $x_0 \in \cS$ and are finite. Then the deformation gradients satisfy across $\cS$ the jump conditions
$$
F^+ - F^- = a \otimes N
$$
where $N$ is the normal to the surface and $a$ is an amplitude capturing the jump of the gradient in the normal direction at $\cS$. This result is the backbone 
in the  construction of twinning solutions in elasticity, \cite{James81, BJ87}.

Consider the system of viscoelasticity of Kelvin-Voigt type,
\begin{align}\label{vekv}
&\partial_{tt}y-\dive  \Big (  \frac{\del W}{\del F} (\nabla y) + \nabla y_t  \Big ) = 0 \, .
\end{align}
It is written as a system of conservation laws expressed in coordinate form, for the quantities 
$v_i = \frac{\del y_i}{\del t}$, $F_{i \alpha} = \frac{\del y_i}{\del x_\alpha}$, by
\begin{equation}\label{vekv-coord}
\begin{aligned}
\del_t F_{i \alpha} &= \del_\alpha v_i
\\
\del_t v_i &= \del_\alpha \Big ( \frac{\del W}{\del F_{i \alpha}} (F) + \del_\alpha v_i \Big )
\\[3pt]
0 &= \del_\beta F_{i \alpha } - \del_\alpha F_{i \beta} \, .
\end{aligned}
\end{equation}
The last equation constrains $F$ to be a gradient. It is an involution, i.e. a constraint that propagates
from the initial data to solutions. 

The construction of oscillating solutions will result from two ingredients, a class of uniform shearing solutions
in several dimensions combined with conditions that guarantee joining these solutions across interfaces. This is similar in spirit as the one
dimensional example in Section \ref{sec:1dnonmon}, but the construction is more elaborate.

We start with the interface conditions.
Suppose that $\cS = \{ (t, x) : f(t,x) = 0\}$ is a space-time hypersurface that splits a space time domain $Q_T = \cR \times (0,T)$ into two parts, 
$Q_T = \cR^+ \cup \cR^-$. We assume that $\cS$ is smooth and depicts the motion of an interface $x = x(t)$ which is moving in the direction of the normal
$\nu = \frac{\nabla_x f}{| \nabla_x f |}$ with speed $\dot x = s \nu $.  Since the interface $x(t)$ moves on the surface, we have $f(x(t), t) = 0$, which
yields for the interfacial speed that $s = - \frac{f_t}{| \nabla_x f|}$.  Suppose a solution $(v, F)$ of \eqref{vekv} 
that has smoothness $v \in C(Q_T)$, and such that $F, \nabla v \in  C \big(\overline{\cR^-}\big)$,  $F, \nabla v \in C \big( \overline{\cR^+}\big)$
and the right and left limits of $F, \nabla v$ exist and are finite on both sides of the space-time surface $\cS$.  The limits will satisfy the Rankine-Hugoniot conditions
\begin{align}
- s [ F_{i \alpha}] &= \nu_\alpha [v_i]
\label{rh1}
\\
- s [v_i] &= \nu_\alpha \left [ \frac{\del W}{\del F_{i \alpha}} (F) + \del_\alpha v_i \right ]
\label{rh2}
\\[3pt]
0 &= \nu_\beta [F_{i \alpha } ] - \nu_\alpha [F_{i \beta}]
\label{rh3}
\end{align}
where $[ \cdot ]$ denotes the jump across the interface.
Condition \eqref{rh3} implies that $[F] \tau = 0 $ for any vector $\tau \perp \nu$ and thus $[F] = a \otimes \nu$. Together with \eqref{rh1} they 
lead to
\begin{equation}\label{kinemjump}
[F] = a \otimes \nu \, , \quad [v] = - s a \, .
\end{equation}
We are here interested in $[F] \ne 0$ and $[v] = 0$. 
The latter precludes $\nabla v$ from having delta masses and allows the jump to be computed via \eqref{rh2}.
This suggests to consider standing shocks $s=0$ in conformance with the considerations of the one-dimensional case studied in \cite{Hoff86}.
In summary,  $s=0$ and $[F] \ne 0$ have to satisfy
\begin{equation}\label{steadysh}
[F] = a \otimes \nu \, , \quad   \big [  T \nu \big ] = \Big [  \Big ( \frac{\del W}{\del F} (F) + \nabla v \Big ) \nu \Big ] = 0
\end{equation}
where $T =  \frac{\del W}{\del F} (F) + \nabla v $ is the total stress. Observe that \eqref{steadysh} requires that the steady interface is in equilibrium.

The second ingredient is special solutions of \eqref{vekv} of the particular form $y(t,x) = F(t) x$ with $F(t)\in \R^{d \times d} $ a time dependent matrix.
One checks that $F(t)$ satisfies $\ddot F = 0$, which implies that the solution of \eqref{vekv} must be of the form
\begin{equation}\label{usmd}
y(t,x) = ( t F_0 +  F_1) x  \, , \quad F_0, F_1 \in \R^{d \times d} \, ,
\end{equation}
where $F_0$, $F_1$ are constant matrices. They describe uniform shear and are
universal in the sense that they are independent of the form of the stored energy $W(F)$.

We use \eqref{kinemjump} and \eqref{steadysh} to construct some special solutions of \eqref{vekv-coord}, which are in turn used to construct 
oscillatory solutions for the same system.

\subsubsection{Steady solutions with jumps in the deformation gradient in elasticity} \label{sec:steadyjump}
The first solution is known from work of Ball and James \cite{BJ87} and
provides solutions with discontinuities in the deformation gradient for the equations of elasticity. The solutions are steady and solve \eqref{vekv-coord}.
Namely, let $F_-$, $ F_+$ be such that $F_+ - F_- = a \otimes \nu$ with $a \in \R^d$, $\nu \in \mathcal{S}^{d-1}$. The hyperplane $\nu \cdot x =0$ passes 
through the origin and is orthogonal to the unit vector $\nu$.  The function
\begin{equation}
y (t, x) = 
\begin{cases} 
F_-  x  & x \cdot \nu < 0 \\
F_+ x &  x \cdot \nu > 0 \\
\end{cases}
\end{equation}
satisfies $v(t,x) = 0$ and is a steady solution of 
\begin{equation}\label{vekv-coord-el}
\begin{aligned}
\del_t F_{i \alpha} &= \del_\alpha v_i
\\
\del_t v_i &= \del_\alpha \Big ( \frac{\del W}{\del F_{i \alpha}} (F) \Big )
\\[3pt]
0 &= \del_\beta F_{i \alpha } - \del_\alpha F_{i \beta}
\end{aligned}
\end{equation}
and also of \eqref{vekv-coord} provided
$$
F_+ - F_- = a \otimes \nu  \, , \qquad 
[  T \nu \big ] = \Big (  \frac{\del W}{\del F} (F_+) - \frac{\del W}{\del F} (F_-) \Big )   \nu = 0 \, .
$$
This solution  has discontinuous deformation gradient across an interface, and  is related to twinning \cite{James81}. The position of
the matrices $F_-$, $F_+$ can be interchanged and leads ato a different weak solution. Combining these together can lead to a persistent steady
oscillatory structure that is related to phase transitions, \cite{BJ87}.

\subsubsection{Dynamic solutions with jumps in deformation gradient and strain rate in viscoelasticity}\label{sec:vejump}
Next consider $F_0$ a baseline deformation gradient,  $F_- = F_0 + a \otimes \nu$, $F_+ = F_0 + b \otimes \nu$,
where $a$, $b \in \R^d$, $\nu \in \mathcal{S}^{d-1}$. Then $F_+ - F_- = (b-a) \otimes \nu$. The hyperplane $x \cdot \nu = 0$ passes through the origin.
Consider the motion
\begin{equation}\label{basicsol}
\begin{aligned}
y (t, x) &= 
\begin{cases} 
y_- (t,x)  =  t \big ( F_0 + a \otimes \nu \big ) x   & \quad  x \cdot \nu < 0 \\
y_+ (t,x)  =  t \big ( F_0 + b \otimes \nu \big ) x  &  \quad x \cdot \nu > 0 \\
\end{cases}
\\[5pt]
v (t, x) &= 
\begin{cases} 
 \big ( F_0 + a \otimes \nu \big ) x   &\quad  x \cdot \nu < 0 \\
 \big ( F_0 + b \otimes \nu \big ) x  &  \quad x \cdot \nu > 0 \\
\end{cases}
\end{aligned}
\end{equation}
Note that $y$ and $v$ are both continuous across the interface $x \cdot \nu = 0$  which is steady in time.  The deformation gradient $\nabla y$ and stretch tensor $\nabla v$
have discontinuities across that interface.  Moreover,  by virtue of \eqref{steadysh},  $y$ will be a weak solution of  \eqref{vekv-coord} on the domain 
$[1,2] \times \R^d$ provided
\begin{equation}\label{cond}
\Big (  \frac{\del W}{\del F} ( t F_- + t (b-a) \otimes \nu ) - \frac{\del W}{\del F} ( t F_-)  + (b - a) \otimes \nu \Big )   \nu = 0 \, ,
\quad \mbox{ for $t \in [1,2]$},
\tag{C}
\end{equation}
is satisfied. The reader should note that under \eqref{cond}  the role of $a$ and $b$ (or $F_-$ and $F_+$) can be interchanged 
so as to obtain a second solution with discontinuities in the deformation gradient $\nabla y$ and stretching $\nabla v$.
%   Figure of oscillating deformation gradient
%%%
\begin{figure}[htbp] 
\begin{tikzpicture}[scale=1.8]%[x=2cm,y=2cm] 
\def\ang{-10} % degrees 
\def\th{.4} % theta \in (0,1) 

\draw[->] (-2,0) -- (3.5,0) node[below, font=\normalsize] {$x_1$}; 
\draw[->] (0,-2) -- (0,3) node[right, font=\normalsize] {$x_2$}; 

\begin{scope}[rotate=\ang] 

\foreach \x in {-2,-1,...,2} { 
  \draw (\x,-1.5) -- (\x,0) node[point]{} -- (\x,2.5) 
    (\x+\th,-1.5) -- (\x+\th,0) node[point]{} -- (\x+\th,2.5); 
  \path (\x+\th/2,1) node{$tF_-$}  (\x+1/2+\th/2,1) node{$tF_+$}; 
  \path[every node/.style={below, fill=white, inner sep=2pt, outer sep=3pt, fill opacity=.85, text opacity=1.}] 
    (\x,0) node[rotate=\ang]{$\x\mathstrut$ } 
    (\x+\th,0) node[rotate=\ang]{\ifnum \x=0 $\theta\mathstrut$ \else $\x{+}\theta\mathstrut$ \fi}; 
}; 
\draw (3,-1.5) -- (3,0) node[point]{} 
  node[below, fill=white, inner sep=2pt, outer sep=3pt, fill opacity=.85, text opacity=1., rotate=\ang]{$3\mathstrut$} -- (3,2.5); 

\path[above] 
(0,2.5) node[yshift=1mm]{$x\cdot\nu=0$} 
(\th,2.5) node{$x\cdot\nu=\theta$} 
(1,2.5) node{$x\cdot\nu=1$} 
(2,2.5) node{$x\cdot\nu=2$}; 

\draw[->, very thick] (0,0) -- (1,0) node[above left]{$\nu$}; 
\draw[dashed] (-2.25,0) -- (3.5,0); 

\end{scope} 
\end{tikzpicture}  
\caption{Deformation gradient for oscillating solutions, $F_+ = F_0+a\otimes\nu$, $F_- = F_0+b\otimes\nu$.} \label{Figoscillation}
\end{figure}
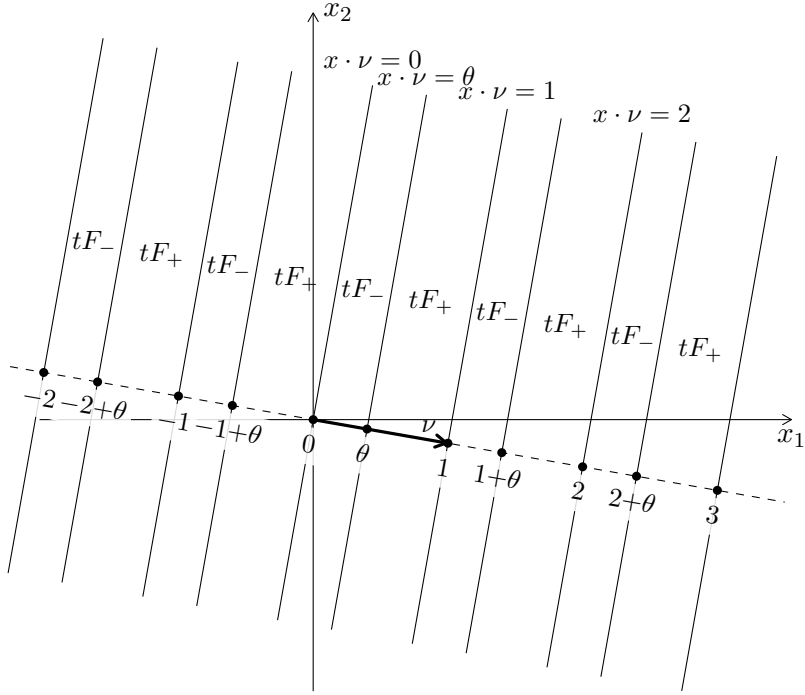 

Suppose that the condition \eqref{cond} is satisfied. By interlacing the solution \eqref{basicsol} with the corresponding solution with the position
of $F_-$ and $F_+$ interchanged, we construct an oscillating solution defined on $[1,2] \times \R^d$ that jumps across the steady interfaces 
$x \cdot \nu = k$ and $x \cdot \nu = k + \theta$, $k \in \mathbb{Z}$. The deformation gradient and stretching are given respectively by
\begin{equation}
\begin{aligned}
\nabla y (t,x) &= \begin{cases} t  ( F_0 + a \otimes \nu )   & \quad   k< x \cdot \nu < k + \theta  \\
                                              t ( F_0 + b \otimes \nu  )   &  \quad  k + \theta < x \cdot \nu < k + 1  \\
\end{cases}   \qquad k \in \mathbb{Z}
\\[5pt]
\nabla v (t,x) &= \begin{cases}  F_0 + a \otimes \nu    & \quad   k< x \cdot \nu < k + \theta  \\
                                                F_0 + b \otimes \nu   &  \quad  k + \theta < x \cdot \nu < k + 1  \\
\end{cases}   \qquad k \in \mathbb{Z}
\end{aligned}
\end{equation}
and the deformation gradient is presented in Figure \ref{Figoscillation}.

A commonly used notion in elasticity theory is rank-one convexity. Recall that $W(F)$ is rank-one convex on the domain $D$ if for any
$F \in D$ 
\begin{equation}\label{ROC}
\sum_{i,j, \alpha, \beta} \frac{\del^2 W}{\del F_{i \alpha} \del F_{j \beta}} (F) \;  \nu_\alpha \nu_\beta \; \xi_i \xi_j  > 0 \, ,
\quad   \xi \in \R^d  - \{0\} \, , \, \nu \in \mathcal{S}^{d-1}  \, .
\tag {ROC}
\end{equation}
Observe that :

\begin{lemma}
If condition \eqref{cond} is satisfied then \eqref{ROC} is violated
\end{lemma}

\begin{proof}
Indeed, using the summation convention, we express \eqref{cond} as
$$
\begin{aligned}
0 &= \Big (  \frac{\del W}{\del F_{i \alpha}} ( t F_- + t (b-a) \otimes \nu ) - \frac{\del W}{\del F_{i \alpha}} ( t F_-)  + (b - a)_i \nu_\alpha  \Big )   \nu_\alpha 
\\
&=  \left ( \int_0^1 \frac{d}{ds} \frac{\del W}{\del F_{i \alpha}} ( t F_- + s  t (b-a) \otimes \nu ) \, ds \right ) \nu_\alpha  + (b-a)_i
\\
&= t  \left ( \int_0^1  \frac{\del^2 W}{\del F_{i \alpha} \del F_{j \beta} } ( t F_- + s  t (b-a) \otimes \nu ) \, ds \right ) (b-a)_j \nu_\alpha \nu_\beta + (b-a)_i
\end{aligned}
$$
IF $W(F)$ is rank-one convex on the domain $D$ containing $t F_- + s  t (b-a) \otimes \nu$  then  we obtain $0 > |b - a|^2$ leading to a contradiction.
\end{proof}

%\vfil\eject

\section{Compressible Navier-Stokes with non-monotone pressures}\label{sec:comprns}

\subsection{The compressible Navier-Stokes system in one space dimension}

Consider the one-dimensional version of the compressible Navier-Stokes system in Eulerian coordinates:
\begin{equation}
\label{compNS}
\begin{aligned}
\rho_t  + (\rho u)_y  &= 0
\\
(\rho u)_t  + ( \rho u^2  +  p (\rho) )_y  &=  \mu u_{y y}
\end{aligned}
\end{equation}
where $\rho (t,y)$ and $u(t,y)$ are the density and velocity of the fluid expressed in Eulerian coordinates $(t,y)$ while $\mu > 0 $ is the viscosity.

Consider a solution where $u$ is continuous but $u_y$ and $\rho$ may experience jump discontinuities on a moving interface.
The Rankine-Hugoniot jump conditions across a jump discontinuity that moves with speed $s$
have the form
\begin{equation}\label{RHgd}
\begin{aligned}
\big [ \rho (u-s) \big ] &= 0
\\
\big [ \rho u (u-s) \big ] &= \big [ - p(\rho) + \mu u_y \big ]
\end{aligned}
\end{equation}
It follows that there can be solutions that are continuous in velocity $[ u ] = 0$ but discontinuous in density, $[\rho] \ne 0$, 
across the interface moving with speed $u = s$. Such solutions will satisfy the conservation of mass on the interface \eqref{RHgd}$_1$.
The conservation of momentum reduces to the condition
\begin{equation}\label{RHmom}
\big [ - p(\rho) + \mu u_y \big ] = 0
\end{equation}
Observe that since $[ u ] = 0$ across the interface the jump $[u_y]$ is well defined.

\subsubsection{Langrangian to Eulerian}
The system \eqref{longonev} is expressing the one-dimensional gas dynamics equations in Lagrangian cooardinates.
 It is now written in the notation $w$ for the strain and $v$ for the velocity
to avoid confusion with $(\rho, u)$ in \eqref{compNS} 
and it now reads
\begin{equation}
\label{longonev1}
\begin{aligned}
w_t &= v_x
\\
v_t &= \sigma(w)_x +  \del_x \left ( \frac{\mu}{w} v_{x} \right )
\end{aligned}
\end{equation}
The equivalence is effected via the transformation
$$
w = \frac{\del y}{\del x} \, , \quad v = \frac{\del y}{\del t}
$$
and setting
\begin{equation}\label{transf}
\frac{\del y}{\del t} (t,x) = v (t, x) = u (t, y(t,x)) \, , \quad \rho (t, y(t,x))  = \frac{1}{\frac{\del y}{\del x} (t,x)} = \frac{1}{w(t,x)}
\end{equation}
A direct calculation shows that for smooth solutions \eqref{longonev1} transforms to \eqref{compNS}. 
The same result can be achieved by a transformation of the weak forms provided that $y(t,x)$ is assumed to be a bi-Lipschitz homeomorphism, 
\cite{b-dafermos16}.

\subsubsection{Transfering the periodic solution \eqref{perY} in Lagrangian coordinates to a periodic solution in Eulerian coordinates}

Consider the special periodic solution of \eqref{longonev} that has the form \eqref{perU}, \eqref{perY}. Recall, that $a< 2a< b$ and that
for this solution
\begin{equation}\label{eqny}
\begin{aligned}
y(t,x) &= 
\begin{cases}
k v_0(\theta)  t + (x-k) a t  & \quad  \; \;  k <  x <  k + \theta
\\
 k v_0(\theta)  t + \theta a t + \big ( x - k - \theta) b t   &  k + \theta < x < k + 1
\end{cases} \, , \quad k \in \mathbb{Z} \, ,
\\
v(t,x)  &=
\begin{cases}
k v_0(\theta)  + (x-k) a   & \quad  \; \;  k <  x <  k + \theta
\\
 k v_0(\theta)  + \theta a  + \big ( x - k - \theta) b    &  k + \theta < x < k + 1
\end{cases} \, , \quad k \in \mathbb{Z} \, .
\end{aligned}
\end{equation}
where $v_0 (\theta) = \theta a + (1-\theta) b$ and $0 < \theta < 1$. Also, recall that
\begin{equation}\label{eqnw}
\begin{aligned}
w(t,x) &= 
\begin{cases}
t a & \quad  \; \;  k <  x <  k + \theta
\\
 t b  &  k + \theta < x < k + 1
\end{cases} \, , \quad k \in \mathbb{Z} \, ,
\end{aligned}
\end{equation}

When we transform \eqref{eqny} and \eqref{eqnw} to Eulerian coordinates via the transformation \eqref{transf}
we obtain the following Eulerian form
\begin{align}
u(t,y) &= \frac{y}{t}
\label{eqnu}
\\
\rho(t,y) &= 
\begin{cases} \;  \frac{1}{ta}  \quad & \quad  0 < \frac{y}{t}  - k v_0 (\theta) <  a \theta   \\[5pt]
                      \;  \frac{1}{tb}  \quad &  \; \; a \theta < \frac{y}{t}  - k v_0 (\theta) < v_0 (\theta)
\end{cases}
\quad k \in \mathbb{Z}
\label{eqnrho}
\end{align}
The reader should note that 
$$
\bar \rho (t, y) = \frac{1}{ta} \, , \quad \bar u (t,y) = \frac{y}{t}   \, , \quad  a > 0
$$
is an exact solution of \eqref{compNS}

For \eqref{eqnu}-\eqref{eqnrho} note that : 
\\
(i) it has discontinuities in $\rho$ on the lines $\xi_k (t) = k t v_0 (\theta)$ and 
$\xi_{k, \theta} = k t v_0 (\theta) + t a \theta$. 
\\
(ii) On the line $\xi_k (t)$ we have
$$
\frac{d \xi_k}{dt} = k v_0(\theta) = u (t,y) \Big |_{y = \xi_k (t)}
$$
(iii) On the line $\xi_{k, \theta} (t)$ we have
$$
\frac{d \xi_k}{dt} = k v_0(\theta)  + a \theta = u (t,y) \Big |_{y = \xi_{k, \theta} (t)}
$$
(iv) Let us denote 
$$I_k = \{ (t, y) : 0 < y - k t v_0 (\theta) < t a \theta \} \, ,  \quad 
J_k = \{ (t, y) : t a \theta < y - k t v_0 (\theta) < t v_0 (\theta) \} \, .
$$ 
Then we compute on each domain $I_k$ and $J_k$ that
$$
\del_t \rho + \del_y (\rho u) = 0
$$
and 
$$
\begin{aligned}
\del_t (\rho u)  + \del_y (\rho u^2 ) &= 0 
\\
\del_y ( - p(\rho) + \mu \frac{\del u}{\del y} ) = 0
\end{aligned}
$$
that is the equations \eqref{compNS} are satisfied on each of $I_k$, $J_k$. 
\\
(v) Suppose that we impose the hypothesis on the pressure
\begin{equation}\label{hyppr}
p \left ( \frac{1}{at} \right ) = p \left ( \frac{1}{bt} \right )  \quad \mbox{for \;  $t \in [1,2]$}
\tag {AP}
\end{equation}
then on the interface $\xi_k (t) = k t v_0 (\theta)$ we have
$$
[ \rho u (u - \dot \xi_k) ] = 0 ,   [ - p (\rho) + \mu u_y  ] = 0
$$
The same relations also hold on the interface $\xi_{k,\theta} (t) = k t v_0 (\theta) + t a \theta$. We thus conclude

\begin{proposition}
Let the pressure satisfy \eqref{hyppr}. Then if $0< a < 2 a < b < 2b$ then \eqref{eqnu}-\eqref{eqnrho} define a weak solution 
of the compressible Navier-Stokes system \eqref{compNS}
defined for $y \in \R$,  $t \in [1,2]$.
\end{proposition}

For states $a$, $b$ satisfying $0< a < 2a < b < 2b$ the condition \eqref{hyppr} can be fulfilled and results 
to nonmonotone equations for the pressure. The construction produces a periodic solution $(\rho, u)$
of \eqref{compNS} defined for $(t,x) \in [1,2]\times \R$.

\subsubsection{ Sustained oscillations}
Define the sequence $(\rho_n , u_n )$ by
\begin{equation}\label{oscisol}
u_n (t, y) = \frac{1}{n} u (t, n y) \, , \quad \rho_n (t, y) = \rho (t, n y)   \qquad y \in (-1,1), \;  t \in [1,2] \, ,
\end{equation}
where $(\rho, u)$ is given by \eqref{eqnu}-\eqref{eqnrho}, $n \in \mathbb{N}$. Then we have
$$
u_n (t,y) = \frac{y}{t} 
$$
and
$$
\rho_n (t,y) = 
\begin{cases} \;  \frac{1}{ta}  \quad &  \qquad  \quad  \frac{k}{n} v_0 (\theta) < \frac{y}{t}  <  \frac{k}{n}  v_0 (\theta) + \frac{1}{n}  a \theta   \\[5pt]
                      \;  \frac{1}{tb}  \quad & \frac{k}{n} v_0 (\theta) + \frac{1}{n}  a \theta   < \frac{y}{t}   <\frac{k+ 1}{n}  v_0 (\theta) 
\end{cases}
\quad k \in \mathbb{Z}
$$

It is easy to check that $(\rho_n, u_n)$ is a weak solution of \eqref{compNS}. Indeed, since $(\rho, u)$ is a weak solution of \eqref{compNS}
on $[1,2] \times \R$, we have for $\phi \in C^\infty_c ([1,2)\times \R)$
\begin{equation}\label{weak1}
\begin{aligned}
\iint \rho(t,y) \del_t \phi (t,y) dt dy + \iint \rho(t,y) u (t,y)  \del_y \phi (t,y) dt dy + \int \rho (1,y) \phi(1,y) dy = 0 \, .
\end{aligned}
\end{equation}
Introduce the test function $\phi (t, y) = \frac{1}{n} \varphi (t, \frac{1}{n} y )$ in \eqref{weak1} and perform a change of variables $y' = \frac{1}{n} y$.
Using \eqref{oscisol} we obtain
\begin{equation}\label{newweak1}
\begin{aligned}
\iint \rho_n (t,y) \del_t \varphi (t,y) dt dy + \iint \rho_n (t,y) u_n (t,y)  \del_y \varphi (t,y) dt dy + \int \rho_n  (1,y) \varphi(1,y) dy = 0 \, ,
\end{aligned}
\end{equation}
that is $(\rho_n, u_n)$ satisfies the weak form of \eqref{compNS}$_1$. A similar calculation shows that $(\rho_n, u_n)$ also satisfies the weak form
of \eqref{compNS}$_2$.

If the sequence $(\rho_n, u_n)$ is restricted on $Q = [1,2]\times (-1,1)$ we see that
$$
\rho_n \rightharpoonup \rho = \theta \frac{1}{ta} + (1-\theta) \frac{1}{tb} \quad \mbox{ weakly-$\star$ in  $L^\infty \big ( Q \big )$ }
$$
Of course $u_n = \frac{y}{t}$ converges strongly.

\begin{theorem}
Under the hypothesis \eqref{hyppr} there is a sequence of weakly converging initial data so that the solution $(\rho_n, u_n)$ of
\eqref{compNS} has the property that $\rho_n \rightharpoonup \rho$ weakly while $u_n \to u$ strongly.
\end{theorem}

Note that $(\rho_n, u_n)$ are also weak solutions of the pressureless Euler equations, and 
the assumption \eqref{hyppr} guarantees that the total stress is a function of time.

\subsection{Compressible Navier-Stokes in multi-d}\label{subseq:CNS}

Consider next the compressible Navier-Stokes system in several space dimensions
\begin{equation}
\label{compns}
\begin{aligned}
\del_t \rho + \dive \rho u &= 0
\\
\del_t \rho u + \dive \rho u \otimes u  + \nabla p (\rho) &=  \dive \Big ( \mu (\nabla u + \nabla u^T) + \lambda (\dive u )  \, I \, \Big )
\end{aligned}
\end{equation}
where $\lambda$, $\mu$ are constants with $3 \lambda + 2 \mu > 0$, $\mu > 0$. 
For the purposes of this section the pressure $p(\rho)$ will be nonmonotone.

\subsubsection{ Solutions with discontinuous density and continuous velocity fields}
We write the Rankine-Hugoniot conditions for a shock solution across a jump discontinuity
with normal vector $n$ which  moves with speed $s$; they are
$$
\begin{aligned}
 [ \rho ( u \cdot n - s) ] &= 0
\\
[ \rho u_i ( u \cdot n - s ) ] &= - [p(\rho)] n_i  + \Big [ \mu  \frac{\del u_i}{\del {x_j}} n_j +   \mu \frac{\del u_j}{\del {x_i}} n_j  \Big ]   + \lambda [div u] n_i
\end{aligned}
$$
where the summation convention over repeated indices is used.

We look for solutions where $u$ is continuous across the interface. Proceeding as in Section \ref{sec:twin}, the continuity of $u$ implies that 
$\Big [\frac{\del u_i}{\del {x_j}} \Big ]  n_j =  \Big [ \frac{\del u_j}{\del {x_i}} \Big ] n_j $. In turn, this  implies $[\nabla u] = a \otimes n$
and $[\nabla u] n = [\nabla u^T] n$, where $a$ is the amplitude and $n$ the normal.  We deduce 
$$
[\nabla u] n  = (a \cdot n) n = [\div u] n
$$

We conclude that solution that $u$ is continuous and $[\rho] \ne 0$ on an interface moving with speed s has to satisfy on the interface
$$
\begin{aligned}
 &[ \rho ( u \cdot n - s) ] = 0
\\
0 &= [ - p(\rho) + (2 \mu + \lambda ) \div u ]  \, n_i  
\end{aligned}
$$
The scalar quantity $z = - p(\rho) + (2 \mu + \lambda ) \div u$ is called effective viscous flux and has to be continuous across this kind of interfaces
(with $u$ continuous).
It was identified by Hoff \cite{Hoff95} and plays  an important role in the existence theory of the compressible Navier-Stokes system.

\subsubsection{Periodic solutions in multi-dimensions}

Next, we construct a periodic solution and an associated weakly convergent sequence of
weak solutions to \eqref{compns}.
Consider first the pressureless Euler equations in $d$ dimensions,
\begin{equation}
\label{pleuler}
\begin{aligned}
\del_t \rho + \dive \rho u &= 0
\\
\del_t \rho u + \dive \rho u \otimes u   &=  0  \, .
\end{aligned}
\end{equation}
A direct computation shows that for $\rho_0 > 0$ constant, the function
\begin{equation}\label{unif}
\hat \rho = \frac{\rho_0}{t^d} \, , \quad \hat u = \frac{y}{t}  \, , \qquad y \in \mathbb{R}^d \, , \; t > 0 \, ,
\end{equation}
is an exact solution of \eqref{pleuler}. Indeed,
$$
\begin{aligned}
\left ( \frac{\del}{\del t} + \hat u_j \frac{\del}{\del y_j} \right ) \frac{y_i}{t} &= 0
\\[5pt]
\del_t \frac{\rho_0}{t^d}   + \dive \left ( \frac{\rho_0}{t^d} \frac{y}{t}  \right ) & = 0
\end{aligned}
$$

For $a,b $ positive constants and $0 < \theta < 1$ the function $(\rho, u)$ defined on the
domain $Q = [1,2] \times \R^d$ by
\begin{equation}\label{propsol}
\begin{aligned}
\rho (t,y) &= 
\begin{cases} 
\frac{a}{t^d}  \quad & \qquad \quad k t < |y| < (k + \theta) t \\[5pt]
\frac{b}{t^d}  \quad  & \; (k + \theta) t < |y| < (k + 1) t \\
\end{cases}
\; , \qquad k \in \mathbb{N}_0 = \{ 0, 1, 2, ... \} \, ,
\\[5pt]
u(t, y) &= \frac{y}{t} \, .
\end{aligned}
\end{equation}
consists of a smooth velocity field  $u$  with density field $\rho$ that has discontinuities  at interfaces as depicted in Figure \ref{fig:1}.
\begin{figure}
\begin{tikzpicture}[scale=1.1] 
\def\nn{6} 
\foreach \i in {1,3,...,\nn} { 
	\draw[red, semithick] (.45*\i,0) arc (0:360:.45*\i) node[xshift=-6]{$\frac{a}{t^2}$}; 
	\draw[semithick] ({.45*(\i+1)},0) arc (0:360:{.45*(\i+1)}) node[xshift=-6]{$\frac{b}{t^2}$}; 
}
\node[red] at (.45*\nn,0)[right] {$\cdots$}; 
\end{tikzpicture} 
\caption{ density jumps at moving interfaces - dimension $d=2$}\label{fig:1}
\end{figure}
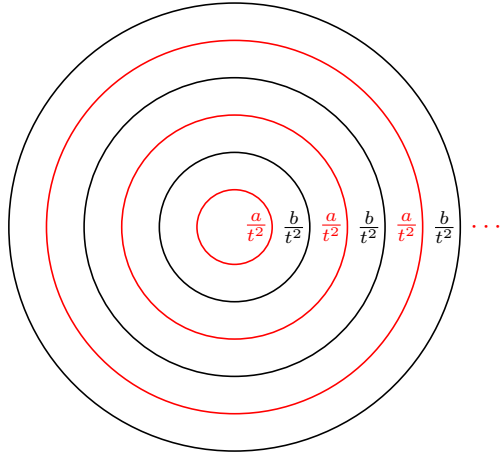 
Note that $(\rho, u)$ is an exact solution in each of the annuli $k t < |y| < (k + \theta) t$ and $(k + \theta) t < |y| < (k + 1) t$, $k \in \mathbb{N}_0$.

The discontinuities move on the shock surfaces $S_k$ defined by $\xi_k (t) = k t \hat y$, $\hat y  = \frac{y}{|y|} \in \mathcal{S}^{d-1}$ 
and on the shock surfaces $S_{k + \theta}$ defined by $\xi_{k + \theta}  (t) = (k+\theta)  t \hat y$, $\hat y  = \frac{y}{|y|} \in \mathcal{S}^{d-1}$.
On the shock surface $S_k$ the Rankine-Hygoniot conditions are satisfied:
$$
\begin{aligned}
-s [\rho] + \nu_j [\rho u_j ] &= (\nu_j u_j - s ) [\rho] = (\hat y \cdot u |_{S_k} -  |\dot \xi_k | )   [\rho] = 0
\\
-s [\rho u_i ] + \nu_j [\rho u_i u_j ] &= u_i  (\hat y \cdot u |_{S_k} -  |\dot \xi_k | )  [\rho] = 0 \, .
\end{aligned}
$$
A similar calculation shows that the Rankine-Hugoniot conditions are also satisfied on $S_{k + \theta}$. We conclude that $(\rho, u)$
is a weak solution of \eqref{pleuler} on the domain $\{ t >0 \} \times \R^d$.

Consider next the compressible Navier-Stokes system \eqref{compns}.
Observe that $(\hat \rho, \hat u)$ in \eqref{unif} is still
an exact solution of \eqref{compns} for any pressure function $p (\rho)$. Consider $(\rho, u)$ defined by \eqref{propsol}
and test under what conditions it is a weak solution on the domain $(1,2) \times \R^d$. We examine the Rankine-Hugoniot conditions on
the surfaces $S_k$ and $S_{k+\theta}$. The balance of mass is automatically satisfied while  the balance of momentum reduces to
$$
n_i  \left [ - p(\rho) \delta_{i j} +  ( 2 \mu + \lambda ) \div u   \right ] = 0
$$
This reduces to the condition
\begin{equation}\label{hypprmd}
p \left ( \frac{a}{t^d} \right ) = p \left ( \frac{b}{t^d} \right )  \quad \mbox{for \;  $t \in [1,2]$}.
\tag {{H}$_{md}$}
\end{equation}
If $0 < a< b$ and we impose $a < \frac{b}{2^d}$ we see that we can construct  functions $p(\rho)$ that satisfy \eqref{hypprmd}. Such a function
$p(\rho)$ will be non-monotone. 

\begin{proposition}
Let the pressure satisfy \eqref{hypprmd} for some $0 < \frac{a}{2^d} < a < \frac{b}{2^d} < b$. Then, \eqref{propsol} is a weak solution,
defined for $y \in \R$,  $t \in [1,2]$, 
of the compressible Navier-Stokes system \eqref{compns} with non-monotone pressure.  
\end{proposition}

\subsubsection{ Sustained oscillations}.
The function \eqref{propsol} is rescaled  to provide sustained oscillations as follows:
\begin{equation}\label{oscisol2}
\begin{aligned}
u_n (t, y) &= \frac{1}{n} u (t, n y) = \frac{y}{t}  \\
 \rho_n (t, y) &= \rho (t, n y)   
\end{aligned}
\quad , \qquad   t \in [1,2] \, , \quad y \in \R^d \\[5pt]
\end{equation}
where $(\rho, u)$ is given by \eqref{propsol}, $n \in \mathbb{N}$. 

Note the direct calculations
\begin{align*}
\del_t \rho_n + \dive \rho_n u_n &=  \big ( \del_t \rho + \dive \rho u \big ) \Big |_{(t, ny)}
\\
\del_t  (\rho_n u_n)  + \dive ( \rho_n u_n \otimes u_n ) &=  \frac{1}{n} \big ( \del_t  (\rho u)  + \dive ( \rho u \otimes u \big ) \Big |_{(t, ny)}
\\
 \nabla p(\rho_n) + \dive \Big ( \mu (\nabla u_n + \nabla u_n^T) + \lambda (\dive u_n )  \, I \, \Big ) &= n  \left [  \nabla p(\rho) 
 + \dive \Big ( \mu (\nabla u + \nabla u^T) + \lambda (\dive u )  \, I \, \Big ) \right ] \Big |_{(t, ny)}
 \end{align*}
 The equation \eqref{compns} is not invariant under the scaling \eqref{oscisol}, nevertheless because $(\rho, u)$ in \eqref{propsol} satisfies
 simultaneously both the pressureless Euler equation \eqref{pleuler} and the steady equation
 $$
  \nabla p(\rho) + \dive \Big ( \mu (\nabla u + \nabla u^T) + \lambda (\dive u )  \, I \, \Big ) = 0 \, ,
 $$
 it follows that $(\rho_n, u_n)$ in \eqref{oscisol2} satisfies the compressible Navier-Stokes system.
 
 When $(\rho_n, u_n)$ is restricted to  $Q = [1,2] \times B_1$ it provides $u_n = \frac{y}{t}$ and
 \begin{equation}\label{propsol2}
\begin{aligned}
\rho_n (t,y) &= 
\begin{cases} 
\frac{a}{t^d}  \quad &  \qquad \frac{k}{n}   < \frac{|y|}{t}  < \frac{k + \theta}{n}   \\[5pt]
\frac{b}{t^d}  \quad & \quad  \frac{k + \theta}{n} < \frac{|y|}{t}  < \frac{k+1}{n}  \\
\end{cases}
\; , \qquad   k = 0, 1, ... , n-1  \, .
\end{aligned}
\end{equation}
$(\rho_n, u_n)$ are weak solutions of \eqref{compns} with the properties $u_n \to u = \frac{y}{t}$ strongly while
$$
\rho_n \rightharpoonup \rho = \theta \frac{a}{t^d} + (1-\theta) \frac{b}{t^d} \quad \mbox{ weakly-$\star$ in  $L^\infty \big ( Q \big )$ }
$$
 The Young measure associated with the sequence $\rho_n$ is 
 $$
 \nu = \theta \delta_{\frac{a}{t^d}} + (1-\theta) \delta_{\frac{b}{t^d}}
 $$
 The sequence $\{ \rho_n\}$ exhibits persistent oscillations induced  by oscillations in the initial data.
 While $(\rho, u)$ is a solution of compressible Navier-Stokes system, the reader should note that 
 $\{ \rho_n \}$ has the property that $p(\rho_n)  \rightharpoonup q$ with  in general $q  \ne p(\rho)$.

%\vfil\eject

\section{Homogenization for scalar conservation laws: Generalized kinetic solutions}\label{sec:kineticcl}

Physical phenomena that involve sustained oscillations are commonplace, with examples appearing in problems of phase transitions, 
in multi-phase fluid flows, and are quite widespread in flows associated with fluid instabilities or with turbulence.
The existence of crystals with microstructure was long known experimentally and was given an explanation via 
minimizing sequences starting from the seminal work of Ball-James \cite{BJ87}. It has been conjectured that microstructure
may appear in dynamical problems as at long times the dynamic behavior can mimic that of minimizing sequences;
 see \cite{BHJPS91,SH92,Sengul10} for interesting studies of the dynamical appearance of microstructure in phase transitions,
 and \cite{Theil98} for the dynamics of Young measures describing microstructure.

In previous sections, we gave explicit examples of oscillating solutions for viscoelastic systems describing phase transitions
and for the compressible Navier-Stokes with non-monotone pressures. 
That raises the issue of devising equations describing the effective response of persistent oscillations.
The problem goes under the name nonlinear homogenization and is a difficult problem that is 
open in most contexts.  In simple cases that involve only one oscillating variable it can be addressed
and presenting such examples is our next task.

We start with a review of the generalized kinetic solutions that pertain to 
the  problem of oscillations for scalar conservation laws.
This is followed, in the next section, by effective equations for oscillating solutions in a one-dimensional hyperbolic-parabolic 
system modeling phase transitions. Finally, in Section \ref{sec:homcomNS}, we outline two approaches for the calculation of effective equations 
for the homogenization problem of the compressible Navier-Stokes system with non-monotone pressures.

\subsection{Kinetic formulation of scalar conservation laws}\label{sec:kineticformcl}
Consider the scalar conservation law for simplicity in one dimension,
\begin{equation}\label{conlaw}
\del_t u + \del_x f(u) = 0 \, ,
\end{equation}
where $u(t,x)$ is a scalar-valued function and $f(\cdot)$ smooth. Smooth solutions for the Cauchy problem break down in finite time,
and for global solvability it is necessary to consider weak solutions involving shocks. It is then known that in order to guarantee 
uniqueness one needs to exploit the nonlinear structure of the problem.

For scalar conservation laws, this is achieved using the notion of entropy solutions introduced by Kruzhkov. 
We recall $\eta - q$ is an entropy - entropy flux pair if they satisfy the relation
$$
q^\prime (\xi) = f^\prime (\xi) \eta^\prime (\xi)
$$
where $\lambda(\xi) = f^\prime (\xi)$ is the wave speed.  Given $\eta - q$ any smooth solution of \eqref{conlaw}
satisfies the additional conservation laws
$$
\del_t \eta(u) + \del_x q(u) = 0 \, .
$$

Entropy pairs serve to introduce the notion of Kruzhkov entropy solution, by requiring that a weak solution $u(t,x)$
satisfies the entropy inequalities
\begin{equation}\label{entropysoln}
\del_t \eta (u) + \del_x q(u) \le 0 \quad \mbox{$\forall$ entropy pairs $\eta - q$ with $\eta (u)$ convex}.
\end{equation}
This notion is sufficient to guarantee uniqueness of solutions for scalar conservation laws in various contexts, {\it e.g.} \cite{b-dafermos16}.

The kinetic formulation for conservation laws provides an equivalent notion of solutions introduced by Lions-Perthame-Tadmor \cite{LPT94a}. 
The notion is based on representation of the entropy pairs $\eta - q$ via the formulas
\begin{align}
\eta (u) - \eta(0) &=  \int_0^u \eta^\prime (\xi) d\xi
\nonumber
\\
&= \begin{cases}  
        \; \; \int_\R \charf_{0 < \xi < u} \eta^\prime (\xi) d\xi  & u > 0 
        \\
        \int_\R - \charf_{u < \xi < 0} \eta^\prime (\xi) d\xi  & u < 0 
      \end{cases}
\nonumber
\\
&= \int_\R \chi( \xi, u) \eta^\prime (\xi) d\xi 
\label{repformulaeta}
\\
q(u) - q(0) &= \int_\R \lambda (\xi) \chi( \xi, u) \eta^\prime (\xi) d\xi 
\label{repformulaq}
\end{align}
The "kinetic function" $\chi(\xi, u)$ 
\begin{equation}\label{kinfunction}
\chi(\xi, u) = 
\begin{cases} 
\; \; \charf_{0  < \xi < u}  & \mbox{for} \; u > 0 \\
- \charf_{u < \xi < 0}  & \mbox{for} \;  u < 0 \\
\end{cases}
\end{equation}
serves as the kernel representing all entropy pairs.

\begin{theorem}\cite{LPT94a}
Let $u(t,x) \in L^\infty ( (0,T) \times \Tone)$. Then $u$ satisfies \eqref{entropysoln} if and only if there is a positive, bounded measure $m$
such that
\begin{equation}\label{kinform}
\del_t \chi (\xi, u) + \lambda(\xi) \del_x \chi (\xi, u) = \del_\xi m
\end{equation}
in distributions.
\end{theorem}

\begin{proof}
We only show that the Kruzhkov entropy solution imples \eqref{kinform}. Using the representation formulas
for the entropies one may express
$$
\del_t \eta (u) + \del_x q(u) = \int_\R \Big ( \del_t \chi (\xi, u) + \lambda(\xi) \del_x \chi (\xi, u) \Big ) \eta^\prime (\xi) d\xi  \le 0
$$
for any entropy pair $\eta - q$ with $\eta$ convex. If we define $\del_\xi m := \del_t \chi (\xi, u) + \lambda(\xi) \del_x \chi (\xi, u) $
then we have
$$
\langle \del_\xi m , \eta^\prime (\xi) \rangle = - \langle m , \eta^{\prime \prime} (\xi) \rangle \le 0 \quad \forall \; \eta^{\prime \prime} \ge 0
$$
This means that the distribution $m$ is positive and thus a positive measure.
\end{proof}

The notion of the kinetic formulation has been extended to the compressible Euler system in one space dimension
\cite{LPT94} as well as for the system of one-dimensional elasticity \cite{PT00}. The main tool is always a representation formula
for the convex entropies.  For systems, the kinetic function is associated to the Riemann function for the equation generating the entropies,
and the system that produces the kinetic formulation generally involves the presence of macroscopic moments.

\subsection{The emergence of generalized kinetic solutions from viscous limits of scalar conservation laws}\label{sec:genkinsoln}
The weak limits of kinetic functions  can be used as a tool for deriving effective equations describing oscillations, see
\cite{PT00}, \cite[Ch 4]{b-Perthame02}. We outline the method
by introducing a kinetic description for viscous limits following the approach in \cite{HT02,CP03}.

Let $\{ u^n \}$ be a sequence of approximate solutions to the initial value problem for a scalar conservation law
in one-space dimension,
\begin{equation}\label{ivpscalarvcl}
\begin{aligned}
\del_t u^n + \del_x f(u^n) &= \eps_n \del^2_x u^n
\\
u^n (x,0) &= u_0^n (x) \, ,
\end{aligned}
\end{equation}
where  $u^n : (0,T) \times \Tone \to \R$ are scalar valued
and periodic and $\eps_n > 0$ is a small parameter converging to zero.
The entropy pairs $\eta(u) - q(u)$ associated to the scalar conservation law 
satisfy the entropy equation 
\begin{equation}
\label{entropystr}
\del_t \eta(u^n) + \del_x q(u^n) = \del^2_x ( \eps_n \eta(u^n) ) - \eps_n \eta^{\prime\prime} (u^n) (\del_x u^n )^2 \, .
\end{equation}

The solutions $\{ u^n\}$ of the initial value problem \eqref{ivpscalarvcl} satisfy the a-priori bounds
\begin{align}
\label{linfbound}
\sup_{x \in \Tone} |u^n (x,t) | &\le \sup_{x \in \Tone} |u^n_0  (x) | 
\\
\label{enecl}
\int_{\Tone}  \tfrac{1}{2} |u^n|^2 (x,t) dx + \eps_n \int_0^t \int_{\Tone}  &|\del_x u^n |^2 \, dx d\tau \le \int_{\Tone} \tfrac{1}{2} |u_0^n|^2 (x) dx \, .
\end{align}
We assume the initial data are uniformly bounded  in $L^\infty (\Tone)$ and seek to derive an equation that describes
the effective response of the sequence of oscillating solutions as $\eps_n$ tends to zero. (Note, we have not assumed any genuine 
nonlinearity properties for $f(u)$.)

For simplicity of the exposition we assume the initial data are positive, $u_0^n (x) \ge 0$. Then the sequence of solutions $\{u^n\}$
satisfies that $u^n$ is nonnegative and bounded in $L^\infty$,
\begin{equation}
\label{unifbound}
0 \le  u^n (t,x) \le M := sup_n \| u^n_0 \|_\infty
\end{equation}
There exists a subsequence (again denoted by $\{ u^n \}$) and functions $u(t,x) \in L^\infty$,  $F(t,x,\xi) \in L^\infty$ such that
\begin{equation}\label{wkproperties}
\begin{aligned}
u^n &\rightharpoonup u \quad \mbox{wk-$\ast$ in $L^\infty_{t,x} $}
\\
\charf_{u^n < \xi} &\rightharpoonup F \quad \mbox{wk-$\ast$ in $L^\infty_{t,x, \xi}$} \, .
\end{aligned}
\end{equation}
The representation theory of weak limits via Young measures \cite{Tartar79,Ball89} implies that there exists a parametrized family of probability measures 
$\nu_{t,x} (\lambda)$ that along a further subsequence represents the weak limits
\begin{equation}\label{def1YM}
g(u^n) \rightharpoonup \langle \nu_{t,x} , g (\lambda) \rangle  = \int g (\lambda) \, d\nu_{t,x} (\lambda)  \qquad \mbox{wk-$\ast$  in $L^\infty (Q_T)$}
\end{equation}
for any $g (u)$ continuous. By \eqref{unifbound}, the support of the Young measure verifies $\supp \nu_{t,x} \subset [0,  M]$.

It is instructive to explore the relation between $F$ defined in \eqref{wkproperties} and the Young measure $\nu$. 
Let $\theta \in C_c (\R)$ and consider the class of entropies $g(\xi)$ generated by the formula
$$
g(u) = \int_{-\infty}^u  \theta (\xi)  \, d\xi = \int_{-\infty}^\infty  \theta (\xi) (1 - \charf_{u  < \xi } ) \, d\xi  
$$
Then $g$ is continuous and for some $m > 0$ its support $\supp g \subset [-m, \infty)$ and $g(u) = const$ for $u > m$.
Using \eqref{wkproperties},
\begin{equation}
\label{repf1YM}
g(u^n) = \int_\R  \theta (\xi) (1 - \charf_{u^n  < \xi } ) \, d\xi   \rightharpoonup \int_\R \theta (\xi) (1 - F(t,x,\xi) ) \, d\xi 
\end{equation}
weak-$\ast$ in $L^\infty$ for $\theta\in C_c (\R)$.

By comparing \eqref{def1YM} and \eqref{repf1YM} we deduce
$$
\int g (\lambda) \, d\nu_{t,x} (\lambda) = \int_\R g^\prime (\xi) (1 - F(t,x,\xi) ) \, d\xi 
$$
for functions $g \in C^1 (\R)$ with $\supp g \subset [-m. \infty)$ and $g(\xi) = const $ for $\xi \in [m, \infty)$. If we interpret $F (\cdot, \xi)$ to be the 
right continuous representative then $F(\xi) = \nu_{t,x} ( (-\infty, \xi] )$ and the formula can be deduced directly 
 by the integration by parts formula \cite[Thm 3.36]{b-Folland99}. Therefore $F$  can be thought of as the distribution function of the Young measure.

Next, the entropy dissipation is expressed via a single equation as follows. Consider entropies $\eta \in C^1_c (\R)$ with compact support
and interpret $\eta^\prime$ as a test function. Using the representation formulas, we have
$$
\begin{aligned}
\eta (u^n) &= \int_\R \eta^\prime (\xi)  (1 - \charf_{u^n < \xi} ) d\xi
\\
q(u^n) &= \int_\R \eta^\prime(\xi) \lambda (\xi)  (1 - \charf_{u^n < \xi} ) d\xi
\end{aligned}
$$
Observe that formally we may write
$$
\begin{aligned}
\eta^{\prime\prime} (u^n) \eps_n (\del_x u^n)^2 &= \int_\R \eta^{\prime\prime}(\xi) \eps_n (\del_x u^n)^2 \delta (\xi - u^n) d\xi
%\\
&= - \left\langle \eta^\prime (\xi) , \del_\xi \left ( \eps_n (\del_x u^n)^2 \delta (\xi - u^n) \right ) \right\rangle
\end{aligned}
$$

To make this rigorous, note the  estimate \eqref{enecl} implies the quantity $G^{{\eps_n}}  = \eps_n (\del_x u^n)^2  \in_b L^1_{t,x}$ is uniformly bounded in $L^1_{t,x}$.
We wish to define  $\delta(\xi - u^n ) G^{{\eps_n}}$ as a distribution in $\disxtx$ by its action on tensor products
\begin{equation}
\label{dist:basic}
\langle \delta(\xi - u^n ) G^{{\eps_n}} ,\varphi \otimes \eta' \rangle =
\int_{x,t}G^{{\eps_n}}(x,t) \varphi(x,t) \eta'(u^n (x,t)) dx dt 
\, .
\end{equation}
This follows from the Schwartz kernel Theorem: Define the linear map
$$
\mathcal{K} : C_{c}^{\infty} (\RR) \to \dis ( \Rd \times \RR^{+})
\quad \text{by} \quad \mathcal{K} \psi = G^{{\eps_n}} (x,t) \psi (u^{\eps}(x,t))
$$
If $\psi_{j} \to 0$ in $C_{c}^{\infty} (\RR)$ then 
$\mathcal{K} \psi_{j} \to 0$ in $\disxt$. The kernel Theorem implies
that $\delta(\xi - u^n ) G^{{\eps_n}}$ is well defined as a distribution
in $\disxtx$ and acts on tensor products via \eqref{dist:basic}.
Moreover,
\begin{equation}
\langle \del_{\xi}\delta(\xi - u^n ) G^{{\eps_n}} , \varphi \otimes \eta' \rangle =
-\int_{x,t}G^{{\eps_n}} (x,t) \varphi(x,t) \eta''(u^{n}(x,t)) dx dt
\, . 
\end{equation}

Thus \eqref{entropystr} is written as
\begin{eqnarray}
\label{relax:dist}
\lefteqn{ 
 \Big  \langle  \del_{t} (1 - \charf_{u^n < \xi} ) +  \lambda(\xi)  \del_x (1 - \charf_{u^n < \xi} )  - \eps_n \del^2_x (1 - \charf_{u^n < \xi} )  \, , \,
\eta'(\xi)\varphi(x,t) \Big \rangle }
                       \nonumber
\\  
& &  \qquad \qquad \qquad  \qquad   = \Big \langle \del_{\xi} \big (\delta(\xi - u^n ) G^{{\eps_n}} \big ) \, , \, \eta'(\xi)\varphi(x,t) \Big \rangle
\nonumber.
\end{eqnarray}
Since the subspace generated by the direct sum test functions
$\varphi \otimes \eta'$ is dense in 
$ C^{\infty}_{c}(\RR^{+} \times \T \times \RR)$, 
the bracket (\ref{relax:dist}) is extended to test functions
$ \theta(t,x,\xi) \in C^{\infty}_{c}(\RR^{+} \times \T \times \RR)$. So, the equation \eqref{entropystr} is interpreted via duality as an identity
in $\mathcal{D}^\prime_{t,x,\xi}$ of the form
\begin{equation}\label{genkineq}
\del_t (1 - \charf_{u^n < \xi} ) + \del_x \lambda(\xi) (1 - \charf_{u^n < \xi} ) = \eps_n \del^2_x (1 - \charf_{u^n < \xi} ) 
+ \del_\xi m^n
\end{equation}
where we have set $m^n = \eps_n (\del_x u^n)^2 \delta (\xi - u^n) $.  Now $m^n$ is a positive measure and in view of \eqref{enecl}
$m^n$ is uniformly bounded in measures and along a subsequence
$$
m^n \to m \qquad \mbox{weak-$\ast$ in measures}
$$
where $m$ is a positive bounded measure.
Passing to the limit, we conclude that $F(t,x,\xi)$ satisfies in the sense of distributions the initial value problem
\begin{equation}
\label{effeqcl}
\begin{aligned}
\del_t (1 - F) + \del_x \lambda (\xi) ( 1 - F) &= \del_\xi m
\\
F(0,x,\xi) &= \mbox{wk-$\ast$ lim}  \charf_{u_0^n (x) < \xi} \, .
\end{aligned}
\end{equation}

\begin{remark}
The equation \eqref{effeqcl} should not be confused with \eqref{kinform} the kinetic formulation for a scalar conservation law. 
The latter is an equation whose solution is constrained to be of the form $\chi (\xi, u)$ in \eqref{kinfunction}.
For positive solutions this reduces to $\chi(\xi, u) = \charf_{0 < \xi < u}$ . In the effective equation 
\eqref{effeqcl} the function $F(\xi)$ is nondecreasing with $0 \le F \le 1$; it is the distribution function of the Young measure
associated to a measure valued entropy solution. 
Perthame calls these solutions generalized kinetic solutions and used this concept to provide a proof of the DiPerna's uniqueness theorem for
entropic measure-valued solutions, \cite[Thm 4.3.1]{b-Perthame02}.
\end{remark}

%\vfil\eject

\subsection{ Representation of weak limits for nonlinear functions via the distribution function of the Young measure}\label{sec:repweaklim}

The Young measures (or parametrized measures) are the standard tool used in the representation of weak limits.
Their theory is traced back to an idea of L.C. Young that was systematized in a functional analysis framework in the
pioneering work of L. Tartar \cite{Tartar79}.  The reader is referred to Ball \cite{Ball89} and Perdregal \cite{b-Pedregal97} for systematic
expositions on the subject.

Borel measures on $\R$ are in correspondence with their associated distribution functions, {\it e.g.} \cite[Sec 1.5]{b-Folland99}. 
Due to this property, oscillations of  real-valued sequences $\{ u^n \}$ can be described using the distribution function of the Young measure.
It leads to a simpler object for the representation of weak limits, which  avoids the difficulties encountered in \cite{Ball89} 
for dealing with weak topologies on measures and is capable to describe the dynamics of oscillations.

In this subsection we review some of the technical aspects expanding on material from \cite{b-Perthame02,PT00}, 
Consider a sequence $\{ u^{n} \}$ weakly compact in $L^{1}_{loc} (Q_T)$ with $Q_T = (0,T)\times \R^d$ and fix $K$ 
a compact subset of $Q_T$. We recall that  a uniformly bounded sequence $\{ u^n \}$ in $L^1(K)$, $K$ of finite measure, 
\begin{equation}\label{L1b}
\sup_{n \in \N} \int_K |u^n | \le C  \, ,
\end{equation}
is weakly compact if and only if $\{u^n \}$ is uniformly integrable, {\it i.e.}
given $\eps > 0$ there is $\delta > 0$ such that for any measurable $E \subset K$ with $|E | < \delta$ the integral
$\int_E |u^n | < \eps$ for all $n \in \N$.

The condition of uniform integrability is equivalent to requiring that for $c > 0$, 
\begin{equation}\label{wkcompact}
\sup_{n \in \N} \int_{\{(t,x) \in K :  |u^n| \ge c \} } |u^n| \to 0 
\quad 
\text{ as $c \to \infty$}
\end{equation}
or that 
$$
\begin{aligned}
\text{ for some function $\Psi : \RR^+ \to \RR^+$ } \; 
&\text{ nondecreasing and satisfying
$\frac{\Psi ( \tau)}{ \tau } \to \infty$ as $\tau \to \infty$,}
\\
&\text{$\Psi(|u^\eps|)$ is bounded in $L^{1} (K)$.}
\end{aligned}
$$
We refer to Dunford-Schwartz \cite[pp. 289-295]{b-DS67} and 
Dellacherie-Meyer \cite[pp. 21-28]{b-DM75} for properties
of the weak topology in  $L^1$ and 
the Dunford-Pettis characterization of weak $L^1$ compactness. 
The requirement of uniform integrability \eqref{wkcompact} is placed to exclude concentrations. The simplest example 
of concentrations is given by the sequence
$$
u^n (x) =
\begin{cases}
n &x \in (0, \tfrac{1}{n})
\\
0 & \mbox{otherwise}
\end{cases}
$$
which has the limit $u^n \rightharpoonup \delta_{x=0}$ weak-$\ast$ in measures. The reader is referred to \cite{b-Pedregal97}[Ch 6.3, 6.4]
for a more sophisticated example of a sequence with concentrations and the handling  of concentrations via Chacon's biting convergence.

The Young measure theorem \cite{Ball89} implies that there exists a subsequence (for simplicity again denoted by) $\{ u^{n} \}$ 
and a parametrized family of probability measures $\nu_{(t,x)}$ such that 
$$
S(u^n ) \rightharpoonup \overline{S(u)} = \int S(\lambda) d\nu_{(t,x)} (\lambda)  \quad \mbox{weakly in $L^1(K)$}
$$
for any continuous function $S$ such that $\{ S(u^n) \}$ is sequentially weakly compact in $L^1 (K)$.

Henceforth, we focus on real valued sequences $u^n : Q_T \to \R$ and functions $S : \R \to \R$ and use the correspondence 
of Borel measures with their distribution functions on $\R$, and the representation formulas 
\eqref{repformulaeta}, \eqref{repformulaq} to derive an alternate representation of weak limits.
The notation $\chi(\xi, u)$ in \eqref{kinfunction} is used for the kinetic function and, 
after extraction of  subsequences, there is 
$u \in L^{1}_{\text loc} $ and $g(t,x,\xi)\in L^\infty$ that satisfy 
\begin{align}
\label{seno10}
 u^n &\rightharpoonup u, \quad wk-L^1_{\text loc}
\\
\label{seno11}
\chi (\xi, u^n) &\rightharpoonup g, \quad wk-\ast  \; L^\infty 
\end{align}
From the definition of weak convergence, we deduce
\begin{equation}\label{seno12}
\begin{cases}
\; \;  \;  0 \leq g(t,x,\xi) \leq 1,  &\quad \text{for} \; \xi \geq 0,\\
 -1 \leq g(t,x,\xi) \leq 0,  &  \quad \text{for} \; \xi \leq 0.
\end{cases}
\end{equation}
The main step is a remark that allows to use
$g$ in order to replace the Young measures. It is more convenient to use because
we deal with a $ L^{1}_{{\text loc} ; \;  t,x}(L^{1}_{\xi}) $ function,

\begin{proposition}
Under the framework \eqref{seno10}, \eqref{seno11}
the function $g$ belongs to 
$L^{1}_{{\text loc} ; \;  t,x}(L^{1}_{\xi})$,
and satisfies for $\text{a.e.} \; (t,\;x)$,
\begin{equation}\label{seno13}
\int_{\R} \vert g(t,x,\xi) \vert \;d\xi =\text{w-lim}\vert u^n (t,x)
\vert \; \; 
\in L^{1}_{loc},
\end{equation}
\begin{equation}\label{seno14}
 \int_{\R} g(t,x,\xi)\;d\xi = u(t,x), 
\end{equation}
\begin{equation}\label{seno15}
\int_{\R}  S'(\xi) g(t,x,\xi) \; d\xi =\text{w-lim} \;
S\big(u^n (t,x)\big),
\quad \forall S'\in L^\infty, \; S(0)=0,
\end{equation}
Moreover,
\begin{equation}\label{seno16}
\text{w-lim} \; S\big(u^n (t,x)\big) = S(u) \, ,
\;  \forall S'\in L^\infty, \; S(0)=0, 
\quad \text{if and only if} \quad
g =\chi (\xi, u)   \, .
\end{equation}
\end{proposition}

\begin{proof} The definition of weak convergence applied to \eqref{seno11} implies that for $K$ compact in  $Q_T$ we have
\begin{equation}\label{defwkstar}
\iint_K \int_\R \chi(\xi, u^n ) \psi (t, x) \phi(\xi) dt dx d\xi \to \iint_K \int_\R  g(t,x,\xi) \psi (t, x) \phi(\xi) dt dx d\xi 
\end{equation}
for $\psi \in L^1(K)$, $\phi \in L^1(\R)$. For $\xi \ge 0$, consider a test function
with $\psi \ge 0$ supported in a compact $K$ and $\phi \ge 0$ compactly supported in $(0, \infty)$. Then \eqref{seno12}  implies
$$
0 \le \iint_K \int_\R  g(t,x,\xi) \psi (t, x) \phi(\xi) dt dx d\xi \le \iint_K \int_\R  \psi (t, x) \phi(\xi) dt dx d\xi \, .
$$
Hence $0 \le g(t,x,\xi) \le 1$ for a.e. $\xi \ge 0$, $(t,x) \in K$. A similar argument gives $-1 \le g(t,x,\xi) \le 0$
for $\xi \le 0$, $(t,x) \in K$, thus giving \eqref{seno12}.

To prove $g\in L^{1}_{loc; t,x} ( L^{1}_{\xi})$, note that \eqref{defwkstar} a-fortiori holds for test functions $\psi \in L^\infty(K)$ and
$\phi \in L^1(\R)$ with $0 \le (\sgn \xi ) \phi (\xi) \le 1$. Moreover,
$$
\begin{aligned}
\Big | \iint_K  \psi (t,x) \Big ( \int_\R \chi(\xi, u^n) \phi(\xi) d\xi \Big ) dt dx \Big | &\le \| \psi \|_{\infty} \iint_K \Big ( \int_\R | \chi(\xi, u^n) | d\xi \Big ) dt dx
\\
&= \| \psi \|_{\infty} \iint_K |u^n|  dt dx
\end{aligned}
$$
Passing to the limit,
$$
\begin{aligned}
\Big | \iint_K  \psi (t,x) \Big ( \int_\R g(t,x,\xi) \phi(\xi) d\xi \Big ) dt dx \Big | \le \| \psi \|_{\infty}  \liminf  \iint_K |u^n|  dt dx 
\end{aligned}
$$
Apply this to a sequence $\{ \phi_k \} \subset  \L^1 (\R)$ such that $0 \le (\sgn \xi) \phi_k  \le 1$ and $\phi_k \to \sgn \xi$ pointwise, to
deduce
$$
\frac{1}{ \| \psi \|_{\infty}} \Big | \iint_K  \psi (t,x) \big ( \int_\R |g(t,x,\xi) |  d\xi \big ) dt dx \Big | \le  \liminf  \iint_K |u^n|  dt dx \, .
$$
Using that $\big ( L^1 \big)^\prime = L^\infty$ and the Hahn-Banach Theorem gives
$$
\begin{aligned}
\iint_K \int_{\RR} |g| &\le  \liminf \iint_K \int_{\RR} |\chi (\xi, u^n)|
= \liminf \iint_K |u^n|
\end{aligned}
$$
and thus $g\in L^{1}_{loc; t,x} ( L^{1}_{\xi})$.

Next, we prove the identity  \eqref{seno14},
the other identities are variants and we will give a detailed proof of a similar statement Proposition \ref{prop65}. We introduce two real numbers, $R>0$ (large) and $S<0$ (small),
we have
\begin{equation}\label{seno17}
u^n = \int_{S}^{R} \chi (\xi, u^n) \; d\xi
+\int_{R}^{+\infty}
\chi (\xi, u^n) \; d\xi+\int_{-\infty}^{S} \chi (\xi, u^n)
\; d\xi.
\end{equation}
From the definition of $\chi (\xi, u^n)$, we deduce, after
a further subsequence,
$$
\int_{R}^{+\infty}
 \chi (\xi, u^n) \; d\xi = (u^n -R)_{+} \rightharpoonup v^{R}(t,x),
\quad wk-L^1_{\text loc} \, .
$$
The Dunford-Pettis characterisation of weak compactness in
$L^{1}$ implies
that, as $R \rightarrow \infty$,
\begin{equation}\label{seno18}
v^{R}(t,x) \rightharpoonup 0, \quad wk-L^1_{\text loc}. 
\end{equation}
In the same way, we find a function $v_{S}$ which also satisfies
\eqref{seno18}, and passing to the limit in \eqref{seno17}, we find
$$u(t,x) = v^{R}(t,x) + v_{S}(t,x) + \int_{S}^{R} g(t,x,\xi) \;d\xi.$$
Passing to the limit in $R$ and $S$ and using that
$g\in L^{1}_{loc; t,x} ( L^{1}_{\xi})$ we recover \eqref{seno14}.
The last statement of the Lemma is a direct consequence of  \eqref{seno15}.
\end{proof}

\begin{remark}\label{seno2}
1. The relation of $g$ with 
the Young (probability) measure $\mu $ associated with $\{u^n \}$ is
$$ 
\partial_{\xi}g(t,x,\xi) =\delta(\xi=0) - \nu_{(t,x)}(\xi).
$$

\noindent
2.
Sometimes the need arises to use the function $g_+(t,x,\xi)$ defined by
\begin{equation}\label{seno20}
\charf_{\xi < u^\eps}  \rightharpoonup g_+ , \quad wk-\ast \; L^\infty  \, .
\end{equation}
Note that $0 \le g_+ \le 1$ and that,  since $\charf_{\xi < u^\eps}$ is 
nonincreasing in $\xi$, the function $g$ is nonincreasing in $\xi$ as
well. By \eqref{seno2}, the functions $g$ and $g_+$ are connected by
\begin{equation}\label{seno21}
g = g_+ \charf_{\xi > 0} - (1 - g_+) \charf_{\xi <0}
\quad \text{a.e.} \; t, x, \xi \, .
\end{equation}

\noindent
3.  Let $S$ be a convex function,  $\alpha \in \RR$ and  
consider the minimization problem
$$
\begin{aligned}
\inf \int_{\RR} S' ( \xi ) f(\xi) \, d\xi
\quad \text{over $f \in L^1 (\RR) $ such that} \; \;
&\begin{cases}
\; \; 0 \le f \le 1 &\text{$\xi > 0$} \\
-1 \le f \le 0 &\text{$\xi < 0$}
\end{cases}
\\
\text{and} \; \; &\int_{\RR} f = \alpha
\end{aligned}
$$
Brenier \cite{Brenier83} showed that this minimization problem 
achieves its minimum at $f = \chi (\xi, \alpha)$, and that if $S$
is strictly convex the minimizer is unique (see \cite[Sec 2.2]{b-Perthame02} for the proof).
As an implication, for $S$ strictly convex with  
$S'\in L^\infty, \; S(0)=0$, we have
\begin{equation}\label{seno19}
\text{w-lim} \; S\big(u^n (t,x)\big)
= \int_{\R}  S'(\xi) g(t,x,\xi) \; d\xi \ge 
\int_{\R}  S'(\xi) \chi ( \xi, u (t,x)) d\xi = S(u)
\end{equation}
with equality if and only if $g(t,x,\xi) = \chi ( \xi, u (t,x))$.
Notice that, when formulated in terms of Young measures, \eqref{seno19}
amounts to Jensen's inequality.
\end{remark}

Next, consider a sequence $\{ u_n \}$ that is uniformly bounded in $L^p (K)$, for $K$ compact and $p \ge 1$,
\begin{equation}\label{Lpb}
\sup_{n \in \N} \int_K |u^n| \le C \, ,
\end{equation}
which is $L^p$-uniformly integrable, that is it satisfies for $p \ge 1$ and $c > 0$, 
\begin{equation}
\sup_{n \in \N} \int_{\{(t,x) \in K : |u^n| \ge c\} } |u^n|^p \to 0 \, , \quad \mbox{as $c \to \infty$}.
\end{equation}
After extracting subsequences (still denoted by $\{ u^n\}$)
\begin{align}
\label{seno10p}
 u^n &\rightharpoonup u, \quad wk-L^p_{\text loc}
\\
\label{seno11p}
\chi (\xi, u^n) &\rightharpoonup g, \quad wk-\ast  \; L^\infty 
\end{align}
Let $S(u)$ satisfy for $C>0$ the bound
\begin{equation}\label{Sbound1}
| S(u) | \le C (1 + |u|^p ) \, .
\tag {S$_1$}
\end{equation}
Then $|S(u)| \le 2C |u|^p$ for $|u| \ge 1$  and thus the sequence $\{ S(u^n) \}$ is uniformly integrable in $L^1(K)$.
An application of the Young measure Theorem \cite{Ball89} suggests that after extracting a subsequence there exists
a parametrized family of probability measures such that
\begin{equation}
\label{YMrep1}
S(u^n ) \rightharpoonup \overline{S(u)} = \int S(\xi) d\nu_{(t,x)} (\xi)  \quad \mbox{weakly in $L^1(K)$}
\end{equation}
for any continuous function $S$ satisfying \eqref{Sbound1}.

We prove:

\begin{proposition}\label{prop65}
Suppose that $S(u) : \R \to \R$ satisfies for $p \ge 1$ the bound
\begin{equation}\label{derb}
| S^\prime (u) | \le  C ( |u|^{p -1} + 1)
\tag {S$_2$}
\end{equation}
The weak limit $w-lim S(u^n) = \overline{S(u)}$ in \eqref{YMrep1} satisfies 
\begin{equation}\label{YMrep}
\int S(\xi) d\nu_{(t,x)} (\xi)  =  S(0) + \int_\R S^\prime (\xi) g(t,x, \xi) d\xi
\end{equation}
for any continuous $S$ satisfying \eqref{derb}.
\end{proposition}

\begin{proof} Observe that after integration the bound \eqref{derb} implies \eqref{Sbound1} and thus \eqref{YMrep1} holds.
The formulas \eqref{repformulaeta}, \eqref{kinfunction} give for $c > 0$
\begin{equation}
\label{repapprox}
\begin{aligned}
S(u^n) - S(0) &= \int_{-\infty}^{-c} \big ( -\charf_{u^n < \xi < 0} \big ) S^\prime(\xi) d\xi + 
\int_{-c}^{c}  \chi(\xi, u^n) S^\prime(\xi) d\xi 
+ \int_{c}^{\infty} \charf_{0 < \xi < u^n} S^\prime(\xi) d\xi
\\
&= \Psi_- (u^n, -c) + \int_{-c}^{c}  \chi(\xi, u^n) S^\prime(\xi) d\xi  + \Psi_+ (u^n, c)
\end{aligned}
\end{equation}

Using \eqref{derb} for $c >1$ we have
$$
\begin{aligned}
|\Psi_+ (u^n, c) | &= \Big | \int_{c}^{\infty} \charf_{0 < \xi < u^n} S^\prime(\xi) d\xi \Big |
\\
&= \Big | \int_{c}^{u^n} S^\prime(\xi) d\xi \Big | \charf_{c < u^n}
\\
&\le C \int_c^{u^n} (1 + \xi^{p-1} ) d\xi  \charf_{c < u^n}
\\
&\le C \big ( |u^n|^p - c^p \big ) \charf_{c < u^n}
\end{aligned}
$$
As a result, it follows
$$
\int_K \big | \Psi_+ (u^n, c)  \big | \le C \int_{\{(t,x) \in K : u^n > c \}} \big ( |u^n|^p - c^p \big )^+
$$
This implies $\Psi_+ (u^n, c) \in_b L^1(K)$ and is uniformly integrable. Hence, 
$$
\begin{aligned}
\Psi_+ (u^n, c) &\rightharpoonup \Psi_+^c (x)  \quad wk-L^1
\\
\Psi_+^c (x) &\rightharpoonup 0  \qquad wk-L^1
\end{aligned}
$$

In a similar fashion, for $c > 1$, 
$$
\begin{aligned}
|\Psi_- (u^n, -c) | &\le C \big ( |u^n|^p - c^p \big ) \charf_{u^n < - c}
\\
\int_K \big |\Psi_- (u^n, -c) \big | &\le C \int_{\{(t,x) \in K : u^n < -c\}} \big ( |u^n|^p - c^p \big )\charf_{u^n < -c}
\\
&= \int_{\{(t,x) \in K : u^n < -c\}} \big ( |u^n|^p - c^p \big )^+
\end{aligned} 
$$
$\Psi_- (u^n, -c) \in_b L^1(K)$ and is uniformly integrable, and
$$
\begin{aligned}
\Psi_- (u^n, -c) &\rightharpoonup \Psi_-^c (x)  \quad wk-L^1
\\
\Psi_-^c (x) &\rightharpoonup 0  \qquad wk-L^1
\end{aligned}
$$

Using \eqref{YMrep1} we pass to the limit in  \eqref{repapprox} and obtain
$$
\overline{S(u)}  - S(0) = \Psi_-^c (x) + \int_{-c}^c g(t,x,\xi) S^\prime (\xi) d\xi + \Psi_+^c (x)
$$
Subsequently, sending $c \to \infty$ and using $g \in L^{1}_{{\text loc} ; \;  t,x}(L^{1}_{\xi}) $ yields \eqref{YMrep}.
\end{proof}

\subsection{A proof of Tartar's theorem on cancellation of oscillations for scalar conservation laws}\label{sec:genkinproof}

The correspondence between Young measures and distribution functions 
was used in \cite{PT00} to give a particularly simple proof of Tartar's theorem on oscillations for scalar
conservation laws \cite{Tartar79}; namely, that the Young measure describing oscillations for a scalar conservation is supported in domains where 
the wave speed remains constant.

\begin{theorem}\label{seno3} 
Assume the flux $f$ is $C^{2}$ and let $\{ u^{\eps} \}$ satisfy
\eqref{seno10}, \eqref{seno20} and
\begin{equation}\label{seno22}
\del_t \eta (u^\eps) + \del_x q (u^\eps)
\; \; \; \text{lies in a compact of $H^{-1}_{loc}$} 
\end{equation}
for any $(\eta, q)$ with $\eta_u \in C_c^1 (\R)$.

Let $I = \{ \xi \in \RR : 0 < g_+ (x, t, \xi) < 1 \}$, then
for a.e. $(x,t)$ :
\begin{itemize}
\item[(i)] either $I$ an interval in which case the speed $\lambda (\xi)$ must
remain constant on $I$;

\item[(ii)] or $I$ is empty or a single point in which case $g = \charf (u, \xi)$.
\end{itemize}
\end{theorem}

\begin{proof}
Following the argument of Tartar for compensated compactness, we consider
two classes of entropy-entropy flux pairs of the type:
\begin{align}
&\eta_1 (u) = \int_\R \charf_{{\xi < u}} \phi (\xi) d\xi,
\quad 
&&\eta_2 (u) = \int_\R \charf_{u  < \theta} \psi (\theta) d\theta,
\\
&q_1 (u) = \int_\R \charf_{{\xi < u}} \lambda(\xi) \phi (\xi) d\xi,
\quad
&&q_2 (u)  = \int_\R \charf_{u < \theta} \lambda (\theta)
      \psi (\theta) d\theta,
\end{align}
where $\phi , \, \psi \in C_c^1 (\R)$ and we recall that $\lambda=f'$. Both
pairs are constant near infinity.
We denote with
brackets the various weak limits, in $(t,x)$ or in $(t,x,\xi)$, and 
apply the usual compensated compactness identity
$$
\langle  \eta_1 q_2 - \eta_2 q_1  \rangle =
\langle \eta_1 \rangle \langle q_2 \rangle  -  \langle \eta_2 \rangle  \langle q_1 \rangle \, .
$$

Using the representation formulas
$$
\eta_1 (u^{\eps}) \rightharpoonup
\int_\R \langle \charf_{{\xi < u^{\eps}}} \rangle  \phi (\xi) d\xi
\, , \quad
\eta_2 (u^{\eps}) \rightharpoonup 
\int_\R \langle \charf_{u^{\eps} < \theta} \rangle  \psi (\theta) d\theta
$$
and so on, we obtain
$$
\int_\R \int_\R  \big ( \lambda (\theta) - \lambda (\xi) \big )
\Big [ \langle \charf_{{\xi < u^{\eps}}} \charf_{u^{\eps} < \theta} \rangle
- \langle \charf_{{\xi < u^{\eps}}} \rangle  \; 
\langle \charf_{u^{\eps} < \theta} \rangle \Big ]
\phi(\xi) \psi(\theta) \; d\xi d\theta = 0 ,
$$
from where we conclude
\begin{equation}\label{seno23}
(\lambda (\theta) - \lambda (\xi) )
\Big [
\langle \charf_{{\xi < u^{\eps}}} \charf_{u^{\eps} < \theta}  \rangle
- \langle \charf_{{\xi < u^{\eps}}} \rangle \; \langle  \charf_{u^{\eps} < \theta} \rangle
\Big ] 
 = 0 \quad \text{a.e. $\xi , \theta$}.
\end{equation}

Recalling that $g_{+} (t,x,\xi) = \langle \charf_{{\xi < u^{\eps}}} \rangle$,
we have for $\theta < \xi$
$$
\big ( \lambda (\theta) - \lambda (\xi) \big ) \;
g_{+} (\xi) \big ( 1 - g_{+} (\theta) \big ) = 0 \, . 
$$
The function $g_{+}$ is decreasing and, due to the fact that
$g \in L^{1}_{\xi}$ and \eqref{seno21}, we have $g_{+} (-\infty) = 1$
and $g_{+} (\infty) = 0$.  The set $I$ is either empty, or a
single point, or an interval. If $I$ is an interval and
$\theta < \xi$ any interior points then 
$\lambda (\theta) = \lambda (\xi)$ and thus $\lambda$ stays constant
on $I$. If $I$ is empty or a single point, 
then $g_{+} (\xi) = \charf_{\xi < a}$ for some $a \in \RR$, and we
conclude from \eqref{seno21}, \eqref{seno2} and \eqref{seno14} that $a = u$ and
$g = \charf (u, \xi)$.  
\end{proof}

%\vfil\eject

\section{ Homogenization of one-dimensional viscoelastic models in phase transitions.} \label{sec:osci}

The problem of calculating effective equations for the propagation of oscillations in models
describing phase transitions for rate-dependent materials  can be addressed by a similar methodological approach. 
The system
\begin{equation}\label{vemodel}
\begin{aligned}
\del_t u &= \del_x v
\\
\del_t v &= \del_x (\sigma (u) + v_x ) \, ,
\end{aligned}
\end{equation}
describes shear motions of a viscoelastic material, with $u$ the shear deformation, $v$ the velocity in the shear direction,
and the total stress $S = \sigma (u) + v_x$
are defined for $0 \le x \le 1$, $t > 0$. We impose traction-free boundary conditions
\begin{equation}\label{vebc}
S (0,t) = S(1,t) = 0
\end{equation}
and consider highly oscillatory initial data
\begin{equation}\label{veid}
u^\eps (x, 0)= u_0^\eps (x) \, , \quad v^\eps (x, 0) = v_0^\eps (x) \, ,
\end{equation}
that are assumed to satisfy the uniform bounds
\begin{equation}\label{ub2data}
\| u^\eps_0 \|_{L^\infty} \le C \, , \quad  \| v^\eps_0 \|_{W^{1,\infty}} \le C   \tag{A}.
\end{equation}
As seen in the example of Section \ref{sec:1dnonmon} and Remark \ref{rmk:oscshear}, oscillating data can induce sustained oscillations in solutions
$(u^\eps, v^\eps)$ of \eqref{vemodel}-\eqref{veid}. The objective is to compute the effective equations describing the oscillations.
In the sequel, we drop the $\eps$-dependence where possible, but it will always be implied.
Existence of solutions for \eqref{vemodel} is studied in \cite{AB82} and references therein, whereas the problem of stability of solutions
with discontinuities in strain is studied in \cite{Pego87}.

\subsection{A-priori bounds and convergence framework}\label{weakconv}
We first set up a natural framework to consider the homogenization problem for \eqref{vemodel}.
The stress function $\sigma(u)$ is assumed non-monotone as appropriate for models with phase transitions.
It is postulated that  $\sigma(u) \to \pm \infty$ as $u \to \pm \infty$, and that the non-convex stored energy
$W(u) = \int_0^u \sigma (\tau) d\tau$ satisfies, for some $p > 1$ and constants $c, C  > 0$, the bounds
\begin{equation}\label{boundW}
\begin{aligned}
c | u|^p - C &\le W(u) \le C ( |u|^p  + 1)
\\
|\sigma(u)| &\le C \big ( 1 + |u|^{p-1} \big )
\end{aligned}
\tag{H$_1$} 
\end{equation}

For initial data satisfying uniform bounds \eqref{ub2data}  the formal energy identity
$$
\del_t \Big ( \frac{1}{2} v^2 + W(u) \Big ) - \del_x \big (v (\sigma (u) + v_x ) \big ) + v_x^2 = 0
$$
and the existence theory for \eqref{vemodel} yield the bound
\begin{equation}\label{energyb}
\int \tfrac{1}{2} (v^\eps)^2 (x,t)  + W(u^\eps (x,t) ) \, dx   + \int_0^t \int (v^\eps_x)^2 \le \int \tfrac{1}{2} (v_0^\eps)^2 + W( u_0^\eps ) \, dx 
= \cE^\eps (0)  \le C \, .
\end{equation}
Standard theory of weak convergence yields
$$
u^\eps \rightharpoonup u  \quad \mbox{wk-$\ast$  in $L^\infty (L^p)$}
$$

Moreover, the Aubin-Lions lemma states: Let  $X_0$, $X$, $X_1$ Banach spaces with $X_0 \subset X \subset X_1$ where the embedding
$X_0 \subset X$ is compact and the embedding $X \subset X_1$ is continuous. For $1 \le r, q \le \infty$ let 
$$
W = \big \{ u \in L^r ( (0,T) ; X_0) \; : \: \del_t u \in L^q ( (0,T) ; X_1 ) \big \}
$$ 
\begin{itemize}
\item[(i)] If $r < \infty$ then the embedding of $W$ into $L^r ( (0,T) ; X )$ is compact
\item[(ii)] If $r = \infty$ then the embedding of $W$ into $C ( [0,T] ; X )$ is compact
\end{itemize}

In the present case,  let $I= (0,1)$, $p \ge 2$ and $p^\prime \le 2$ its dual exponent. The embedding $H^1 (I) \subset L^2 (I)$ is compact while
the embedding $L^2 (I) \subset W^{-1, p^\prime} (I)$ is continuous. From the hypothesis \eqref{boundW} and \eqref{energyb} we deduce on
the one hand $v^\eps \in_b L^2 ( (0,T) ; H^1 (I) )$. On the other hand, we have 
$\sigma(u^\eps) \in_b L^\infty ( (0,T) ; L^{p^\prime} (I) )$ and $v^\eps_x \in_b L^2 ( (0,T) ; L^2 (I) )$ ; hence, using \eqref{vemodel}, 
$$
\del_t v^\eps = \del_x (  \sigma(u^\eps) + v^\eps_x  ) \in_b L^2 ( (0,T) ; W^{-1, p^\prime}  (I) ) \, .
$$
The Aubin-Lions lemma then implies
$$
v^\eps \to v  \quad  \mbox{strongly in  $L^2 \big( (0,T) ; L^2 (I) \big)$}
$$

\subsubsection{$L^\infty$ bounds for the strain}
One may get improved estimates, by adapting an idea of Andrews-Ball \cite{AB82},
 to yield uniform bounds on the shear strain (see \cite{Tzavaras24} for further details).
Using \eqref{vemodel}, we check that the function
\begin{equation}\label{defg}
g(x,t) = u(x,t) - \int_0^x v(y,t) dy  \, , \quad A(x,t) = \int_0^x v(y,t) dy
\end{equation}
satisfies the ordinary differential equation
\begin{equation}\label{eqng}
\del_t g + \sigma (g) = - \left (\int_0^1  \sigma^\prime \big ( g(x,t) + (1-s) A(x,t) \big ) ds \right ) A(x,t)
\end{equation}
Since $A(x,t)$ is estimated via \eqref{energyb} by
$$
| A(x,t)| := \Big | \int_0^x v(y, t) dy \Big | \le \Big ( \int_0^1 v^2(y,t) dy \Big )^{\tfrac{1}{2}} \le \sqrt{\cE(0)} = : M
$$
we see that $g$ satisfies the differential inequalities
\begin{align}
\label{ineq1}
\del_t g + \sigma (g)  &\le \max_{s \in [0,1] \, \; |A| \le M} \big | \sigma^\prime ( g + (1-s) A ) \big | | A|  = : B(g, M)
\\
\del_t g + \sigma (g)  &\ge  - \max_{s \in [0,1] \, \; |A| \le M} \big | \sigma^\prime ( g + (1-s) A ) \big | | A|  =   - B(g, M)
\end{align}
Define $\displaystyle{B(u,M) = \max_{s \in [0,1] \, \; |A| \le M} \big | \sigma^\prime ( u + (1-s) A ) \big | | A|}$.  Then we have:

\begin{lemma} \label{lem1}
Suppose $\sigma (u)$ has the limiting behavior
\begin{equation}\label{hypgr}
\begin{aligned}
\lim_{u \to \infty} \big ( \sigma (u) - B(u,M) \big ) &= + \infty
\\
\lim_{u \to - \infty} \big ( \sigma (u) + B(u,M) \big ) &= - \infty
\end{aligned}
\tag {H$_2$}
\end{equation}
Then solutions $(u^\eps, v^\eps)$ emanating from data \eqref{ub2data} satisfy the uniform bound
\begin{equation}\label{ubound}
| u^\eps (x,t) | \le K \, .
\end{equation}
\end{lemma}

{\it Proof.} 
Let $g$ satisfy \eqref{eqng} and note $|g^\eps_0 (x) | \le K$ uniformly in $\eps$.  By virtue of \eqref{hypgr} we may select $g_+ > 0$, $g_- < 0$ such that
$$
\begin{aligned}
\sigma(g ) - B(g, M) > 0  \quad \mbox{for $g > g_+$}
\\
\sigma(g ) + B(g, M) < 0  \quad \mbox{for $g < g_-$}
\end{aligned}
$$
Then
$$
\min\{ -K, g_- \} < g (x,t) < \max\{K, g_+ \}  \, .
$$
Indeed, by the assumption of the data, $- K < g(x,t) < K$ for $t$ sufficiently small. Let $t$ be the first time that for some $x$ we have $g(x,t) = \max\{K, g_+ \} $.
Then \eqref{ineq1} implies that at $(x,t)$ as above it is $g_t (x,t) \le 0$ which gives a contradiction. Similarly, follows the lower bound.
Equation \eqref{defg} implies that $|u^\eps (x,t)| \le K$ uniformly in $\eps$.
\qed

\subsubsection{Regularizing effect for the total stress}
The total stress  $S = \sigma (u) + v_x$
satisfies the linear parabolic equation
\begin{equation}\label{eqstress}
\del_t S = \del_{x x} S + \sigma^\prime (u) \big ( S - \sigma (u) \big )
\end{equation}
with boundary conditions \eqref{vebc} and initial conditions
$$
S (x,0) = \sigma (u_0^\eps) + v_{0 x}^\eps \in L^\infty \, ,
$$
uniformly bounded by  \eqref{ub2data}. Standard parabolic theory then yields.

\begin{lemma}\label{regstress}
Under the hypotheses of Lemma \ref{lem1} the total stress $S^\eps$ satisfies the uniform bound
\begin{equation}\label{stressb}
| S^\eps (x,t) | \le C 
\end{equation}
for $(x,t) \in Q_T = (0,1) \times (0,T]$, and along a subsequence
\begin{equation}\label{stressconv}
S^\eps \to S \quad \mbox{ a.e. $(x,t) \in Q_T$}.
\end{equation}
\end{lemma}

\subsection{Kinetic functions and Young measures}
Consider a family $(u^\eps, v^\eps)$ of solutions to \eqref{vemodel}-\eqref{veid} and proceed to calculate the effective
equation describing the homogenization process. Under hypotheses \eqref{ub2data}  the family $(u^\eps, v^\eps)$ satisfies \eqref{ubound}, \eqref{stressb} and has the convergence properties
that, along a subsequence,
\begin{equation}\label{wklim}
\begin{aligned}
u^\eps \rightharpoonup u  \quad \mbox{wk-$\ast$  in $L^\infty (Q_T)$}
\\
v^\eps \to v  \qquad \qquad  \mbox{in  $L^2 (Q_T)$}
\\
S^\eps \to S \quad \mbox{ a.e. $(x,t) \in Q_T$}
\end{aligned}
\end{equation}
We introduce the Young measure associated with the sequence $\{ u^\eps\}$, namely a parametrized family of probability measures
$\nu_{t,x} (\lambda)$ that represents weak limits
\begin{equation}\label{defYM}
\beta(u^\eps) \rightharpoonup \langle \nu_{t,x} , \beta(\lambda) \rangle  = \int \beta (\lambda) \, d\nu_{t,x} (\lambda)  \qquad \mbox{wk-$\ast$  in $L^\infty (Q_T)$}
\end{equation}
for any $\beta(u)$ continuous. Due to the uniform bound \eqref{ubound} the support of the Young measure verifies $\supp \nu_{t,x} \subset [-K, K]$.

We start by computing
\begin{equation}\label{epsfamily}
\del_t \beta (u^\eps) = (S^\eps - \sigma(u^\eps) ) \beta^\prime (u^\eps)
\end{equation}
and using the Young measure representation and Lemma \ref{regstress} we derive the family of equations
\begin{equation}\label{family}
\del_t \Big\langle \nu_{t,x} , \beta(\lambda) \Big\rangle = \Big  \langle \nu_{t,x} ,  S \beta^\prime (\lambda)  - \sigma (\lambda) \beta^\prime (\lambda)  \Big \rangle 
\end{equation}
An interesting formal derivation of an effective equation that describes the entire family of equations \eqref{family} is given in  \cite{Serre91}.
It exploits the fact that oscillations are scalar and is based  on a rendition of the Young-measure via  rearrangements, an idea attributed to  L. Tartar.

We present here an approach based on the kinetic function. 
Let $\{ u^\eps \}$ be the family of scalar-valued functions. The bound \eqref{ubound} implies, along a subsequence if necessary, that
\begin{equation}\label{defF}
\charf_{u^\eps < \xi}  \rightharpoonup F(t,x,\xi) \qquad \mbox{wk-$\ast$  in $L^\infty (Q_T \times \R)$}
\end{equation}
to some $F \in L^\infty (Q_T \times \R)$.  The function $F$ is nondecreasing in $\xi$ and satisfies $F(t,x, -\infty) = 0$, $F(t,x,\infty)=1$.
 The role of $F$ is indicated by the formulas 
\begin{equation}\label{repform1}
\begin{aligned}
\beta(u^\eps) = - \int_{-\infty}^{\infty} \beta^\prime (\xi) \charf_{u^\eps < \xi} d\xi  \rightharpoonup - \int_{-\infty}^{\infty} \beta^\prime (\xi) F(t,x,\xi) d\xi
\end{aligned}
\end{equation}
\begin{equation}\label{repform2}
\begin{aligned}
\beta(u^\eps) = \int_{-\infty}^{\infty} \beta^\prime (\xi) \charf_{\xi < u^\eps} d\xi  \rightharpoonup  \int_{-\infty}^{\infty} \beta^\prime (\xi) (1 - F(t,x,\xi)) d\xi
\end{aligned}
\end{equation}
for $\beta \in C^1_c (\R)$. 
Comparing with \eqref{defYM} we see that
\begin{equation}\label{ibp}
\int \beta (\lambda) \, d\nu_{t,x} (\lambda) = - \int_{-\infty}^{\infty} \beta^\prime (\xi) F(t,x,\xi) d\xi =  \int_{-\infty}^{\infty} \beta^\prime (\xi) (1 - F(t,x,\xi)) d\xi
\quad \mbox{for $\beta \in C^1_c(\R)$}.
\end{equation}
Hence, in the sense of distributions
$$
\nu_{t,x} (\xi) = \del_{\xi } F(t,x,\xi) \, .
$$
An alternate way to introduce $F(t,x,\xi)$ is to realize that since $\nu_{t,x} (\lambda)$ is a Borel measure on $\R$ one may introduce its distribution function
\begin{equation}\label{defdistr}
F(t,x,\xi) =  \int_{(-\infty, \xi]} d\nu_{t,x}(\lambda)
\end{equation}
This choice picks the right continuous representative of the function $F_{t,x} (\xi)$. The integration by parts formula 
\cite[Thm 3.30]{b-Folland99} for BV functions provides an alternate viewpoint to \eqref{ibp}.

We use the notation $dF_{t,x} (\xi)$ for integration against the Young measure. Then the weak$-\ast$ limit in $L^\infty$ of a continuous function $f(u)$
can be computed via the formula
\begin{equation}\label{moments}
\begin{aligned}
\overline{f(u)} &= \int f(\lambda) d\nu_{t,x} (\lambda) = \int f(\xi) dF_{t,x} (\xi)
\\
&= \int_{(-\infty, 0]} f(\xi) dF_{t,x} (\xi) + \int_{(0, \infty)} f(\xi) d(F_{t,x} (\xi) - 1)
\\
&= f(0) - \int_{-\infty}^0 f^\prime (\xi) F(t,x,\xi) d\xi + \int_0^\infty f^\prime (\xi) ( 1 - F(t,x,\xi) ) d\xi 
\end{aligned}
\end{equation}
where we have used the integration by parts formula, with the right continuous representative of $F(\xi)$, and the facts that $\nu$ has compact support 
and $F(-\infty) =0$, $F(+\infty) =1$. Formula \eqref{moments} computes the weak limit $\overline{f(u)} = \mbox{wk-lim} f(u^n)$. Representation formulas of weak limits via kinetic functions are discussed in Section \ref{sec:repweaklim}.  Items 1. and 2. in Remark \ref{seno2} 
present the relations of the various available tools; note that $g_+$ in \eqref{seno21} and $F$ in \eqref{defF} are related by $g_+ = 1 -F$.

\subsubsection{The effective system}\label{sec:effs}
We compute two forms of the effective system describing the propagation of oscillations in \eqref{vemodel}-\eqref{veid}:

The first one is
\begin{equation}\label{effs1}
\begin{cases}
\del_t F + \del_\xi \Big ( \big ( v_x + \overline{\sigma(u)} - \sigma(\xi) \big ) F \Big ) + \sigma^\prime (\xi) F = 0   & \\[5pt]
\del_t v = \del_{xx} v + \del_x  \big (\overline{\sigma(u)} \big )    & \\[5pt]
S =  \overline{\sigma(u)}  + v_x   = \int \sigma(\xi) dF_{t,x} (\xi)  + v_x    & \\
\end{cases}
\quad (x,t) \in Q_T \, \; \; \xi \in \R
\end{equation}
with boundary and initial data
\begin{align}
S(0,t) &= S(1,t) = 0  \, , \quad   0 < t < T
\label{effbc1}
\\
v(0,x) &= v_0 (x) = \mbox{s-lim} v_0^\eps (x) \, , \quad F(0,x,\xi) = \mbox{wk-$\ast$ lim} \Big (  \charf_{ u_0^\eps (x) < \xi } \Big )
\label{effdata1}
\end{align}
where $v(t,x)$ and $F(t,x,\xi)$ are defined in \eqref{wklim} and \eqref{defF}.

To obtain \eqref{effs1}$_1$, start from \eqref{epsfamily} with $\beta(\cdot) \in C_c^1 (\R)$ and use the
representation formula \eqref{repform1} and the formula $\del_\xi \charf_{u^\eps < \xi} = \delta (\xi - u^\eps)$ to write
$$
\begin{aligned}
\del_t \Big ( - \int \beta^\prime(\xi) \charf_{u^\eps < \xi} d\xi \Big ) &= \int \big (S^\eps - \sigma(\xi) \big )   \beta^\prime(\xi) \delta(\xi - u^\eps) d\xi
\\
&= \int \beta^\prime(\xi)  \Big [ \del_\xi \Big ( \big(S^\eps - \sigma(\xi) \big ) \charf_{u^\eps < \xi} \Big )+ \sigma^\prime(\xi) \charf_{u^\eps < \xi} \Big ] d\xi
\end{aligned}
$$
Using an argument as  in the derivation of the generalized kinetic equation \eqref{genkineq} we conclude that in the sense of distributions
\begin{equation}
\del_t \charf_{u^\eps < \xi}  + \del_\xi \Big (  \big (S^\eps - \sigma(\xi) \big ) \charf_{u^\eps < \xi} \Big ) + \sigma^\prime(\xi) \charf_{u^\eps < \xi} = 0 \, .
\end{equation}
Next, using \eqref{defF}, \eqref{stressb} and \eqref{stressconv}, we pass to the limit $\eps \to 0$ and obtain \eqref{effs1}$_1$. The remaining
equations easily follow from \eqref{wklim}.

The equation \eqref{effs1}$_1$ may be written in the form
\begin{equation}
\del_t F + ( S-\sigma(\xi) ) \del_\xi F = 0
\end{equation}
and this writing is rigorous when $\sigma(\xi)$ is $C^1$ and $\del_\xi F$ is a measure. Let $f(u)$ be a $C^1$ function.
Using \eqref{moments}, we compute
\begin{align}
\del_t \overline{f(u)} &= \del_t \Big ( f(0) - \int_{-\infty}^0 f^\prime (\xi) F d\xi - \int_0^\infty f^\prime (\xi) (F-1) d\xi \Big )
\nonumber
\\
&= - \int_{-\infty}^0 f^\prime (\xi)  \del_t F d\xi - \int_0^\infty f^\prime (\xi) \del_t (F-1) d\xi
\nonumber
\\
&= \int_{-\infty}^0 (S - \sigma) f^\prime \del_\xi F d\xi + \int_0^\infty (S-\sigma) f^\prime \del_\xi F d\xi
\nonumber
\\
&= \int_{(-\infty, 0]} (S - \sigma) f^\prime  dF_{t,x} (\xi) + \int_{(0,\infty)} (S-\sigma) f^\prime dF_{t,x} (\xi)
\nonumber
\\[5pt]
&= \overline{(S-\sigma) f^\prime}
\label{formula3}
\end{align}

We apply this formula to the test function $f(u) = u$ and obtain
$$
\del_t u = \overline{(S-\sigma) } = v_x
$$
which provides a derivation of \eqref{vemodel}$_1$ for solutions of the system \eqref{effs1}. One may also use \eqref{formula3}
to derive an alternate form of the effective system. Namely, using \eqref{formula3} for $f(u) = \sigma(u)$, we compute
$$
\del_t S = v_{xt} + \del_t  \overline{\sigma(u)} = S_{xx} + \overline{(S-\sigma) \sigma^\prime}
$$

This leads to a second form for the effective system,
\begin{equation}\label{effs2}
\begin{cases}
\del_t F + \del_\xi \Big ( \big ( S - \sigma(\xi) \big ) F \Big ) + \sigma^\prime (\xi) F = 0   & \\[5pt]
\del_t S = \del_{xx} S + \int \sigma^\prime(\xi) (S - \sigma(\xi) ) dF_{t,x} (\xi)    & \\
\end{cases}
\; , \quad (x,t) \in Q_T \, \; \; \xi \in \R \, ,
\end{equation}
subject to boundary and initial data
\begin{align}
S(0,t) &= S(1,t) = 0  \, , \quad   0 < t < T \, , 
\label{effbc2}
\\
S(0,x) &= S_0 (x) = \mbox{wk lim}\big ( \sigma(u^\eps_0(x)) + \del_x v_0^\eps \big ) \, , \quad F(0,x,\xi) = \mbox{wk-$\ast$ lim} 
\Big (  \charf_{ u_0^\eps (x) < \xi } \Big ) \, .
\label{effdata2}
\end{align}
An alternative way for deriving \eqref{effs2}$_2$ is to start from \eqref{eqstress}  and pass to the limit $\eps \to 0$ using
\eqref{stressconv} and Young measures to represent the reaction term.

\medskip
We summarize:

\begin{theorem}\label{mainthm}
Let $(u^\eps, v^\eps)$ and $S^\eps$ be a family of solutions of \eqref{vemodel}-\eqref{veid} induced by data satisfying \eqref{ub2data}
with $\sigma(u)$  nonmonotone  satisfying  \eqref{boundW} and \eqref{hypgr}. Let $\big ( v, u, S \big )$ be defined via \eqref{wklim}
and $F(t,x,\xi)$ via \eqref{defF}. Then $F$ is connected to the Young measure via  \eqref{defdistr}  and $(F, v)$ satisfies the system \eqref{effs1}-\eqref{effdata1}.
As an alternative, $(F,S)$ can be determined via the system \eqref{effs2}-\eqref{effdata2}.
\end{theorem}

Given $(F,v)$ the function $u$ is determined by the moment
\begin{equation}\label{formu}
u = - \int_{-\infty}^0 F(t,x,\xi) d\xi - \int_0^\infty (F(t,x,\xi) -1) d\xi
\end{equation}
and satisfies $u_t = v_x$. Alternatively, given $(F, S)$ determined via \eqref{effs2}-\eqref{effdata2}, 
the function $u$ is again determined by \eqref{formu}. We wish to define now $v(t,x)$
by the relation
$$
v_t = S_x \, ,  \quad v_x = S - \overline{\sigma(u)}
$$
Due to  \eqref{effs2}$_2$ and the formula $\del_t  \overline{\sigma(u)} = \overline{(S-\sigma)\sigma^\prime}$ (following from \eqref{formula3}),
these equations define $v$ and satisfy the desired equations $v_t = S_x$ and $u_t = v_x$.

\subsubsection{The behavior of $F$ across interfaces for prescribed stress.}
We saw  the effective response is captured by $(F,S)$ determined by the coupled system \eqref{effs2}.
The system consists of a kinetic equation, whose coefficients depends on the total stress $S$, coupled with
an equation for the total stress with a right hand side evaluated via a moment. (Alternatively, one may work with \eqref{effs1}.)

To gain some intuition, we consider the effect of the kinetic equation when the total stress $S(t,x)$ is given.
The distribution function $F(t,x,\xi)$ then satisfies the initial value problem
\begin{equation}\label{effivp}
\begin{aligned}
\del_t F + ( S-\sigma(\xi) ) \del_\xi F = 0
\\
F(0, x,\xi) = F_0(x,\xi) \, .
\end{aligned}
\end{equation}
Recall that $F$ is almost everywhere equal to the distribution function of the Young measure,
that is $F$ satisfies \eqref{defdistr}.  In particular, $F$ will be nondecreasing, $0 \le F \le 1$, and if
$\supp \nu_{(t,x)} \subset [\xi_1 , \xi_2]$ then $F = 0$ for $\xi < \xi_1$, $F = 1$ for $\xi \ge 2$ and 
$0 \le F \le 1$ for $\xi \in (\xi_1, \xi_2)$.
%$$
%F(t,x,\xi) = 
%\begin{cases}  
%0  & \xi <  \xi_1  \\
%0 \le F(t,x, \xi ) \le 1   & \xi_1 < \xi < \xi_2  \\
%1  & \xi_2 \le \xi  \\
%\end{cases}
%$$

The initial data $F_0$ are generated by 
$$
F_0 (x,\xi) = \mbox{wk-$\ast$ lim} \Big (  \charf_{ u_0^\eps (x) < \xi } \Big ) \, . 
$$
This means that for $x$-fixed the function $F_0(x, \cdot)$ is the distribution function that describes the statistics
of the oscillating sequence  of values $\{ u_0^\eps (x) \}$.  When we have strong convergence $u_0^\eps \to u_0$ then $F_0$ is a Heaviside 
function with a single jump  at $u_0$, and the corresponding Young measure is $\nu_{(0, x)} = \delta (\xi - u_0(x))$.

Consider the problem \eqref{effivp} and to simplify study the case that $S = S(x)$ is independent of $t$. Fix $x_0$ and let
$S_0 = S(x_0)$. The equation
$$
\del_t F + ( S_0 -\sigma(\xi) ) \del_\xi F = 0
$$
is described through its characteristic system
$$
\begin{aligned}
\frac{ \del \xi}{\del t} &= S_0 - \sigma (\xi)
\\
\frac{\del F}{\del t} &= 0
\\
\xi (0) &= \xi_0 \, , \quad F(0) = F_0
\end{aligned}
$$
This implies that $F(t)  = const = F_0$ on the characteristic $\xi(t ; \xi_0)$ emanating from $\xi_0$. 

Suppose now that the function $\sigma(u)$ is nonmonotone and of the form presented in Fig. \ref{fig:graph}, and let $T$ and $S$ be the top and bottom of
the stress curve. Fix $S_0$ and let $a$, $b$, $c$ the three intersections that give the equilibria of the ordinary differential equation
$\frac{ \del \xi}{\del t} = S_0 - \sigma (\xi)$. Note that $a$ and $c$ are stable equilibria while $b$ is unstable. Moreover, the phase space is split
to initial data $\xi_0 < b$ where the corresponding characteristic $\xi(t ; \xi_0) \to a$ as $t \to \infty$ and initial data $\xi_0 > b$ where the characteristic
$\xi(t ; \xi_0) \to c$ as $t \to \infty$.
\begin{figure}[htbp]
\begin{center}
\includegraphics[scale=0.45]{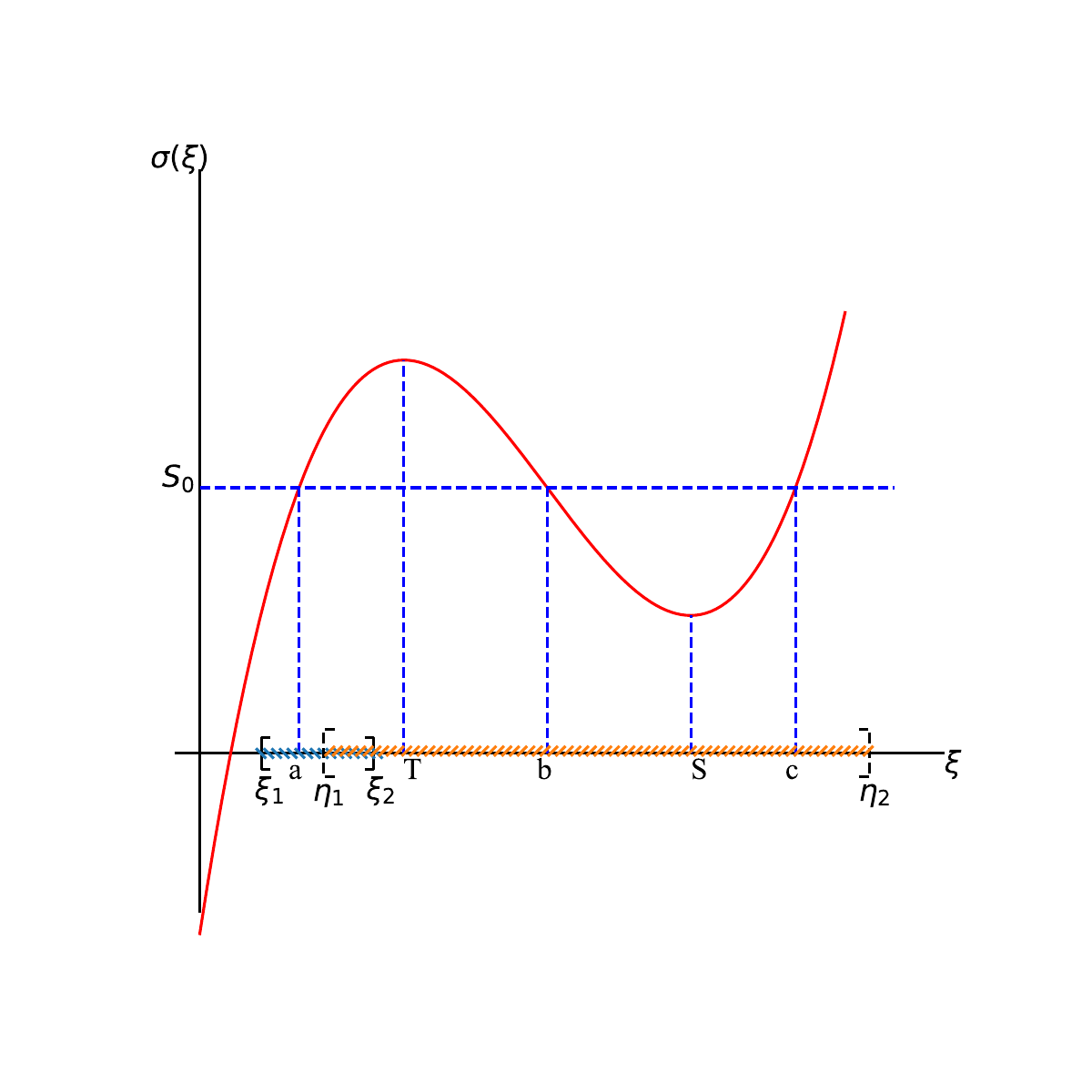}
\vspace{-0.8cm}
\caption{Placement of initial data relative to the constitutive function}
\label{fig:graph}
\end{center}
\end{figure}

This provides a particularly simple solution to the initial value problem \eqref{effivp} for smooth $F_0(x,\xi)$ that gives insight 
on the behavior of the distribution function. Suppose that for $x_0$ fixed and $S_0 = S(x_0)$ the initial data are of the form
$$
F(x_0 ,\xi) = 
\begin{cases}  
0  & \xi <  \xi_1  \\
A(\xi)    & \xi_1 < \xi < \xi_2  \\
1  & \xi_2 \le \xi  \\
\end{cases}
$$
where $\xi_1 < \xi_2 < b$ and $A(\xi)$ is a continuous function increasing from $A(\xi_1) =0$ to $A(\xi_2)=1$. The reader is referred to 
Fig.\ref{fig:graph} for the placement of the support of the initial oscillations $[\xi_1, \xi_2]$. Analyzing the characteristics
we see that the solution $F(t, x_0, \xi)$ then tends as $t \to \infty$ to the Heaviside function $H ( \xi - a)$, where $H(\cdot)$ is the function
$H(\xi) = 0$ for $\xi < 0$ and $H(\xi) = 1$ for $\xi \ge 0$. In that case the limiting solution 
has a single phase.

On the other hand if the initial data are
$$
F(x_0 ,\xi) = 
\begin{cases}  
0  & \xi <  \eta_1  \\
B(\xi)    & \eta_1 < \xi < \eta_2  \\
1  & \eta_2 \le \xi  \\
\end{cases}
$$  
where $\eta_1  < b < \eta_2$ and $B(\xi)$ is continuous and increasing from $B(\eta_1) =0$ to $B(\eta_2)=1$, see Fig.\ref{fig:graph}.
The solution then splits into two parts that eventually in infinite time develop to two distinct phases, and $F(t, x_0, \xi)$ will
converge to the sum of two Heaviside functions $\alpha_1 H(\xi -a) + \alpha_2 H(\xi -c)$ supported on the two stable equilibria $a$, $c$.
The associated Young measure will approach in large times the measure $\mu = \alpha_1 \delta_a + \alpha_2 \delta_c$.

\section{Homogenization for the compressible Navier-Stokes system}\label{sec:homcomNS}

We focus next on the homogenization problem for the compressible Navier-Stokes system.
The problem was suggested and formally developed in Serre \cite{Serre91}. Hillairet \cite{Hillairet07} provided a
rigorous description of the problem that we now outline.

\subsection{The homogenization problem}
For the isentropic gas law $p(\rho) = \rho^\gamma$ with $\gamma > \tfrac{3}{2}$ a class of
global finite-energy weak solutions $(\rho^n, u^n)$ are considered that are uniformly bounded
in energy, that is either like in the work of Lions,
\begin{equation}\label{hypdata1}
\sup_{t \in (0,T)} E  \big[ \rho^n u^n, \rho^n \big] (t) < E_0  \quad n \in \N
\end{equation}
or as in the works of Feireisl 
\begin{equation}\label{hypdata2}
\limsup_{t \to 0+} E  \big[ \rho^n u^n, \rho^n \big] (t) < E_0  \quad n \in \N \, .
\end{equation}
Under these assumptions, using the Lions or the Feireisl existence theory, there exists a pair
$(\rho, u) : Q_T \to (\R^+ , \R^3)$ 
\begin{equation}\label{regularity}
\rho \in L^\infty \big ( (0,T) ; L^\gamma (\Omega) \big ) \, , \quad u \in L^2 \big ( (0,T) ; H^1_0 (\Omega) \big )
\end{equation}
that satisfies 
\begin{equation}\label{wklimNS}
\begin{aligned}
&\rho^n \rightharpoonup \rho  \quad \mbox{ wk-$\ast$ in $L^\infty \big ( (0,T) ; L^\gamma (\Omega) \big )$}
\\
&u^n \rightharpoonup u \quad \mbox{ wk in $L^2 \big ( (0,T) ; H^1_0 (\Omega) \big )$} 
\end{aligned}
\end{equation}

Hillairet \cite{Hillairet07} considers the problem of derivation of equations describing the effective
response of the sequence $(\rho^n, u^n)$. The problem is addressed  using Young measure solutions. 
The idea is to derive equations for
the weak limits of renormalized solutions (see part (iii) of Definition \ref{def:fews}). It is proved:

\begin{theorem}\label{thmhomoCNS} \cite{Hillairet07}
Assume $(\rho^n, u^n)$ satisfies hypothesis \eqref{hypdata1} or \eqref{hypdata2} and $(\rho, u)$ with regularity
\eqref{regularity} is the weak limit \eqref{wklimNS}. There exists a measurable family of parametrized probability measures 
$\big \{ \nu_{(t,x)}  \big \}_{(t,x) \in Q_T}$ such that
\begin{itemize}
\item[(a)] $\rho$, $\nu$ satisfy
$$
\rho = \langle \nu_{(t,x)}  , \lambda \rangle \, , \quad \overline{p(\rho)} := \langle \nu_{(t,x)}  , p(\lambda) \rangle = \int p(\lambda) d\nu_{(t,x)}(\lambda)
$$
\item[(b)] for any $b \in C([0, \infty))$, smooth and with compact support
\begin{equation}\label{renormYM}
\begin{aligned}
&\del_t \langle \nu_{(t,x)} , b(\lambda) \rangle + \div \Big (\big \langle \nu_{(t,x)} , b(\lambda) \big\rangle u \Big )
+ \big \langle \nu_{(t,x)} , \lambda b^\prime(\lambda) - b(\lambda) \big\rangle \div u
\\
&\qquad = \tfrac{1}{\lambda + 2 \mu} \Big ( \big \langle \nu_{(t,x)} , \lambda b^\prime(\lambda) - b(\lambda) \big\rangle \overline{p(\rho)}  
-  \big \langle \nu_{(t,x)} , (\lambda b^\prime(\lambda) - b(\lambda)) p(\lambda) \big\rangle \Big )
\end{aligned}
\end{equation}
in $\dis (Q_T)$.
\item[(c)]  $(\rho, u)$ satisfies the equations
$$
\begin{aligned}
\rho_t + \div (\rho u) &= 0
\\
(\rho u)_t + \div (\rho u \otimes u ) + \nabla \; \overline{p(\rho)} &= \mu \triangle u + (\lambda + \mu) \nabla (\div u)
\end{aligned}
\qquad \mbox{in $\dis (Q_T)$} \, .
$$
\end{itemize}
\end{theorem}

The proof of Hillairet \cite{Hillairet07} uses the key steps needed in the existence theory of Lions and Feireisl.
As the author states the techniques developed in the existence theory \cite{b-Lions98, FNP02, b-Feireisl04}
enable to justify rigorously the formal calculus leading to Theorem \ref{thmhomoCNS}.  
Beyond providing the proof of that theorem, a main novelty in  \cite{Hillairet07} is the study of a particular type of Young measure solutions of the form
$$
\nu_{(t,x)} = \sum_{i=1}^n \alpha_i (t,x) \delta_{\rho_i (t,x)}
$$
that consists on $n$ non-interacting phases and to explore the properties of such multi-phase flows. This is an interesting connection
between multiphase flows and the homogenization problem; see  \cite{Hillairet07} for theoretical results on the existence of such solutions.

The derivation of the homogenized system can be found in \cite{Hillairet07} with some
elements  summarized below. First, it is proved that
$$
\begin{aligned}
\rho^n u^n \rightharpoonup \rho u \qquad \qquad &\mbox{wk in $L^2 ( (0,T) ; L^{\frac{2\gamma}{\gamma +1}} (\Omega) \big )$}  
\\
\rho^n u^n \otimes u^n  \rightharpoonup \rho u \otimes u \qquad &\mbox{wk in $\dis (Q_T)$}
\end{aligned}
$$
which yields part (c).

The approximate solutions satisfy the renormalized equation for the density
\begin{equation}\label{apprenentr}
\del_t b(\rho^n) + \div \big( b(\rho^n) u^n \big) = \big ( b(\rho^n) - \rho^n b^\prime (\rho^n) \big ) \div u^n  \quad \mbox{in $\dis(Q_T)$}
\end{equation}
for $b(\cdot)$ continuous functions on $\RR$ that satisfy $b(\rho) = const$ for $\rho$ sufficiently large. Young measures are introduced to
represent the weak limits
$$
b(\rho^n) \rightharpoonup \overline{b(\rho)} = \int b (\lambda) d\nu_{(t,x)} (\lambda)  \quad \mbox{weak-$\ast$ in $L^\infty$}.
$$
The passage to the limit of \eqref{apprenentr} is based on three ingredients that are part of the existence theory of compressible Navier-Stokes.
An a-priori estimate that $p(\rho^n)$ is uniformly bounded in $L^{\alpha_c} (Q_T)$ for some $\alpha_c > 1$. This estimate, achieved via harmonic analysis techniques by Lions or using the Bogovskii operator by Feireisl, 
 allows to represent weak limits for  functions $b(\rho)$ that grow up to $\rho^\gamma$
via Young measures,  in particular 
showing that $p(\rho^n) \rightharpoonup \overline{p(\rho)}$, see \cite{Feireisl01} \cite[Sec 3.1, 4.2, 4.3]{b-NS04}.

The remaining ingredients are weak convergence properties that pertain the theory of compressible Navier-Stokes,
like the property
\begin{equation}\label{weaku}
b(\rho^n) u^n \rightharpoonup \big ( \overline{b(\rho)} \big ) u \quad \mbox{in $\dis(Q_T)$}.
\end{equation}
and the special role of the effective viscous flux
$z^n = p(\rho^n) - (\lambda + 2 \mu) \div(u^n)$ regarding its weak convergence properties: namely, that 
following application of compensated compactness \cite{Lions93} \cite{Feireisl01}, \cite[Sec 8.2]{BJ18},
\begin{equation}\label{weakz}
\Big ( p(\rho^n) - (\lambda + 2 \mu) \div(u^n) \Big ) b(\rho^n) \rightharpoonup \big (\overline{p(\rho)} - (\lambda + 2 \mu) \div u \big ) \overline{b(\rho)}
\quad \mbox{in $\dis (Q_T)$.}
\end{equation}
Using these ingredients allows to derive equation \eqref{renormYM} for the Young measures. We refer to 
\cite{Hillairet07} for the proof and note that the full analysis is valid for the isentropic gas 
law $p(\rho) = \rho^\gamma$ with $\gamma > \frac{3}{2}$.

\subsection{An approximate effective kinetic equation}\label{sec:effmethod1}
Equation \eqref{renormYM} is not easily amenable to analysis. However, as the only oscillating variable is the density
which is scalar, the distribution functions of the Young measure will be employed to study oscillations.
This is in analogy to the kinetic description of scalar conservation laws in Section \ref{sec:genkinsoln} or the
description of oscillations for phase transitions in one-dimensional viscoelasticity discussed in Section \ref{sec:osci}.

For compressible Navier-Stokes of a  barotropic gas, Plotnikov and Sokolowski \cite[Ch 7]{b-PS12}
exploit the correspondence between Young measures and distribution functions (in the scalar context) to derive a kinetic equation 
for weak limits of the renormalized family  \eqref{renormYM}. They employ it to derive an alternate proof of Lions's theorem 
on propagation of compactness for a barotropic gas.

 It is expedient to put the problem of describing oscillations for the density in a wider context, and allow the pressure $p(\rho)$ to be non-monotone, thus 
 including models like Van-der-Waals gases. 
 We refer to Bresch-Jabin \cite{BJ18} for a discussion of  global existence for general constitutive laws and the compactness properties of 
 approximate solutions.  We use the same assumptions on the pressure
 \begin{equation}
 \begin{aligned}
 p(\rho) \in C^1([0, \infty)) \, , \quad p(0) = 0
 \\
 \frac{1}{a}  \rho^{\gamma -1} - b \le p^\prime (\rho) \le a \rho^{\gamma -1} + b
 \end{aligned}
 \tag {HP}
 \end{equation}
 with $\gamma > \frac{3}{2}$ and some constants $a, b > 0$ in dimension $d=3$.

 The first objective is to obtain an approximate kinetic equation and then  survey the conditions necessary to pass to the limit and derive a
 kinetic model for the distribution function. The equation
\begin{equation}
\del_t b(\rho^n) + \div \big (  b(\rho^n) u^n \big ) + \big ( \rho^n b^\prime (\rho^n)  - b(\rho^n) \big ) \div u^n  = 0
\end{equation}
is expressed in the form
\begin{equation}\label{approx}
\del_t b(\rho^n) + \div \big( b(\rho^n) u^n \big)
- \tfrac{1}{\lambda + 2 \mu} \big ( \rho^n b^\prime (\rho^n)  - b(\rho^n) \big )  z^n 
+ \tfrac{1}{\lambda + 2 \mu}  \big ( \rho^n b^\prime (\rho^n)  - b(\rho^n)  \big ) p(\rho^n) = 0 \, ,
\end{equation}
understood in the sense of distributions $\dis(Q_T)$, using the effective viscous flux
\begin{equation}\label{effvf}
z^n = p(\rho^n) - (\lambda + 2 \mu) \div(u^n) \, .
\end{equation}

The equation \eqref{approx} will be represented for a class of renormalized solutions. Let $\theta(\xi) \in C^1_c (\R)$ and consider functions 
$b(\rho)$ defined by
$$
b(\rho) = \int_{-\infty}^\rho \theta(\xi) d\xi = \int_\R \charf_{\xi < \rho} \theta(\xi) d\xi = \int_\R \big ( 1 - \charf_{\rho < \xi} \big ) \theta(\xi) d\xi
$$
Note that $\theta = b^\prime \in C_c^1 (\R)$ and that if $\supp \theta \subset [-m, m]$ then $b(\rho) = 0$ for $\rho < -m$ and $b(\rho) = const$ for $\rho > m$.

The goal is to represent the last two terms in \eqref{approx}. Consider the first term $I^n_1$ and observe
that at least formally
$$
\begin{aligned}
I_1^n := \big ( \rho^n b^\prime (\rho^n)  - b(\rho^n) \big ) z^n
&= \int_\R ( \xi b^\prime (\xi) - b(\xi) ) \delta (\xi - \rho^n) \,  z^n d\xi
\\
&= \int_\R ( \xi b^\prime (\xi) - b(\xi) ) \big ( \del_\xi ( \charf_{\rho^n < \xi} - 1) \big ) z^n d\xi
\\
&= - \int_\R \xi b^{\prime \prime} (\xi) \big( \charf_{\rho^n < \xi} -1 \big) z^n d\xi 
\\
&= \int_\R b^\prime (\xi) \del_\xi \Big ( \xi ( \charf_{\rho^n < \xi} - 1) \, z^n \Big ) d\xi
\end{aligned}
$$
To check the validity of this equation and interpret it appropriately as a distribution we express $b(\xi)$ in terms of 
$\theta (\xi) = b^\prime(\xi)$ and write
$$
\begin{aligned}
I_1^n &= \big ( \rho^n \theta (\rho^n) - \int_{-\infty}^{\rho^n} \theta(s) ds \big ) z^n
\\
&= \Big ( \int_{-\infty}^{\rho^n}  s \theta^\prime (s) ds \Big ) z^n
\\
&= \int_\R \charf_{s < \rho^n} (s \theta^\prime (s) ) z^n \, ds
\\
&= \big \langle \theta^\prime (\xi) , \xi \charf_{\xi < \rho^n} z^n \big \rangle
\\
&= \Big \langle \theta (\xi) , - \del_\xi \big ( \xi \big ( 1 - \charf_{\rho^n < \xi} \big)  z^n \big ) \Big \rangle
\end{aligned}
$$
where the last two formulas are interpreted in distributions.

Consider next the term $I_2^n$ given by
$$
\begin{aligned}
I_2^n &:= \big ( \rho^n b^\prime (\rho^n)  - b(\rho^n) \big ) p(\rho^n) 
\\
&= \big ( \rho^n \theta (\rho^n) - \int_{-\infty}^{\rho^n} \theta(s) ds \big ) p(\rho^n)
\\
&= \Big ( \int_{-\infty}^{\rho^n}  s \theta^\prime (s) ds \Big ) p(\rho^n)
\\
&= \int_{-\infty}^{+\infty} \charf_{\tau < \rho^n}  \frac{d}{d\tau} \Big [ \Big ( \int_{-\infty}^{\tau}  s \theta^\prime (s) ds \Big ) p(\tau) \Big ] d\tau
\\
&= \int_\R \theta^\prime (\tau) \tau p(\tau) \big ( 1 - \charf_{\rho^n < \tau}  \big ) d\tau
+ \int_\R \left ( \int_\R \charf_{s<\tau} s \theta^\prime (s) ds \right ) p^\prime(\tau) \big (1 - \charf_{\rho^n < \tau } \big ) d\tau
\\
&= J_1^n + J_2^n
\end{aligned}
$$
The term $J_1^n$ is expressed as a distribution
$$
\begin{aligned}
J_1^n &= \int_\R \theta^\prime(\xi) \, \xi p(\xi) \big (1- \charf_{\rho^n < \xi} \big ) \, d\xi
\\
&= - \Big \langle \theta (\xi) , \del_\xi \Big ( \xi p(\xi) \big (1 - \charf_{\rho^n < \xi} \big )  \Big ) \Big \rangle
\end{aligned}
$$
Similarly, the term $J_2^n$ gives
$$
\begin{aligned}
J_2^n &= \int_\R \theta^\prime(\xi) \, \xi \left (\int_{-\infty}^\infty \charf_{\xi < \zeta} \charf_{\zeta < \rho^n } p^\prime (\zeta)   \, d\zeta \right )d\xi
\\
&= - \Big \langle \theta (\xi) , \del_\xi \Big ( \xi \int_\xi^\infty p^\prime(\zeta) \big (1- \charf_{\rho^n < \zeta} \big ) d\zeta \Big ) \Big \rangle
\end{aligned}
$$
Combining the two formulas we have
$$
I_2^n = - \Big \langle \theta (\xi) , \del_\xi 
\left ( \xi p(\xi) \big ( 1 - \charf_{\rho^n < \xi} \big ) + \xi \int_\xi^\infty p^\prime(\zeta) \big ( 1 - \charf_{\rho^n < \zeta} \big ) d\zeta \right )
 \Big \rangle
$$

We conclude that for $\theta \in C^\infty_c (\R)$ we have
$$
\begin{aligned}
 &\del_t \int_\R \theta (\xi) \big ( 1 - \charf_{\rho^n < \xi} \big ) d\xi  + \div_x \Big ( \int_\R \theta (\xi) u^n \big ( 1 - \charf_{\rho^n < \xi} \big ) d\xi \Big )
\\
&\qquad  + \tfrac{1}{\lambda + 2 \mu} \big \langle \theta(\xi) , \del_\xi \big ( \xi \big ( 1 - \charf_{\rho^n < \xi} \big ) z^n \big ) \Big \rangle
\\
&\qquad - \tfrac{1}{\lambda + 2 \mu}  \Big \langle \theta (\xi) , \del_\xi 
\Big ( \xi p(\xi) \big ( 1 - \charf_{\rho^n < \xi} \big ) + \xi \int_\xi^\infty p^\prime(\zeta) \big (1 - \charf_{\rho^n < \zeta} \big ) d\zeta \Big )
 \Big \rangle
 = 0
\end{aligned} 
$$
in distributions $\dis (Q_T)$. The formula can be explicitly expressed for a test function $\varphi (t,x)  \in C^\infty_c ( [0, T) \times \Omega)$ in
the form
$$
\begin{aligned}
&- \iint \int_\R (\del_t \varphi ) \theta \big ( 1 - \charf_{\rho^n < \xi} \big ) + (\nabla \varphi ) \theta \cdot u^n \big ( 1 - \charf_{\rho^n < \xi} \big )
\\
&- \iint \int_\R  \varphi (\del_\xi \theta )\tfrac{1}{\lambda + 2 \mu} \big ( \xi \big (1 - \charf_{\rho^n < \xi} \big ) z^n \big ) 
\\
&+ \iint \int_\R  \varphi (\del_\xi \theta )\tfrac{1}{\lambda + 2 \mu} \Big ( \xi p(\xi) \big ( 1 - \charf_{\rho^n < \xi} \big ) + \xi \int_\xi^\infty p^\prime(\zeta) 
\big (1 -  \charf_{\rho^n < \zeta}  \big ) d\zeta \Big )
= 0
\end{aligned}
$$

An argument as in the derivation of \eqref{genkineq} shows that $\charf_{\rho^n < \xi} $ satisfies the initial
value problem
\begin{equation}\label{approxeqn}
\begin{aligned}
\del_t \big ( 1 - \charf_{\rho^n < \xi} \big )  + \div \Big ( u^n \big ( 1 - \charf_{\rho^n < \xi} \big ) \Big ) 
+ \tfrac{1}{\lambda + 2 \mu}  \del_\xi \Big ( \xi \big ( 1 - \charf_{\rho^n < \xi} \big ) z^n \Big ) 
\\
- \tfrac{1}{\lambda + 2 \mu} \del_\xi \left ( \xi p(\xi) \big ( 1 - \charf_{\rho^n < \xi} \big ) 
+ \xi \int_\xi^\infty p^\prime(\zeta) \big ( 1 - \charf_{\rho^n < \xi} \big )  d\zeta \right )
&= 0
 \\
 \lim_{t \to 0} \charf_{\rho^n (t,x) < \xi} &= \charf_{\rho_0^n (x) < \xi}
\end{aligned}
\end{equation}
The equation captures the nonlinear structure of the renormalized solution via a duality argument acting on the test function $\theta$.

At the point of this writing the implications of the approximate kinetic equation have not been fully analyzed, whether analytically or as tools for
the computation of propagating oscillations.
Results from conservation laws indicate that such equations lead to better understanding of oscillations, 
like the cancellation of oscillations properties obtained in \cite{b-Perthame02}, \cite{PT00} and Section \ref{sec:kineticcl},
or like the propagation of compactness in \cite[Ch 7]{b-PS12} for density oscillations in compressible Navier-Stokes for a barotropic gas. 
This belief is reinforced by the analysis presented in Section \ref{sec:osci} for homogenization in
one-dimensional viscoelastic models of phase transitions.

\subsection{An alternative approach to obtain an effective kinetic equation}\label{sec:effmethod2}

In this second approach we introduce  functions $b(\rho) = \rho g(\rho)$ and use $g^\prime(\rho) \in C_c^1 (\R)$ as a test function.
The equation for the renormalized approximate solutions \eqref{apprenentr} is expressed in the form
$$
\del_t \big ( \rho^n g(\rho^n) \big ) + \div \big ( u^n \rho^n g(\rho^n) \big ) + (\rho^n )^2 g^\prime (\rho^n) \div u^n = 0
$$
In turn, with the objective to exploit the special properties of the effective viscous flux \eqref{effvf}, this is written in the form
\begin{equation}\label{appren2}
\begin{aligned}
&\del_t \big ( \rho^n g(\rho^n) \big ) + \div \big ( u^n \rho^n g(\rho^n) \big ) 
\\
&\qquad = \tfrac{1}{\lambda + 2 \mu} (\rho^n )^2 g^\prime (\rho^n) z^n  - \tfrac{1}{\lambda + 2 \mu}  (\rho^n )^2 g^\prime (\rho^n) p(\rho^n)
\\
&\qquad = \tfrac{1}{\lambda + 2 \mu}  I_1^n - \tfrac{1}{\lambda + 2 \mu}  I_2^n
\end{aligned}
\end{equation}

We next use the representation formula 
$$
g(\rho) = \int_{-\infty}^\rho  g^\prime (\xi)  d\xi = \int_\R   g^\prime (\xi) \charf_{\xi < \rho}  d\xi \, .
$$ 
We will think of $g^\prime \in C_c^1 (\R)$ as a test function. Note that for $\supp g^\prime \subset [-m, m]$ the associated $g$ satisfy
$g(\rho) = 0$ for $\rho < -m$ and $g(\rho) = const$ for $\rho > m$.

Along sequences of solutions $\{ \rho^n \}$ we have
$$
\rho^n g(\rho^n) = \int_\R  \rho^n  \charf_{ \xi < \rho^n } g^\prime (\xi) d\xi
$$
Define now the function $H(t,x,\xi)$ to be the weak limit of
\begin{equation}\label{defH}
\rho^n \charf_{\rho^n  < \xi} \rightharpoonup H
\end{equation}
If the measure $H$ is absolutely continuous with respect to $\rho dx$ then we can visualize
$H = \rho G$ where $G(t,x,\xi)$ is the Radon-Nikodym derivative.

The goal is to derive an approximate equation for the quantity $\rho^n \charf_{\xi < \rho^n}$ and to use the function 
$H \in L^\infty ( [0,\infty) ; L^\gamma (Q_T) )$ as the quantity describing the effective oscillations.
In order to express \eqref{appren2}, observe that at least formally
$I_1^n$ and $I_2^n$ may be expressed via the representation formulas
$$
\begin{aligned}
I_1^n  = (\rho^n )^2 g^\prime (\rho^n) z^n &= \int_\R \xi g^\prime (\xi) \rho^n \delta (\xi - \rho^n ) z^n \, d\xi
\\
&= \int_\R g^\prime (\xi)  \Big ( - \xi  \del_\xi \big ( \rho^n \charf_{\xi < \rho^n} \big )   z^n \Big ) \, d\xi
\end{aligned}
$$
$$
\begin{aligned}
I_2^n  = (\rho^n )^2 g^\prime (\rho^n) p(\rho^n) &= \int_\R \rho^n \xi g^\prime (\xi) p(\xi)  \del_\xi \big ( - \charf_{\xi < \rho^n } \big )  \, d\xi
\\
&= \int_\R g^\prime (\xi)  \Big ( - \xi  p(\xi) \del_\xi \big ( \rho^n \charf_{\xi < \rho^n} \big )  \Big ) \, d\xi
\end{aligned}
$$

To validate the meaning of the first formula use the representation
$$
\rho g^\prime (\rho) = \int_\R \charf_{\xi < \rho} \big ( \xi g^\prime \big )^\prime d\xi
$$
to interpret $I_1^n$ as a distribution via the formulas
$$
\begin{aligned}
(\rho^n )^2 g^\prime (\rho^n) z^n &= \int_\R \rho^n  \charf_{\xi < \rho^n} z^n ( \xi g^\prime \big )^\prime d\xi
\\
&= \big\langle  \big( \xi g^\prime \big )^\prime , \rho^n  \charf_{\xi < \rho^n} z^n \big\rangle
\\
&= \big\langle  g^\prime (\xi) , -\xi \del_\xi \big ( \rho^n  \charf_{\xi < \rho^n} z^n  \big ) \big\rangle
\end{aligned}
$$
Similarly, for $p(\cdot)$ smooth,
$$
(\rho^n )^2 p(\rho^n) g^\prime (\rho^n) = \big\langle  g^\prime (\xi) , -\xi p(\xi) \del_\xi \big ( \rho^n  \charf_{\xi < \rho^n}  \big ) \big\rangle
$$

Combining all together we obtain
$$
\begin{aligned}
&\del_t \left (  \int_\R  g^\prime \rho^n  \charf_{\xi< \rho^n } d\xi \right ) + \div \left ( \int_\R  g^\prime u^n \rho^n  \charf_{\xi < \rho^n } d\xi \right )
\\
&\quad =  \tfrac{1}{\lambda + 2 \mu}  \Big \langle g^\prime ,  - \xi \del_\xi \big ( \rho^n \charf_{\xi < \rho^n} z^n \big ) 
+ \xi p(\xi) \del_\xi \big ( \rho^n \charf_{\xi < \rho^n}  \big ) 
  \Big \rangle
\end{aligned}
$$
holds in $\dis_{t,x}$ for any $g^\prime \in C_c^\infty(\R)$. We thus conclude that the quantity $\rho^n \charf_{\xi < \rho^n}$ satisfies the
transport equation
\begin{equation}
\label{apprkineq}
\begin{aligned}
&\del_t \Big ( \rho^n  \charf_{\xi < \rho^n}   \Big ) + \div \Big ( u^n \rho^n  \charf_{\xi < \rho^n}  \Big) 
%\\
%&\quad 
+  \tfrac{1}{\lambda + 2 \mu} 
\Big ( \xi \del_\xi \big ( \rho^n \charf_{\xi < \rho^n} z^n \big ) 
- \xi p(\xi) \del_\xi \big ( \rho^n \charf_{\xi < \rho^n} \big ) \Big )
 = 0
\end{aligned} 
\end{equation}

\medskip

To pass to the limit, we need a set of properties regarding weak limits,
namely the property \eqref{defH} that defines the function $H(t,x,\xi)$ and the properties
\begin{align}
u^n \big ( \rho^n  \charf_{\rho^n  < \xi} \big ) \rightharpoonup  u H 
\\
z^n \big ( \rho^n  \charf_{\rho^n  < \xi} \big ) \rightharpoonup  z H  
\end{align}
which are the analogs of \eqref{weaku}, \eqref{weakz} in the present context.
We arrive at the effective kinetic equation
$$
\del_t H + div ( u H ) +  \tfrac{1}{\lambda + 2 \mu} \Big ( \xi \del_\xi (H z) - \xi p(\xi)  \del_\xi  H  \Big ) = 0) 
$$
Finally, taking account that $z = \overline{p(\rho)} - (\lambda + 2 \mu) \div u$ where $\overline{p(\rho)} = \mbox{wk-lim} \, p(\rho^n)$,
the transport equation for $H$ simplifies to
\begin{equation}
\del_t H + div ( u H ) - \tfrac{1}{\lambda + 2 \mu} \xi \Big ( p(\xi) - \overline{p(\rho)}  + (\lambda + 2\mu) \div u \Big )  \del_\xi H = 0
\end{equation}
The last term is well defined in the sense of distributions if $p(\xi)$ is sufficiently smooth. Otherwise it will be interpreted via integrations
by parts.

The momentum equation remains the same in Theorem \ref{thmhomoCNS}.
In summary, the limiting system describing the effective behavior reads
\begin{equation}\label{homoS2}
\begin{aligned}
\del_t H + div ( u H ) &= \tfrac{1}{\lambda + 2 \mu}  \xi  (\del_\xi H ) \Big ( p(\xi) - \overline{p(\rho)} + (\lambda + 2 \mu) \div u 
\Big ) 
\\
\del_t (\rho u) + \div (\rho u \otimes u ) + \nabla \; \overline{p(\rho)} &= \mu \triangle u + (\lambda + \mu) \nabla (\div u)
\\
H(0, x, \xi) &= \mbox{wk-lim} \Big ( \rho^n_0 \charf_{\rho^n_0 (x) < \xi} \Big )
\\
u_0 (x) &= \mbox{s-lim} ( u^n_0 (x) )
\end{aligned}
\end{equation}

%\vfil\eject
%\bigskip
\bigskip

{\bf Acknowledgement}.
These Lecture Notes were partly prepared in the Fall 2025 semester while the author was in
residence at the Simons Laufer Mathematical Sciences Institute in Berkeley,
supported in part by the National Science Foundation under Grant No. DMS-2424139
and in part by KAUST baseline funds, No BAS/1/1652-01-01.
I would like to thank the Institute for their hospitality and  Professors 
{\sc Fran\c{c}ois Golse} and {\sc Pierre-Emmanuel Jabin} for helpful discussions.

%\bigskip
%{\bf Statements and Declarations}
%
%{\bf Data availability} 
%The data supporting the findings of this study are available within the paper,
%
%{\bf Competing Interests}
%The author has no relevant competing interests to declare.

%{\bf Funding} 
%Research supported by KAUST baseline funds, No BAS/1/1652-01-01

%\bibliographystyle{plain} % We choose the "plain" reference style 
%\bibliography{biblioosc}

%\bibliographystyle{abbrv} % We choose the "abbrv" reference style = abbreviates initials and journals
%\bibliography{biblioosc}

%
\end{document}